\documentclass[11pt,a4paper,leqno]{amsart}

\usepackage{common}
\usepackage{tikz}
\usepackage{float}
\usepackage{tabularx}
\definecolor{zygblue}{RGB}{47,102,144}
\definecolor{zygorange}{RGB}{190,104,18}
\definecolor{zyggreen}{RGB}{46,125,50}
\definecolor{zygred}{RGB}{166,61,74}
\definecolor{zygviolet}{RGB}{107,76,154}

\title[A characterization of two-depth Journ\'{e} packing]
{A characterization of two-depth Journ\'{e} packing for bi-parameter and Zygmund rectangles}
\author{Henri Martikainen}

\address[H.M.]{Department of Mathematics, Washington University in St. Louis,
1 Brookings Drive, St. Louis, MO 63130, USA}
\email{henri@wustl.edu}

\hypersetup{
	pdftitle={A characterization of two-depth Journe packing for bi-parameter and Zygmund rectangles},
	pdfauthor={Henri Martikainen},
	pdfkeywords={Journe's lemma, Zygmund rectangles, incomparable rectangles, sparse selection,
		probabilistic packing, exponential overlap}
}

\begin{document}

\allowdisplaybreaks

\begin{abstract}
	We completely characterize a two-depth version of Journ\'{e}'s covering lemma in the
	bi-parameter setting. Let $\mathcal U$ be an incomparable family of
	bi-parameter dyadic rectangles contained in $\Omega$, and write
	$e_i(I):=e_i(I;\Omega)$ for the two embeddedness depths. For every coordinatewise
	nonincreasing $w\colon\N^2\to[0,\infty)$, the estimate
	\[
		\sum_{I\in\mathcal U}w(e_1(I),e_2(I))|I|\lesssim_w|\Omega|
	\]
	holds uniformly in $(\mathcal U,\Omega)$ if and only if the following fixed-total
	condition holds:
	\[
		\|w\|_{\mathrm{ft}}
		:=\sup_{N\geq0}\sum_{a=0}^Nw(a,N-a)<\infty.
	\]
	For instance, $\|(a+1)^{-s}(b+1)^{-(1-s)}\|_{\mathrm{ft}}
	\nobreak\sim[s(1-s)]^{-1}$ for $0<s<1$, so this allows weights that are not summable in
	either variable, while one-depth theory holds precisely for summable weights. Exactly the
	same characterization holds for
	dyadic Zygmund rectangles $I=I^1\times I^2\times I^3$ with
	$\ell(I^3)=\ell(I^1)\ell(I^2)$.
	We use probabilistic packing methods to prove the sufficiency, and these techniques also
	yield sparse refinements of the preceding estimates, which self-improve to various forms
	of exponential integrability.
\end{abstract}
\maketitle

\section{Introduction}

Journ\'{e}'s covering lemma replaces the disjointness argument available for maximal dyadic
cubes: incomparable product rectangles can overlap many times.
Incomparability means that no rectangle in the family is contained in another; this removes
the most obvious source of multiplicity but still permits heavy overlap. Journ\'{e}-style lemmas
control this multiplicity by giving less weight to rectangles that lie more deeply inside
an enlargement of their union. This principle is fundamental in product $\BMO$, product
Hardy spaces and multi-parameter singular-integral theory. The original bi-parameter
results are due to Journ\'{e} \cites{Jou1985, Journe1986}; see also Pipher
\cite{Pipher1986}, Cabrelli--Lacey--Molter--Pipher \cite{CLMP2006}, and the modern
non-homogeneous formulation of Hyt\"{o}nen--Martikainen
\cite{Hytonen2014}*{Section 8.1}.

A familiar bi-parameter form of Journ\'{e}'s lemma says that, for a decreasing weight
$w\colon\N\to[0,\infty)$, the one-depth estimate
\begin{equation}\label{eq:intro-one-depth-criterion}
	\sum_{I\in\mathcal U}w(e_i(I;\Omega))|I|
	\lesssim_w|\Omega|
\end{equation}
holds uniformly in $(\mathcal U,\Omega)$ if and only if $w\in\ell^1(\N)$.
For a coordinatewise nonincreasing $w\colon\N^2\to[0,\infty)$, this shows that full
summability on $\N^2$ is much more than is needed: the two-depth estimate already follows
from \eqref{eq:intro-one-depth-criterion} if either boundary trace $k\mapsto w(k,0)$ or
$k\mapsto w(0,k)$ is summable, since $w(e_1,e_2)\leq w(e_1,0)$ and
$w(e_1,e_2)\leq w(0,e_2)$.

Using more than one embeddedness depth is not itself new. The higher-parameter rectangular
formulations of Pipher \cite{Pipher1986} and Cabrelli--Lacey--Molter--Pipher
\cite{CLMP2006} use several depth factors; already in three parameters, incomparability
alone gives no one-depth theorem; see Proposition
\ref{prop:subzygmund-family-failure}. Our question is different: in the bi-parameter
setting, where a summable weight already works with one depth, what additional weights
become available when both depths are used? The Zygmund basis is a natural intermediate
test: its rectangles have three coordinate factors but only two independent scales and,
strikingly, its answer is the same as in the bi-parameter setting. Analogous questions that
use more depth information than the classical summable theory requires arise in higher
parameters and are not settled by the cited formulations.

For a nonnegative function $w$ on $\N^2$, define its \emph{fixed-total norm} by
\begin{equation*}
	\|w\|_{\mathrm{ft}}
	:=\sup_{N\geq0}\sum_{a=0}^Nw(a,N-a).
\end{equation*}
A family with $N+1$ layers, one for each depth pair $(a,N-a)$, shows that any uniform
two-depth estimate forces $\|w\|_{\mathrm{ft}}<\infty$. Our main result is the converse:
for every coordinatewise nonincreasing weight, this elementary necessary condition is
already sufficient.

Thus the simplest layered examples detect the entire obstruction. This is striking because
the fixed-total norm sees only one diagonal $a+b=N$ at a time: it neither sums the weight
over $\N^2$ nor requires summability in either variable. Nevertheless, it controls every
incomparable family, regardless of how complicated its overlap may be.

The fixed-total condition is much weaker than summability on $\N^2$. It includes the
critical powers $w_s(a,b)=(a+1)^{-s}(b+1)^{-(1-s)}$, $0<s<1$, and the logarithmic
weight $w_{\log}(a,b)=[1+\log(a+1)]^{-1}(b+1)^{-1}$, even though neither one-variable
factor is summable in either case.

The same fixed-total characterization indeed holds for Zygmund rectangles. Beyond the
Zygmund boundary, however, uniform two-depth packing for sub-Zygmund or unrestricted
tri-parameter initial rectangles forces full summability on $\N^2$; see Proposition
\ref{prop:subzygmund-family-failure}.

\subsection*{The two settings and the characterization}

Fix a measurable set $\Omega$ of finite measure. For a dyadic rectangle basis $\mathcal B$,
let
\[
	M_{\mathcal B}f(x):=\sup_{R\in\mathcal B:\,x\in R}\frac1{|R|}\int_R|f|,
	\qquad \widetilde\Omega_{\mathcal B}:=\{M_{\mathcal B}1_\Omega>1/2\}.
\]
When the basis is clear, we write simply $\widetilde\Omega$.
Fix dyadic grids $\calD^i$ in $\R^{d_i}$, $i=1,2,3$. In the bi-parameter setting,
$\mathcal B=\calD^1\times\calD^2$. For $I=I^1\times I^2\subset\Omega$, let
$I_i^{(k)}$ replace only $I^i$ by its $k$th dyadic ancestor and define
\[
	e_i(I;\Omega):=\max\{k\geq0:I_i^{(k)}\subset\widetilde\Omega\},
	\qquad i=1,2.
\]

In the second setting, the dyadic Zygmund rectangles are the products
$I=I^1\times I^2\times I^3$ satisfying
\begin{equation*}
	\ell(I^3)=\ell(I^1)\ell(I^2);
\end{equation*}
write $\calD_Z$ for their collection. These rectangles are adapted to the Zygmund
dilations $(x_1,x_2,x_3)\mapsto(\delta_1x_1,\delta_2x_2,\delta_1\delta_2x_3)$ and to
the corresponding singular-integral theory; see, e.g., \cites{FEPI, HLMV}. Enlarging
only $I^1$ or $I^2$ enters the sub-Zygmund basis
\[
	\calD_{\mathrm{sZ}}
	:=\{I:\ell(I^3)\leq\ell(I^1)\ell(I^2)\}.
\]
In this setting, $\widetilde\Omega$ denotes the sub-Zygmund halo
$\widetilde\Omega_{\calD_{\mathrm{sZ}}}$.
Using the same coordinate-enlargement notation $I_i^{(k)}$, define the \emph{flat} depths by
\[
	e_i(I;\Omega):=\max\{k\geq0:I_i^{(k)}\subset\widetilde\Omega\},
	\qquad i=1,2.
\]
The setting will always specify which depths are meant.

In either setting, call $(\Omega,\mathcal U)$ an \emph{admissible pair} if $\Omega$ is
measurable with finite measure and $\mathcal U$ is an incomparable family of rectangles
in that setting, all contained in $\Omega$. For an admissible pair and a weight
$w\colon\N^2\to[0,\infty)$, define the associated rectangle weight $W_w$ and
\emph{weighted height} $H_w$ by
\begin{equation}\label{eq:intro-general-weight-height}
	W_w(I;\Omega):=w(e_1(I;\Omega),e_2(I;\Omega)),
	\qquad
	H_w(x):=\sum_{I\in\mathcal U}W_w(I;\Omega)1_I(x).
\end{equation}
We suppress the dependence of $H_w$ on $\mathcal U$ and $\Omega$ from the notation.
Write $w^\top(a,b):=w(b,a)$ and set
\[
	\begin{aligned}
	\mathcal W_{\mathrm{ft}}
	&:=\{w\colon\N^2\to[0,\infty):\|w\|_{\mathrm{ft}}<\infty\},\\
	\mathcal W_2
	&:=\{w\colon\N^2\to[0,\infty):
		 w(a,b+1)\leq w(a,b)\text{ for all }a,b\geq0\},\\
	\mathcal W_{12}
	&:=\{w\in\mathcal W_2:w^\top\in\mathcal W_2\}.
	\end{aligned}
\]
Thus $\mathcal W_2$ and $\mathcal W_{12}$ record monotonicity in the second variable and
in both variables; also $\|w^\top\|_{\mathrm{ft}}=\|w\|_{\mathrm{ft}}$. No tensor-product
structure is imposed (see Section \ref{sec:concrete-fixed-total-weights} for examples).

For $\mathcal G\subset\mathcal U$, write
$\operatorname{sh}(\mathcal G):=\bigcup_{I\in\mathcal G}I$ for its shadow and define its
unweighted overlap function by
\begin{equation*}
	h_{\mathcal G}(x):=\sum_{I\in\mathcal G}1_I(x).
\end{equation*}
For $0<\eta\leq1$, the family $\mathcal G$ is $\eta$-\emph{sparse} if there are pairwise
disjoint measurable sets $E_I\subset I$, $I\in\mathcal G$, such that
$|E_I|\geq\eta|I|$.

\begin{thm}[Characterization of two-depth Journ\'{e} packing]
\label{thm:intro-planar-admissible}
\label{thm:intro-zygmund-admissible}
Let $w\in\mathcal W_{12}$. In both the bi-parameter and Zygmund settings above, with
$W_w$ and $H_w$ as in \eqref{eq:intro-general-weight-height}, the estimate
\begin{equation}\label{eq:intro-planar-general-packing}
	\int_\Omega H_w(x)\,\ud x
	=\sum_{I\in\mathcal U}W_w(I;\Omega)|I|
	\leq C_w|\Omega|
\end{equation}
holds uniformly over all admissible pairs $(\Omega,\mathcal U)$ if and only if
$w\in\mathcal W_{\mathrm{ft}}$. In either setting, the optimal
constant satisfies
\[
	\|w\|_{\mathrm{ft}}\leq C_w\leq2560\|w\|_{\mathrm{ft}}.
\]
If $w\in\mathcal W_{\mathrm{ft}}$, one can also choose a $1/8$-sparse subfamily
$\mathcal G_0\subset\mathcal U$ such that
\begin{equation}
\label{eq:intro-general-sparse}
	\int_\Omega H_w
	\leq2560\|w\|_{\mathrm{ft}}|\operatorname{sh}(\mathcal G_0)|.
\end{equation}
\end{thm}

We call $w\in\mathcal W_{\mathrm{ft}}\cap\mathcal W_{12}$ an \emph{admissible weight}.
Since $\operatorname{sh}(\mathcal G_0)\subset\Omega$, \eqref{eq:intro-general-sparse}
implies \eqref{eq:intro-planar-general-packing}. The stronger marginal inclusion
probabilities behind this sparse estimate will be crucial below.

\subsection*{Critical and logarithmic consequences}

For a one-depth weight $w(a,b)=w(a)$, the identity
$\|w\|_{\mathrm{ft}}=\|w\|_{\ell^1(\N)}$ recovers
\eqref{eq:intro-one-depth-criterion}. For $\alpha,\beta\geq0$, put
\[
	w_{\alpha,\beta}(a,b):=(a+1)^{-\alpha}(b+1)^{-\beta},
\]
and define $W_{\alpha,\beta}$ and $H_{\alpha,\beta}$ accordingly. Lemma
\ref{lem:power-weight-fixed-total-range} and the main theorem give the following.

\begin{cor}[Power weights and critical packing]
\label{thm:intro-threshold}
\label{thm:intro-planar}
\label{thm:intro-critical}
In both the bi-parameter and Zygmund settings, a uniform packing estimate for
$W_{\alpha,\beta}$ holds if and only if
\[
	\alpha+\beta\geq1,
	\qquad
	(\alpha,\beta)\notin\{(1,0),(0,1)\}.
\]
In particular, for $0<s<1$, define
\begin{equation*}
	w_s(a,b):=w_{s,1-s}(a,b)
	=(a+1)^{-s}(b+1)^{-(1-s)}
\end{equation*}
and write $W_s:=W_{w_s}$ and $H_s:=H_{w_s}$. In either setting one can choose a
$1/8$-sparse subfamily $\mathcal G_0\subset\mathcal U$ such that
\begin{equation*}
	\int_\Omega H_s(x)\,\ud x
	=\sum_{I\in\mathcal U}W_s(I;\Omega)|I|
	\leq\frac{2560}{s(1-s)}|\operatorname{sh}(\mathcal G_0)|.
\end{equation*}
\end{cor}

Neither factor of $w_s$ is summable, and \eqref{eq:critical-power-fixed-total} shows that
the order $1/[s(1-s)]$ is optimal. The logarithmic weight
$w_{\log}(a,b):=[1+\log(a+1)]^{-1}(b+1)^{-1}$ satisfies
$\|w_{\log}\|_{\mathrm{ft}}\leq2$; Section
\ref{sec:concrete-fixed-total-weights} proves these bounds.

\subsection*{How the two-depth packing is proved}\label{subsec:intro-critical-proof}

Section \ref{sec:probability} develops an abstract probabilistic toolkit for producing
weighted packing estimates and sparse subfamilies. The two-depth Journ\'{e} argument is
organized around the alteration result in Lemma
\ref{lem:local-overlap-to-marginals}. Once intersecting ordered pairs have been divided
into near and far pairs, embeddedness supplies the required major subsets for the far
predecessors, while the lemma reduces the remaining task to the localized packing estimate
\eqref{eq:abstract-local-near-overlap} for the later near successors of each fixed
rectangle. The alteration argument itself is relatively short, but its conclusion is
precisely what is needed here: it produces random sparse subfamilies together with a
quantitative marginal inclusion probability for every rectangle.

The crux of the two-depth Journ\'{e} proofs is therefore to verify
\eqref{eq:abstract-local-near-overlap} in the concrete geometry. Much of this verification
is organized by Lemma \ref{lem:exclusion-packing}, an abstract packing result whose
geometric assumptions mirror the disjointness and scale separation available in Journ\'{e}
configurations. Its coefficients are probabilities of an exclusion family. Informally,
the event $\mathsf S_{a,b}$ selects layer $a$ while excluding the next $b$ layers; formally,
\[
	\mathsf S_{a,b}\cap\mathsf S_{a',b'}=\varnothing
	\qquad\text{whenever }a<a'\leq a+b,
\]
independently of $b'$. The remaining probabilistic task is therefore to construct an
appropriate exclusion family for each admissible weight. Lemmas
\ref{lem:marked-weight-realization} and \ref{lem:fixed-total-completion} accomplish exactly
this: after a harmless normalization, every admissible weight is dominated by the
probabilities of such a family on a suitable probability space. Lemma
\ref{lem:exclusion-packing} can therefore be applied uniformly to all admissible weights,
furnishing the local estimate required by Lemma
\ref{lem:local-overlap-to-marginals}.

\subsection*{Exponential integrability}

For $w\in\mathcal W_{12}$, set
\[
	w^\Delta(N):=\max_{0\leq a\leq N}w(a,N-a),
	\qquad
	\|w\|_{\mathrm{env}}:=\|w^\Delta\|_{\ell^1(\N)}.
\]
Since $\|w\|_{\mathrm{ft}}\leq\|w\|_{\mathrm{env}}\leq\|w\|_{\ell^1(\N^2)}$, envelope
summability lies strictly between fixed-total admissibility and full summability: $w_s$
fails it, while $(a+b+1)^{-2}$ satisfies it but is not in $\ell^1(\N^2)$. It nevertheless
forces every row and column of $w$ to be summable, since $w(a,b)\leq w^\Delta(a+b)$.
Its role below is therefore not to enlarge the $L^1$ range, but to recover full exponential
integrability in the Zygmund setting.

The exponential estimates are a separate refinement of the $L^1$ classification. Here
\emph{full}, \emph{square-root} and \emph{cube-root exponential} mean bounds for
$e^{cF}$, $e^{c\sqrt F}$ and $e^{cF^{1/3}}$. The upper block of the table concerns the
unweighted overlap $h_{\mathcal G}$, so $F=h_{\mathcal G}$, and separates sparseness from
sparseness together with incomparability. The sparse entries follow from maximal-function
estimates. A beautiful recent theorem of Rey \cite{Rey2026antichain} proves the
bi-parameter incomparable entry and conjectures the tri-parameter one. For sparse
incomparable Zygmund families, no improvement over the square-root estimate supplied by
sparseness is presently known. The lower block records the weighted results proved here;
in its two rows, $F=H_w/\|w\|_{\mathrm{ft}}$ and
$F=H_w/\|w\|_{\mathrm{env}}$, respectively.

\begin{center}
	\small
	\renewcommand{\arraystretch}{1.25}
	\begin{tabularx}{\textwidth}{@{}l>{\raggedright\arraybackslash}X
		>{\raggedright\arraybackslash}X@{}}
		\hline
		\multicolumn{3}{@{}l}{\emph{Unweighted overlap $h_{\mathcal G}$}} \\[2pt]
		Setting & Sparse $\mathcal G$ (known) & Sparse and incomparable $\mathcal G$ \\
		\hline
		Bi-parameter
		& square-root exponential
		& full exponential (Rey's theorem) \\
	Zygmund
		& square-root exponential
		& square-root exponential; improvement open \\
		Tri-parameter
		& cube-root exponential
		& square-root exponential (Rey's conjecture) \\[3pt]
		\hline
		\multicolumn{3}{@{}l}{\emph{Weighted Journ\'{e} heights proved here}} \\[2pt]
		Weight & Bi-parameter & Zygmund \\
		\hline
		Every admissible weight & full exponential & square-root exponential \\
		$w^\Delta\in\ell^1(\N)$ & full exponential & full exponential \\
		\hline
	\end{tabularx}
\end{center}

The passage from the upper to the lower block is not automatic, even in the bi-parameter
setting. The Journ\'{e} argument must first be organized so as to produce sparse
subfamilies. Moreover, one sparse family controlling only the total weighted mass is
insufficient: the convexity transfer from $H_w$ to $h_{\mathcal G}$ requires the stronger
random construction developed here, with a quantitative marginal inclusion probability
for every rectangle. With this machinery, every admissible weight in the bi-parameter
setting reaches the full exponential scale made available by Rey's incomparable-overlap
theorem. General admissible weights reach the universal square-root scale in the Zygmund
setting, while the stronger envelope condition recovers full exponential integrability
there. Thus the square-root conclusion for general weights is not the end of the Zygmund
story: additional depth decay already recovers the full exponential scale, even though the
corresponding unweighted overlap problem for incomparable families remains open.

In the Zygmund comparison below, $\cosh(c\sqrt t)$ is the convex substitute for
$e^{c\sqrt t}$.

\begin{thm}[Sparse exponential comparisons]\label{thm:intro-exponential-comparisons}
	There exist absolute constants $c,C>0$ such that the following holds. Let $w$ be a
	nonzero admissible weight. For every admissible pair $(\Omega,\mathcal U)$ in the
	bi-parameter setting, one can choose a
	$1/8$-sparse $\mathcal G_0\subset\mathcal U$ such that
	\begin{equation*}
		\int_{\R^{d_1+d_2}}\bigl(e^{c \|w\|_{\mathrm{ft}}^{-1}H_w(x)}-1\bigr)\,\ud x
		\leq\int_{\R^{d_1+d_2}}\bigl(e^{C h_{\mathcal G_0}(x)}-1\bigr)\,\ud x
		\lesssim|\operatorname{sh}(\mathcal G_0)|.
	\end{equation*}
	For every admissible pair $(\Omega,\mathcal U)$ in the Zygmund setting, one can choose a
	$1/8$-sparse
	$\mathcal G_0\subset\mathcal U$ such that
	\begin{equation*}
		\begin{aligned}
		&\int_{\R^{d_1+d_2+d_3}}
		\left[\cosh\!\bigl(c\sqrt{\|w\|_{\mathrm{ft}}^{-1}H_w(x)}\bigr)-1\right]\,\ud x\\
		&\quad\leq\int_{\R^{d_1+d_2+d_3}}
		\left[\cosh\!\bigl(C\sqrt{h_{\mathcal G_0}(x)}\bigr)-1\right]\,\ud x
		\lesssim|\operatorname{sh}(\mathcal G_0)|.
		\end{aligned}
	\end{equation*}
\end{thm}

Theorems \ref{thm:planar-admissible-exponential} and
\ref{thm:zygmund-admissible-exponential} give quantitative forms. Theorem
\ref{thm:diagonal-envelope-exponential} gives full exponential integrability when
$\|w\|_{\mathrm{env}}<\infty$.

\subsection*{Notation and conventions}

We use $\N=\{0,1,2,\ldots\}$ and write $\ave{f}_Q:=|Q|^{-1}\int_Qf$ for averages.
As usual, $A\lesssim_\tau B$ means $A\leq C_\tau B$ for a constant depending only on
$\tau$ and the fixed dimensions; $A\gtrsim_\tau B$ means $B\lesssim_\tau A$, and
$A\sim_\tau B$ means that both estimates hold.

\subsection*{Acknowledgements}

This material is based upon work supported by the National Science Foundation under Grant
No. 2247234. H.M. was also supported by the Simons Foundation through MP-TSM-00002361
(travel support for mathematicians). The author thanks Kangwei Li for helpful discussions
and comments on the manuscript.
The author used OpenAI's Codex as a research and editorial aid in developing this manuscript---
for example, to explore formulations, stress-test proofs and examples, and improve the
organization and exposition.

\section{Probabilistic tools for weighted packing}\label{sec:probability}

Subsection \ref{sec:exclusion-events} isolates the abstract exclusion mechanism. A family of
events satisfying one simple exclusion property selects sparse subfamilies whenever two
fixed geometric conditions hold. The geometry remains unchanged, while the probabilities
of the events determine the coefficients in the resulting packing estimate.

Subsections \ref{sec:random-marked-packing} and \ref{sec:marked-weight-realization}
construct the exclusion events needed for two-depth packing from random marked
nonnegative integers. Subsection \ref{sec:fixed-total-criterion} shows that every
$w\in\mathcal W_{\mathrm{ft}}\cap\mathcal W_2$ can be realized in this way after
normalization and completion. This gives the exact fixed-total packing criterion. The two
Journ\'{e} configurations apply it to $w$ and $w^\top$, which explains the assumption
$w\in\mathcal W_{12}$ in the covering theorems.

Subsection \ref{sec:overlap-to-sparse} contains the sparse-selection mechanism and
organizes the later geometric proofs. Lemma \ref{lem:local-overlap-to-marginals} asks for
major subsets avoiding far predecessors in condition~(i) and a local weighted packing of
the later near successors in condition~(ii). In the concrete Journ\'{e} arguments,
embeddedness verifies condition~(i), while the fixed-total packing estimates verify
condition~(ii). The lemma produces random sparse subfamilies with an individual inclusion
probability for every original rectangle. Subsection \ref{sec:convex-transfer} isolates the
independent convex transfer principle that converts these probabilities into weighted
exponential estimates. Finally, Subsection
\ref{sec:stationary-one-depth-selection} constructs exclusion events on the integers
from a summable weight for the diagonal-envelope and one-depth results. All of these
constructions use only elementary probability and measure theory.

\subsection{Exclusion events and sparse selection}
\label{sec:exclusion-events}

Let $\Gamma\subset\Z$, and let $(\Sigma,\mathcal F,\mathbb P)$ be a probability space.
Suppose that measurable events
\[
	\mathsf S_{a,b}\in\mathcal F,
	\qquad a\in\Gamma,\quad b\geq0,
\]
are given with the following property: for any $a,a'\in\Gamma$ and $b\geq0$ with
$a<a'\leq a+b$, we have
\begin{equation}\label{eq:abstract-exclusion-property}
	\mathsf S_{a,b}\cap\mathsf S_{a',b'}=\varnothing
	\qquad\text{for every integer }b'\geq0.
\end{equation}
We call such a family an \emph{exclusion family on $\Gamma$}, and write
\[
	\pi(a,b):=\mathbb P(\mathsf S_{a,b}).
\]
The following lemma is the complete interface between these events and the geometric
packing argument. In the bi-parameter applications the sets $P_I$ will be coordinate
cubes; in the Zygmund applications they will be lower-dimensional dyadic rectangles.

\begin{lem}[Packing from exclusion events]\label{lem:exclusion-packing}
	Let $X$ be a measurable subset of a Euclidean space with $|X|<\infty$, and let
	$\mathcal A$ be a countable index set. For each $I\in\mathcal A$, let $P_I\subset X$
	be a measurable set of positive measure, let $a(I)\in\Z$, and let $b(I)\geq0$ be an
	integer. Suppose that the following hold:
	\begin{enumerate}[(i)]
		\item for every fixed $m\in\Z$, the sets $P_I$ with $a(I)=m$ are pairwise
		disjoint;
		\item for every $I\in\mathcal A$,
		\begin{equation}\label{eq:exclusion-overlap-bound}
			\Biggl|P_I\cap
			\bigcup_{\substack{K\in\mathcal A\\a(K)>a(I)+b(I)}}P_K\Biggr|
			\leq\frac12|P_I|.
		\end{equation}
	\end{enumerate}
	Let $(\Sigma,\mathcal F,\mathbb P)$ be an arbitrary probability space equipped with
	an exclusion family $(\mathsf S_{a,b})$ on $\{a(I):I\in\mathcal A\}$.
	Then one can associate to each $\omega\in\Sigma$ a subfamily
	$\mathcal S(\omega)\subset\mathcal A$ in such a way that every event
	$\{\omega:I\in\mathcal S(\omega)\}$ is measurable,
	\begin{equation}\label{eq:exclusion-marginals}
		\mathbb P\bigl(\{\omega:I\in\mathcal S(\omega)\}\bigr)
		=\pi(a(I),b(I)),
		\qquad I\in\mathcal A,
	\end{equation}
	and the indexed family $\{P_I:I\in\mathcal S(\omega)\}$ is $1/2$-sparse for every
	$\omega\in\Sigma$. In particular, for some $\omega_0\in\Sigma$,
	\begin{equation}\label{eq:exclusion-packing}
		\sum_{I\in\mathcal A}\pi(a(I),b(I))|P_I|
		\leq\sum_{I\in\mathcal S(\omega_0)}|P_I|
		\leq2|X|.
	\end{equation}
\end{lem}

The first hypothesis says that the indexed sets in each fixed $a$-layer are disjoint.
The integer $b(I)$ marks the end of the uncontrolled range for $P_I$: the second
hypothesis says that all sets beyond that range together cover at most half of $P_I$.
These are the only geometric facts used by the probabilistic selection.

\begin{proof}
	For $I\in\mathcal A$, put
	\[
		\mathsf S_I:=\mathsf S_{a(I),b(I)},
		\qquad
		\rho_I:=\mathbb P(\mathsf S_I).
	\]
	By definition, $\rho_I=\pi(a(I),b(I))$.

	Fix $\omega\in\Sigma$ and define the now deterministic set of selected indices
	\[
		\mathcal S(\omega):=\{I\in\mathcal A:\omega\in\mathsf S_I\}.
	\]
	Thus, for every $I\in\mathcal A$, we have
	$\{\omega:I\in\mathcal S(\omega)\}=\mathsf S_I$, and so
	\[
		\mathbb P\bigl(\{\omega:I\in\mathcal S(\omega)\}\bigr)
		=\mathbb P(\mathsf S_I)=\rho_I=\pi(a(I),b(I)).
	\]
	This proves \eqref{eq:exclusion-marginals}. Suppose that
	$I,K\in\mathcal S(\omega)$ and $a(I)<a(K)$. If
	$a(K)\leq a(I)+b(I)$, then \eqref{eq:abstract-exclusion-property} would make
	$\mathsf S_I$ and $\mathsf S_K$ disjoint, contradicting
	$\omega\in\mathsf S_I\cap\mathsf S_K$. Therefore
	\[
		a(K)>a(I)+b(I).
	\]
	For each $I\in\mathcal S(\omega)$, remove from $P_I$ all selected sets with larger
	$a$-index and denote the resulting measurable set by
	\[
		E_I(\omega):=
		P_I\setminus
		\bigcup_{\substack{K\in\mathcal S(\omega)\\a(K)>a(I)}}P_K.
	\]
	For this fixed $\omega$, the union is countable, so $E_I(\omega)$ is a measurable
	subset of $X$. Every $K$ in this union satisfies
	$a(K)>a(I)+b(I)$, so \eqref{eq:exclusion-overlap-bound} gives
	$|E_I(\omega)|\geq|P_I|/2$.

	We check that the sets $E_I(\omega)$, indexed by $I\in\mathcal S(\omega)$, are
	pairwise disjoint. Let $I,K$ be distinct members of $\mathcal S(\omega)$. If
	$a(I)=a(K)$, hypothesis (i) gives $P_I\cap P_K=\varnothing$, and hence
	$E_I(\omega)\cap E_K(\omega)=\varnothing$. If, say, $a(I)<a(K)$, the union defining
	$E_I(\omega)$ contains $P_K$, so $E_I(\omega)\cap P_K=\varnothing$ and again
	$E_I(\omega)\cap E_K(\omega)=\varnothing$. Together with
	$|E_I(\omega)|\geq|P_I|/2$, this proves that
	$\{P_I:I\in\mathcal S(\omega)\}$ is $1/2$-sparse. In particular,
	\[
		M(\omega):=\sum_{I\in\mathcal S(\omega)}|P_I|
		\leq2\sum_{I\in\mathcal S(\omega)}|E_I(\omega)|
		\leq2|X|.
	\]

	The function $M$ is measurable as a countable sum of measurable nonnegative functions.
	Since the series is nonnegative, we may integrate it term by term:
	\begin{align*}
		\sum_{I\in\mathcal A}\pi(a(I),b(I))|P_I|
		&=\sum_{I\in\mathcal A}\rho_I|P_I|
		=\int_\Sigma\sum_{I\in\mathcal A}1_{\mathsf S_I}(\omega)|P_I|\,
		\ud\mathbb P(\omega)\\
		&=\int_\Sigma M(\omega)\,\ud\mathbb P(\omega)=\mathbb E M.
	\end{align*}
	Since $M$ is bounded, some $\omega_0\in\Sigma$ satisfies
	$M(\omega_0)\geq\mathbb E M$. Hence
	\[
		\sum_{I\in\mathcal A}\pi(a(I),b(I))|P_I|
		=\mathbb E M\leq M(\omega_0)\leq2|X|,
	\]
	which is \eqref{eq:exclusion-packing}.
\end{proof}

\subsection{Random marked integers}
\label{sec:random-marked-packing}

Let $(\Sigma,\mathcal F,\mathbb P)$ be a probability space. Suppose that each outcome
$\omega\in\Sigma$ is assigned a set $\mathcal T(\omega)\subset\N$ such that
$0\in\mathcal T(\omega)$. For $a\geq0$, define the marking event
\[
	\mathsf M_a:=\{\omega\in\Sigma:a\in\mathcal T(\omega)\}.
\]
We assume that $\mathsf M_a\in\mathcal F$ for every $a\geq0$ and call $\mathcal T$ a
random set of marked integers. Thus $\mathsf M_0=\Sigma$. The set $\mathcal T(\omega)$
need not be infinite. For integers $a,b\geq0$, define the event
\begin{equation*}
	\mathsf E_{a,b}
	:=\mathsf M_a\cap\bigcap_{j=1}^b\mathsf M_{a+j}^{\,c}.
\end{equation*}
Here and below an intersection over an empty index set is $\Sigma$, so
$\mathsf E_{a,0}=\mathsf M_a$. Thus $\mathsf E_{a,b}$ is the event that $a$ is marked and
the following $b$ integers $a+1,\ldots,a+b$ are unmarked. In particular,
$\mathsf E_{a,b}\in\mathcal F$. Set
\[
	\pi(a,b):=\mathbb P(\mathsf E_{a,b}),
	\qquad a,b\geq0.
\]

For any nonnegative integers $a,b,a'$ with $a<a'\leq a+b$, we have
\begin{equation}\label{eq:marked-event-exclusion}
	\mathsf E_{a,b}\cap\mathsf E_{a',b'}=\varnothing
	\qquad\text{for every integer }b'\geq0.
\end{equation}
The disjointness property follows immediately from this marked formulation. Fix
$a<a'\leq a+b$ and $\omega\in\mathsf E_{a,b}$. Then $a'$ is unmarked, whereas
membership in $\mathsf E_{a',b'}$ would require $a'$ to be marked, regardless of
$b'$. Hence $\omega\notin\mathsf E_{a',b'}$ for every $b'\geq0$, which proves
\eqref{eq:marked-event-exclusion}.

Thus the events $(\mathsf E_{a,b})_{a,b\geq0}$ form an exclusion family on $\N$, and the
abstract packing lemma immediately gives the following consequence.

\begin{cor}[Packing from random marked integers]\label{lem:random-marked-packing}
	Let $X$ be a measurable subset of a Euclidean space with $|X|<\infty$, and let
	$(P_I)_{I\in\mathcal A}$, with associated integers $a(I),b(I)\in\N$, satisfy the
	geometric hypotheses of Lemma
	\ref{lem:exclusion-packing}. Then the random marked set above produces random
	$1/2$-sparse indexed families $\mathcal S_{\mathcal T}(\omega)$ with
	\[
		\mathbb P\bigl(\{\omega:I\in\mathcal S_{\mathcal T}(\omega)\}\bigr)
		=\pi(a(I),b(I)),
		\qquad I\in\mathcal A.
	\]
	In particular, for some $\omega_0\in\Sigma$,
	\[
		\sum_{I\in\mathcal A}\pi(a(I),b(I))|P_I|
		\leq\sum_{I\in\mathcal S_{\mathcal T}(\omega_0)}|P_I|
		\leq2|X|.
	\]
\end{cor}

\begin{proof}
	Restrict the exclusion family $(\mathsf E_{a,b})_{a,b\geq0}$ to
	$\{a(I):I\in\mathcal A\}$. The exclusion property
	\eqref{eq:abstract-exclusion-property} is precisely
	\eqref{eq:marked-event-exclusion}. Lemma \ref{lem:exclusion-packing} applies and gives
	the asserted family, after denoting it by $\mathcal S_{\mathcal T}(\omega)$.
\end{proof}

Random marked integers have, in addition, a useful last-mark description. For $N\geq0$,
define
\[
	L_N^{\mathcal T}(\omega)
	:=\max\bigl(\mathcal T(\omega)\cap\{0,\ldots,N\}\bigr).
\]
The maximum exists because zero is always marked, and, for $0\leq a\leq N$,
\begin{equation}\label{eq:last-mark-event}
	\{\omega:L_N^{\mathcal T}(\omega)=a\}=\mathsf E_{a,N-a}.
\end{equation}
In particular, $L_N^{\mathcal T}$ is measurable. Moreover,
\[
	L_0^{\mathcal T}=0,
	\qquad
	L_N^{\mathcal T}(\omega)\in\{L_{N-1}^{\mathcal T}(\omega),N\}
	\quad(\omega\in\Sigma,\ N\geq1).
\]
The marked set is recovered from this process by
$\mathcal T(\omega)=\{N\geq0:L_N^{\mathcal T}(\omega)=N\}$.
Conversely, a sequence of measurable random variables
$L_N\colon\Sigma\to\{0,\ldots,N\}$ satisfying $L_0=0$ and
$L_N(\omega)\in\{L_{N-1}(\omega),N\}$ for every $\omega\in\Sigma$ and $N\geq1$
determines a random marked set by
$\mathcal T(\omega):=\{N\geq0:L_N(\omega)=N\}$. Induction on $N$ then gives
$L_N=L_N^{\mathcal T}$. Lemma \ref{lem:marked-weight-realization} will use this converse:
we construct a suitable process $(L_N)$ first and obtain the marked set from it.

The last-mark description also gives a fixed-total property that will be used separately
below. For every fixed $N\geq0$, \eqref{eq:last-mark-event} gives
\begin{equation}\label{eq:fixed-total-partition}
	\mathsf E_{a,N-a}=\{\omega:L_N^{\mathcal T}(\omega)=a\},
	\qquad 0\leq a\leq N,
\end{equation}
and these events partition $\Sigma$ because $L_N^{\mathcal T}$ takes exactly one value in
$\{0,\ldots,N\}$.
Taking $\mathbb P$-measure in this partition gives
\begin{equation}\label{eq:fixed-total-probability}
	\sum_{a=0}^N\pi(a,N-a)=1,
	\qquad N\geq0.
\end{equation}
This additional identity is the constraint characterized in the next subsection.

\subsection{Realization and completion of marked-event probabilities}
\label{sec:marked-weight-realization}

For a random marked set, put $w(a,b):=\mathbb P(\mathsf E_{a,b})$. Such a $w$ satisfies
$0\leq w\leq1$, belongs to $\mathcal W_2$, and, by \eqref{eq:fixed-total-probability}, has
total mass one on every line $a+b=N$. The next lemma shows that these necessary
conditions are also sufficient.

\begin{lem}[Exact realization of marked-event probabilities]
\label{lem:marked-weight-realization}
Let $w\colon\N^2\to[0,1]$. There exist a probability space
$(\Sigma,\mathcal F,\mathbb P)$ and a random marked set
$\mathcal T(\omega)\subset\N$ such that
\begin{equation}\label{eq:marked-weight-realization}
	\mathbb P(\mathsf E_{a,b})=w(a,b),
	\qquad a,b\geq0,
\end{equation}
if and only if $w\in\mathcal W_2$ and
\begin{equation}\label{eq:marked-weight-conditions}
	\sum_{a=0}^Nw(a,N-a)=1,
	\qquad N\geq0.
\end{equation}
\end{lem}

\begin{proof}
Necessity follows from
$\mathsf E_{a,b+1}\subset\mathsf E_{a,b}$ and the fixed-total partition
\eqref{eq:fixed-total-partition}. For sufficiency, suppose that
$w\in\mathcal W_2$ and assume
\eqref{eq:marked-weight-conditions}. Let $\mathcal L$ and $\lambda$ denote the Lebesgue
$\sigma$-algebra and Lebesgue measure on $[0,1]$, respectively, and take
\[
	(\Sigma,\mathcal F,\mathbb P)
	:=\left(\prod_{N=1}^\infty[0,1],
		\bigotimes_{N=1}^\infty\mathcal L,
		\bigotimes_{N=1}^\infty\lambda\right).
\]
Thus every $\omega\in\Sigma$ is a sequence
$\omega=(\omega_1,\omega_2,\ldots)$ with $\omega_N\in[0,1]$. Define the coordinate
functions by
\[
	U_N(\omega):=\omega_N,
	\qquad N\geq1.
\]
By the product measure construction, $U_1,U_2,\ldots$ are independent and uniformly
distributed on $[0,1]$.

We now construct integer-valued random variables
$L_N\colon\Sigma\to\{0,\ldots,N\}$ and prove simultaneously by induction that $L_N$
depends only on $U_1,\ldots,U_N$ and
\begin{equation}\label{eq:last-mark-distribution}
	\mathbb P(L_N=a)=w(a,N-a),
	\qquad 0\leq a\leq N.
\end{equation}
Put $L_0(\omega):=0$ for every $\omega\in\Sigma$. The fixed-total identity at $N=0$
gives $w(0,0)=1$, proving the base case.

Fix $N\geq1$ and suppose that $L_0,\ldots,L_{N-1}$ have been constructed with these
properties. At stage $N$, conditional on $L_{N-1}=a<N$, the new variable $L_N$ will be
either $a$ or $N$. For $0\leq a<N$, set
\[
	r_{N,a}:=
	\begin{cases}
	\displaystyle\frac{w(a,N-a)}{w(a,N-a-1)},
		&w(a,N-a-1)>0,\\[6pt]
	0,&w(a,N-a-1)=0.
	\end{cases}
\]
Since $w\in\mathcal W_2$, we have $0\leq r_{N,a}\leq1$.

Since $L_{N-1}$ is measurable and takes only finitely many values,
\[
	A_N:=\{U_N\leq r_{N,L_{N-1}}\}
	=\bigcup_{a=0}^{N-1}
	  \bigl(\{L_{N-1}=a\}\cap\{U_N\leq r_{N,a}\}\bigr)
	\in\mathcal F.
\]
Define
\[
	L_N(\omega)
	:=L_{N-1}(\omega)1_{A_N}(\omega)
	  +N1_{\Sigma\setminus A_N}(\omega).
\]
Thus $L_N$ is measurable and depends only on $U_1,\ldots,U_N$. For $a<N$,
\[
	\{L_N=a\}=\{L_{N-1}=a\}\cap\{U_N\leq r_{N,a}\}.
\]
The first event depends only on $U_1,\ldots,U_{N-1}$ and is therefore independent of the
second. Hence the induction hypothesis and the definition of $r_{N,a}$ give
\[
	\mathbb P(L_N=a)
	=\mathbb P(L_{N-1}=a)r_{N,a}
	=w(a,N-a).
\]
If $w(a,N-a-1)=0$, then $w(a,N-a)=0$ by monotonicity, so the last equality remains
valid.
Since $L_N\in\{0,\ldots,N\}$, the remaining probability is
\[
	\mathbb P(L_N=N)
	=1-\sum_{a=0}^{N-1}w(a,N-a)
	=w(N,0),
\]
by the fixed-total identity at $N$. This completes the inductive construction of the random
variables $L_N$.

From these random variables define
\[
	\mathcal T(\omega):=\{N\geq0:L_N(\omega)=N\}.
\]
The construction gives $L_0=0$ and
$L_N(\omega)\in\{L_{N-1}(\omega),N\}$ for every $N\geq1$ and $\omega\in\Sigma$.
The converse part of the last-mark description following Corollary
\ref{lem:random-marked-packing} therefore applies: $\mathcal T$ is a random marked set and
$L_N=L_N^{\mathcal T}$. Consequently,
\eqref{eq:last-mark-event} and \eqref{eq:last-mark-distribution} give
\[
	\mathbb P(\mathsf E_{a,b})
	=\mathbb P(L_{a+b}=a)
	=w(a,b),
\]
which is \eqref{eq:marked-weight-realization}.
\end{proof}

The exact realization above requires unit mass on every fixed-total line. The following
purely algebraic lemma converts any nonzero
$w\in\mathcal W_{\mathrm{ft}}\cap\mathcal W_2$ into such a weight without decreasing
any of its normalized values.

\begin{lem}[Fixed-total completion]\label{lem:fixed-total-completion}
Let $w\in\mathcal W_{\mathrm{ft}}\cap\mathcal W_2$ be nonzero and put
$q:=\|w\|_{\mathrm{ft}}^{-1}w$. Define
\begin{equation}\label{eq:fixed-total-completion}
	\pi(a,b):=q(a,b)\quad(b\geq1),
	\qquad
	\pi(a,0):=1-\sum_{j=0}^{a-1}q(j,a-j),
\end{equation}
where the sum is empty when $a=0$. Then $\pi\colon\N^2\to[0,1]$ belongs to
$\mathcal W_2$, satisfies $\pi\geq q$, and has unit mass on every fixed-total line:
\begin{equation}\label{eq:completed-fixed-total-mass}
	\sum_{a=0}^N\pi(a,N-a)=1,
	\qquad N\geq0.
\end{equation}
\end{lem}

\begin{proof}
The definition of the fixed-total norm gives
\[
	q(a,0)+\sum_{j=0}^{a-1}q(j,a-j)\leq1
	\quad\Longleftrightarrow\quad
	\pi(a,0)=1-\sum_{j=0}^{a-1}q(j,a-j)\geq q(a,0).
\]
Also $0\leq q(a,b)\leq1$, because $q(a,b)$ is one term in the fixed-total sum with
$N=a+b$. As $w\in\mathcal W_2$, also $q\in\mathcal W_2$, and hence
\[
	\pi(a,0)\geq q(a,0)\geq q(a,1)=\pi(a,1),
	\quad
	\pi(a,b)=q(a,b)\geq q(a,b+1)=\pi(a,b+1)\quad(b\geq1).
\]
Thus $\pi\colon\N^2\to[0,1]$, $\pi\geq q$, and $\pi\in\mathcal W_2$. Finally,
\eqref{eq:fixed-total-completion} gives
\[
	\sum_{a=0}^N\pi(a,N-a)
	=\sum_{a=0}^{N-1}q(a,N-a)
	 +1-\sum_{a=0}^{N-1}q(a,N-a)
	=1,
\]
which is \eqref{eq:completed-fixed-total-mass}.
\end{proof}

\subsection{The fixed-total packing criterion}
\label{sec:fixed-total-criterion}

Combining Lemmas \ref{lem:marked-weight-realization} and
\ref{lem:fixed-total-completion} with Corollary \ref{lem:random-marked-packing} gives the
following sharp criterion.

\begin{lem}[The fixed-total packing criterion]
\label{lem:fixed-total-packing}
Let $X$ be a measurable subset of a Euclidean space with $0<|X|<\infty$. Let
$(P_I)_{I\in\mathcal A}$, with associated integers $a(I),b(I)\in\N$, satisfy the
geometric hypotheses of Lemma \ref{lem:exclusion-packing}. Let
$w\in\mathcal W_{\mathrm{ft}}\cap\mathcal W_2$ be nonzero. Then there exist a
probability space
$(\Sigma,\mathcal F,\mathbb P)$ and a random subfamily
$\mathcal S_w(\omega)\subset\mathcal A$ such that
\begin{equation}\label{eq:fixed-total-marginals}
	\mathbb P\bigl(\{\omega:I\in\mathcal S_w(\omega)\}\bigr)
	\geq\frac{w(a(I),b(I))}{\|w\|_{\mathrm{ft}}},
	\qquad I\in\mathcal A,
\end{equation}
and the indexed family $\{P_I:I\in\mathcal S_w(\omega)\}$ is $1/2$-sparse for every
$\omega$. Moreover, for some $\omega_0\in\Sigma$,
\begin{equation}\label{eq:fixed-total-sparse-packing}
	\sum_{I\in\mathcal A}w(a(I),b(I))|P_I|
	\leq\|w\|_{\mathrm{ft}}
	\sum_{I\in\mathcal S_w(\omega_0)}|P_I|
	\leq2\|w\|_{\mathrm{ft}}|X|.
\end{equation}
In particular, the universal estimate
\begin{equation}\label{eq:fixed-total-packing}
	\sum_{I\in\mathcal A}w(a(I),b(I))|P_I|
	\leq C|X|
\end{equation}
holds with $C=2\|w\|_{\mathrm{ft}}$. If $C_w^{\mathrm{loc}}$ denotes the best constant
$C$ in \eqref{eq:fixed-total-packing}, then
\[
	\|w\|_{\mathrm{ft}}\leq C_w^{\mathrm{loc}}\leq2\|w\|_{\mathrm{ft}}.
\]
\end{lem}

\begin{proof}
Let $q:=\|w\|_{\mathrm{ft}}^{-1}w$ and let $\pi$ be its completion from Lemma
\ref{lem:fixed-total-completion}.
Lemma \ref{lem:marked-weight-realization} therefore realizes $\pi$ as the probabilities of
the events associated with a random marked set. Applying Corollary
\ref{lem:random-marked-packing} gives a random $1/2$-sparse indexed family
$\mathcal S_w(\omega)$ with inclusion probabilities
\[
	\pi(a(I),b(I))\geq q(a(I),b(I))
	=\frac{w(a(I),b(I))}{\|w\|_{\mathrm{ft}}}.
\]
This proves \eqref{eq:fixed-total-marginals}. The selected-outcome conclusion of the
corollary also gives some $\omega_0$ such that
\[
	\frac1{\|w\|_{\mathrm{ft}}}
	\sum_{I\in\mathcal A}w(a(I),b(I))|P_I|
	\leq\sum_{I\in\mathcal A}\pi(a(I),b(I))|P_I|
	\leq\sum_{I\in\mathcal S_w(\omega_0)}|P_I|
	\leq2|X|,
\]
which proves \eqref{eq:fixed-total-sparse-packing}. In particular,
\eqref{eq:fixed-total-packing} holds with $C=2\|w\|_{\mathrm{ft}}$.

For the lower bound on $C_w^{\mathrm{loc}}$, fix $N\geq0$, take distinct abstract
indices $I_0,\ldots,I_N$, put $\mathcal A:=\{I_0,\ldots,I_N\}$, and set
\[
	P_{I_a}:=X,
	\qquad a(I_a):=a,
	\qquad b(I_a):=N-a.
\]
For $0\leq m\leq N$,
\[
	\{I\in\mathcal A:a(I)=m\}=\{I_m\},
\]
and, for any $0\leq a,c\leq N$,
\[
	a(I_c)=c\leq N=a(I_a)+b(I_a).
\]
Thus hypothesis (i) of Lemma \ref{lem:exclusion-packing} holds, and hypothesis (ii)
also holds because the union appearing there is empty. The assumed universal estimate
therefore gives
\[
	\sum_{a=0}^N w(a,N-a)|X|\leq C_w^{\mathrm{loc}}|X|.
\]
Taking the supremum over $N$ proves
$\|w\|_{\mathrm{ft}}\leq C_w^{\mathrm{loc}}$.
\end{proof}

This criterion is the complete two-depth probabilistic interface used below. Once a new
admissible weight $w$ has been proposed, its application to the covering theorem requires
only an estimate of $\|w\|_{\mathrm{ft}}$.

\subsection{From an overlap estimate to a sparse family}
\label{sec:overlap-to-sparse}

For $0<\eta\leq1$, a measurable set $A\subset I$ is an $\eta$-\emph{major subset} of $I$
if $|A|\geq\eta|I|$.
In the applications, Lemma \ref{lem:fixed-total-packing} supplies the local
weighted-overlap bound for later near successors, while embeddedness produces
major subsets that avoid far predecessors. The next lemma combines these two facts and
produces random sparse subfamilies of the original family, with a quantitative inclusion
probability for every original set. At the abstract level, near and far are only labels for the
two classes. The order of the entries is part of the label: for fixed sets with $J\prec I$,
the label is attached to $(J,I)$, with $J$ the predecessor and $I$ the later set, rather
than to $(I,J)$. In the rectangle applications the labels
describe separation in scale, not in space: a pair is far when its relevant generation gap
exceeds the embeddedness depth of the later rectangle. The proof first includes each set
independently in a preliminary random family and then discards a selected set if its selected
later near successors cover too much of it. This select-and-delete mechanism is a weighted,
local form of the alteration method from probabilistic combinatorics; see, e.g.,
\cite{AlonSpencer2016}*{Chapter 3}. The local overlap hypothesis controls the deletion
probability separately for each set and therefore yields the individual inclusion probabilities
below.

\begin{lem}[A local overlap bound produces random sparse families]
	\label{lem:local-overlap-to-marginals}
	Let $\mathcal U$ be a countable family of measurable subsets of a Euclidean
	space, each of finite positive measure, equipped with a strict total order $\prec$.
	Suppose that every ordered pair $(J,I)$ with
	$J\prec I$ and $J\cap I\neq\varnothing$ has been assigned exactly one of the labels near
	or far. Let $0\leq w(I)\leq1$ for $I\in\mathcal U$. Suppose that the following hold
	for some $0<\eta\leq1$ and $C\geq0$.
	\begin{enumerate}[(i)]
		\item For every $I\in\mathcal U$, there is a measurable set $A_I\subset I$ such that
		$|A_I|\geq\eta|I|$ and $A_I\cap J=\varnothing$ for every far predecessor
		$J$ of $I$.
		\item For every $J\in\mathcal U$, the later near successors of $J$ satisfy
		\begin{equation}\label{eq:abstract-local-near-overlap}
			\sum_{\substack{I:\,J\prec I\\(J,I)\text{ near}}}
			w(I)|J\cap I|\leq C|J|.
		\end{equation}
	\end{enumerate}
	Then one can find a probability space $(\Theta,\mathcal A,\mathbb P)$ and, for every
	$\theta\in\Theta$, an $\eta/2$-sparse subfamily
	$\mathcal G(\theta)\subset\mathcal U$ such that
	\begin{equation}\label{eq:abstract-sparse-marginals}
		\mathbb P\bigl(\{\theta\in\Theta:J\in\mathcal G(\theta)\}\bigr)
		\geq\frac{\eta}{8(C+1)}w(J)
		\qquad(J\in\mathcal U).
	\end{equation}
\end{lem}

\begin{proof}
	Set
	\[
		\varepsilon:=\frac{\eta}{4(C+1)}.
	\]
	Take $\Theta:=\{0,1\}^{\mathcal U}$, equip it with the product $\sigma$-algebra
	$\mathcal A$, and let $\mathbb P$ be the product probability measure for which the
	coordinate functions
	$\xi_I(\theta):=\theta(I)$ are independent and
	\[
		\mathbb P(\xi_I=1)=\varepsilon w(I).
	\]
	These are valid probabilities because $0\leq\varepsilon w(I)\leq1$. Write $\mathbb E$ for
	expectation with respect to $\mathbb P$. For $\theta\in\Theta$, define
	\[
		\mathcal R(\theta):=\{I\in\mathcal U:\xi_I(\theta)=1\}.
	\]

	For each $J\in\mathcal U$, define
	\begin{equation}\label{eq:abstract-later-near-coverage}
		Z_J(\theta):=\Bigg|J\cap
		\bigcup_{\substack{I\in\mathcal R(\theta):\,J\prec I\\(J,I)\text{ near}}}I\Bigg|
		=\Bigg|J\cap
		\bigcup_{\substack{I:\,J\prec I,\ (J,I)\text{ near}\\\xi_I(\theta)=1}}I\Bigg|.
	\end{equation}
	Thus $Z_J$ measures the part of $J$ covered by those later near successors that belong
	to $\mathcal R(\theta)$. Choose increasing finite subfamilies $\mathcal V_{J,N}$ whose
	union is the family of all later near successors of $J$, and put
	\[
		Z_{J,N}(\theta):=\Bigg|J\cap
		\bigcup_{\substack{I\in\mathcal V_{J,N}\\\xi_I(\theta)=1}}I\Bigg|.
	\]
	Each $Z_{J,N}$ is measurable and $Z_{J,N}\uparrow Z_J$. Moreover, subadditivity and
	finite linearity of expectation give
	\[
		\mathbb E Z_{J,N}
		\leq\varepsilon\sum_{I\in\mathcal V_{J,N}}w(I)|J\cap I|.
	\]
	The Monotone Convergence Theorem and \eqref{eq:abstract-local-near-overlap} therefore give
	\[
		\mathbb E Z_J
		=\lim_{N\to\infty}\mathbb E Z_{J,N}
		\leq\varepsilon
		\sum_{\substack{I:\,J\prec I\\(J,I)\text{ near}}}
		w(I)|J\cap I|
		\leq\varepsilon C|J|.
	\]
	Define
	\[
		B_J:=\left\{\theta\in\Theta:Z_J(\theta)>\frac\eta2|J|\right\}.
	\]
	This event is measurable because $Z_J$ is measurable.
	By the definition of $B_J$ and the preceding expectation estimate,
	\[
		\frac\eta2|J|\,\mathbb P(B_J)
		\leq\int_{B_J}Z_J(\theta)\,\ud\mathbb P(\theta)
		\leq\mathbb E Z_J
		\leq\varepsilon C|J|.
	\]
	Consequently,
	\[
		\mathbb P(B_J)\leq\frac{2\varepsilon C}{\eta}
		=\frac{C}{2(C+1)}\leq\frac12.
	\]

	The event $B_J$ is determined entirely by the variables $\xi_I$ corresponding to later
	near successors $I$ of $J$. In particular, it does not depend on $\xi_J$. The
	independence of the coordinate variables therefore shows that $B_J$ and
	$\{\xi_J=1\}$ are independent. Define
	\[
		\mathcal G(\theta):=
		\{J\in\mathcal R(\theta):\theta\notin B_J\}.
	\]
	For any fixed $J\in\mathcal U$, independence gives
	\[
		\begin{aligned}
			\mathbb P\bigl(\{\theta\in\Theta:J\in\mathcal G(\theta)\}\bigr)
			&=\mathbb P\bigl(\{\xi_J=1\}\cap(\Theta\setminus B_J)\bigr)\\
			&=\mathbb P(\xi_J=1)\mathbb P(\Theta\setminus B_J)
			\geq\frac\varepsilon2w(J)
			=\frac{\eta}{8(C+1)}w(J),
		\end{aligned}
	\]
	which proves \eqref{eq:abstract-sparse-marginals}.

	It remains to verify sparseness. Fix $\theta\in\Theta$. By the definition of
	$\mathcal G(\theta)$, we have $\theta\notin B_J$ for every
	$J\in\mathcal G(\theta)$. Recalling the definition of $B_J$, this means that
	\[
		Z_J(\theta)\leq\frac\eta2|J|
		\qquad(J\in\mathcal G(\theta)).
	\]
	For each $J\in\mathcal G(\theta)$, put
	\[
		E_J:=A_J\setminus
		\bigcup_{\substack{I\in\mathcal G(\theta):\,J\prec I\\
		(J,I)\text{ near}}}I.
	\]
	This set is measurable because $\mathcal U$ is countable.
	Since $A_J\subset J$ and $\mathcal G(\theta)\subset\mathcal R(\theta)$,
	\eqref{eq:abstract-later-near-coverage} gives
	\[
		\Bigg|A_J\cap
		\bigcup_{\substack{I\in\mathcal G(\theta):\,J\prec I\\
		(J,I)\text{ near}}}I\Bigg|
		\leq\Bigg|J\cap
		\bigcup_{\substack{I\in\mathcal R(\theta):\,J\prec I\\
		(J,I)\text{ near}}}I\Bigg|
		=Z_J(\theta)\leq\frac\eta2|J|.
	\]
	Using also $|A_J|\geq\eta|J|$, we conclude that
	\[
		|E_J|\geq|A_J|-\frac\eta2|J|\geq\frac\eta2|J|.
	\]
	Let $I,J$ be distinct members of $\mathcal G(\theta)$. Since $\prec$ is a strict total
	order, after interchanging their names if necessary, we may assume $J\prec I$. If they
	are disjoint, then so are
	$E_J$ and $E_I$. If $(J,I)$ is near, the definition of $E_J$ removes the later set $I$,
	so $E_J\cap E_I=\varnothing$. If $(J,I)$ is far, hypothesis (i) gives
	$A_I\cap J=\varnothing$; since $E_I\subset A_I$ and $E_J\subset J$, we again have
	$E_J\cap E_I=\varnothing$. Thus the sets $E_J$, $J\in\mathcal G(\theta)$, are pairwise
	disjoint. We have proved that $\mathcal G(\theta)$ is $\eta/2$-sparse.
\end{proof}

\begin{rem}[The weaker summed-overlap hypothesis]
	Suppose that $\mathcal U$ is countable and that we retain hypothesis (i) of Lemma
	\ref{lem:local-overlap-to-marginals}, but weaken the local assumption
	\eqref{eq:abstract-local-near-overlap} to the summed estimate
	\begin{equation}\label{eq:abstract-summed-near-overlap}
		\sum_{\substack{J\prec I\\(J,I)\text{ near}}}
		w(J)w(I)|J\cap I|
		\leq C\sum_{I\in\mathcal U}w(I)|I|<\infty.
	\end{equation}
	One can still conclude that there is an $\eta/2$-sparse subfamily
	$\mathcal G\subset\mathcal U$ such that
	\[
		\sum_{I\in\mathcal U}w(I)|I|
		\lesssim_{\eta,C}\sum_{I\in\mathcal G}|I|.
	\]
	The proof is an easier version of that of Lemma
	\ref{lem:local-overlap-to-marginals}: independently include $I$ with probability
	$\varepsilon w(I)$, discard the selected sets whose selected near predecessors cover
	too much of them, and then remove the remaining near predecessors from the major
	subsets. We omit the details and do not invoke this weaker principle below. It would
	suffice for the weighted $L^1$ packing conclusions in both settings. The exponential
	integrability conclusion in the bi-parameter setting genuinely requires the stronger
	local hypothesis, which gives the individual inclusion probabilities in
	\eqref{eq:abstract-sparse-marginals}. We also call these marginal probabilities.
\end{rem}

\begin{rem}[The weighted $L^1$ consequence of the marginal estimate]
	\label{rem:marginals-to-scalar}
	Suppose that the hypotheses of Lemma \ref{lem:local-overlap-to-marginals} hold and that
	$|\operatorname{sh}(\mathcal U)|<\infty$. Put
	\[
		\kappa:=\frac{\eta}{8(C+1)},
		\qquad
		S:=\sum_{I\in\mathcal U}w(I)|I|,
		\qquad
		M(\theta):=\sum_{I\in\mathcal G(\theta)}|I|.
	\]
	Since $\mathcal G(\theta)$ is $\eta/2$-sparse, its disjoint major subsets are contained
	in $\operatorname{sh}(\mathcal U)$, and hence
	\[
		M(\theta)\leq\frac2\eta|\operatorname{sh}(\mathcal U)|,
		\qquad \theta\in\Theta.
	\]
	On the other hand, \eqref{eq:abstract-sparse-marginals} and the non-negativity of the
	terms give
	\[
		\mathbb E M
		=\sum_{I\in\mathcal U}
		\mathbb P\bigl(\{\theta:I\in\mathcal G(\theta)\}\bigr)|I|
		\geq\kappa S.
	\]
	Thus $S<\infty$. If $S>0$, choose $\theta_0\in\Theta$ such that
	$M(\theta_0)\geq\kappa S/2$. Such an outcome exists because otherwise
	$\mathbb E M\leq\kappa S/2$, contrary to the preceding estimate. If $S=0$, choose any
	$\theta_0$. Then
	$\mathcal G:=\mathcal G(\theta_0)$ is $\eta/2$-sparse and
	\[
		S\leq\frac2\kappa\sum_{I\in\mathcal G}|I|
		\leq\frac4{\kappa\eta}|\operatorname{sh}(\mathcal G)|.
	\]
	This is the weighted $L^1$ sparse conclusion in the preceding remark. Thus, compared with the
	summed hypothesis \eqref{eq:abstract-summed-near-overlap}, the stronger local hypothesis
	\eqref{eq:abstract-local-near-overlap} is needed not for that conclusion, but for the
	individual inclusion probabilities in \eqref{eq:abstract-sparse-marginals} and their
	exponential consequences via Subsection \ref{sec:convex-transfer}.
\end{rem}

\subsection{Convex transfer from inclusion probabilities}
\label{sec:convex-transfer}

The following elementary observation is independent of how the random subfamilies are
constructed. It uses only their inclusion probabilities and transfers a uniform convex
overlap estimate for the selected families to the original weighted family.

\begin{lem}[From inclusion probabilities to convex overlap estimates]
	\label{lem:inclusion-probabilities-to-convex-overlap}
	Let $(X,\mu)$ be a $\sigma$-finite measure space, let $\mathcal U$ be a countable family of measurable
	subsets of $X$, and let $w(I)\geq0$ for $I\in\mathcal U$. Let
	$(\Theta,\mathcal A,\mathbb P)$ be an arbitrary probability space. Suppose that one can
	associate to each $\theta\in\Theta$ a subfamily $\mathcal G(\theta)\subset\mathcal U$ so
	that each event
	\[
		\{\theta\in\Theta:I\in\mathcal G(\theta)\}
	\]
	is measurable. Define
	\[
		h_\theta(x):=\sum_{I\in\mathcal G(\theta)}1_I(x),
		\qquad
		H_w(x):=\sum_{I\in\mathcal U}w(I)1_I(x).
	\]
	Assume that, for some $\kappa>0$,
	\begin{equation}\label{eq:abstract-inclusion-probabilities}
		\mathbb P\bigl(\{\theta\in\Theta:I\in\mathcal G(\theta)\}\bigr)
		\geq\kappa w(I),
		\qquad I\in\mathcal U.
	\end{equation}
	Let $\Phi\colon[0,\infty)\to[0,\infty)$ be increasing and convex. Then
	\begin{equation}\label{eq:abstract-convex-transfer}
		\int_X\Phi\bigl(\kappa H_w(x)\bigr)\,\ud\mu(x)
		\leq
		\int_\Theta\int_X\Phi\bigl(h_\theta(x)\bigr)
		\,\ud\mu(x)\,\ud\mathbb P(\theta).
	\end{equation}
	Moreover, suppose that
	\[
		\int_\Theta\int_X\Phi\bigl(h_\theta(x)\bigr)
		\,\ud\mu(x)\,\ud\mathbb P(\theta)<\infty.
	\]
	Then there is $\theta_0\in\Theta$ such that
	\begin{equation*}
		\int_X\Phi\bigl(\kappa H_w(x)\bigr)\,\ud\mu(x)
		\leq\int_X\Phi\bigl(h_{\theta_0}(x)\bigr)\,\ud\mu(x).
	\end{equation*}
\end{lem}

\begin{proof}
	Enumerate $\mathcal U=\{I_1,I_2,\ldots\}$, with the evident modification when the
	family is finite, and put
	\[
		A_j:=\{\theta\in\Theta:I_j\in\mathcal G(\theta)\}.
	\]
	For $N\geq1$, define
	\[
		H_N(x):=\sum_{j=1}^Nw(I_j)1_{I_j}(x),
		\qquad
		h_{\theta,N}(x):=\sum_{j=1}^N1_{A_j}(\theta)1_{I_j}(x).
	\]
	For every fixed $x\in X$, finite linearity of expectation and
	\eqref{eq:abstract-inclusion-probabilities} give
	\[
		\mathbb E h_{\theta,N}(x)
		=\sum_{j=1}^N\mathbb P(A_j)1_{I_j}(x)
		\geq\kappa H_N(x).
	\]
	Since $\Phi$ is increasing and convex, Jensen's inequality yields
	\[
		\Phi\bigl(\kappa H_N(x)\bigr)
		\leq\Phi\bigl(\mathbb E h_{\theta,N}(x)\bigr)
		\leq\mathbb E\Phi\bigl(h_{\theta,N}(x)\bigr)
		\leq\mathbb E\Phi\bigl(h_\theta(x)\bigr).
	\]
	Integrating this pointwise inequality over $X$ and interchanging the two non-negative
	integrals on the right gives
	\[
		\int_X\Phi\bigl(\kappa H_N(x)\bigr)\,\ud\mu(x)
		\leq\int_\Theta\int_X\Phi\bigl(h_\theta(x)\bigr)
		\,\ud\mu(x)\,\ud\mathbb P(\theta).
	\]
	Finally, $H_N\uparrow H_w$ because all the weights are non-negative. Since every
	finite-valued increasing convex function on $[0,\infty)$ is continuous,
	$\Phi(\kappa H_N)\uparrow\Phi(\kappa H_w)$. The Monotone Convergence Theorem
	proves \eqref{eq:abstract-convex-transfer}.

	For the final assertion, put
	\[
		F(\theta):=\int_X\Phi\bigl(h_\theta(x)\bigr)\,\ud\mu(x).
	\]
	The function
	\[
		(\theta,x)\longmapsto h_\theta(x)
		=\sum_{j\geq1}1_{A_j}(\theta)1_{I_j}(x)
	\]
	is measurable on $\Theta\times X$ as a countable sum of measurable functions. Since
	$\Phi$ is continuous, $\Phi(h_\theta(x))$ is also jointly measurable. Hence $F$ is
	measurable as the integral in $x$ of a non-negative jointly measurable function. By
	hypothesis,
	\[
		m:=\int_\Theta F(\theta)\,\ud\mathbb P(\theta)<\infty.
	\]
	Let $E:=\{\theta:F(\theta)<\infty\}$, so $\mathbb P(E)=1$. Suppose that
	$\{\theta\in E:F(\theta)\geq m\}=\varnothing$. Then
	\[
		E=\bigcup_{n=1}^\infty D_n,
		\qquad
		D_n:=\{\theta\in E:m-F(\theta)\geq1/n\}.
	\]
	Hence $\mathbb P(D_n)>0$ for some $n$, and
	\[
		0=\int_E(F-m)\,\ud\mathbb P
		=-\int_E(m-F)\,\ud\mathbb P
		\leq-\frac1n\mathbb P(D_n)<0,
	\]
	a contradiction. Choose $\theta_0\in E$ such that $F(\theta_0)\geq m$. Then
	the first part of the proof, namely \eqref{eq:abstract-convex-transfer}, gives the
	first inequality in
	\[
		\int_X\Phi\bigl(\kappa H_w(x)\bigr)\,\ud\mu(x)
		\leq m\leq F(\theta_0)
		=\int_X\Phi\bigl(h_{\theta_0}(x)\bigr)\,\ud\mu(x)
	\]
	and proves the final assertion.
\end{proof}

\subsection{Exclusion events from summable weights}
\label{sec:stationary-one-depth-selection}

A decreasing summable weight on $\N$ produces the following
exclusion family on $\Z$. Its principal application is the diagonal-envelope
exponential theorem in Subsection \ref{sec:zygmund-full-exponential}; its later
one-depth use merely packages the classical residue-class argument in the present
probabilistic framework.

\begin{lem}[Exclusion events from summable weights]
\label{lem:stationary-one-depth-events}
Let $\omega\colon\N\to[0,\infty)$ be decreasing, summable and nonzero.
There exist a probability space $(\Theta_\omega,\mathcal A_\omega,\mathbb P_\omega)$ and
an exclusion family $(\mathsf S^\omega_{a,b})_{a\in\Z,\,b\geq0}$ such that
\begin{equation}\label{eq:stationary-one-depth-probability}
	\mathbb P_\omega(\mathsf S^\omega_{a,b})
	=\frac{\omega(b)}{\|\omega\|_{\ell^1(\N)}},
	\qquad a\in\Z,\quad b\geq0.
\end{equation}
Moreover, these events satisfy the two-sided exclusion property
\begin{equation}\label{eq:stationary-two-sided-exclusion}
	\mathsf S^\omega_{a,b}\cap\mathsf S^\omega_{a',b'}=\varnothing
	\quad\text{whenever}\quad
	0<|a-a'|\leq\max\{b,b'\}.
\end{equation}
\end{lem}

\begin{proof}
Define
\[
	\delta(k):=\omega(k)-\omega(k+1),
	\qquad k\geq0.
\]
These numbers are nonnegative. Since summability gives $\omega(N)\to0$, telescoping
first gives
\begin{equation}\label{eq:one-depth-weight-telescoping}
	\omega(j)=\sum_{k=j}^{\infty}\delta(k),
	\qquad j\geq0.
\end{equation}
A second finite telescoping calculation gives
\[
	\begin{aligned}
		\sum_{k=0}^N(k+1)\delta(k)
		&=\sum_{k=0}^N(k+1)\omega(k)
		  -\sum_{k=0}^N(k+1)\omega(k+1)\\
		&=\sum_{k=0}^N(k+1)\omega(k)
		  -\sum_{j=1}^{N+1}j\omega(j)\\
		&=\sum_{k=0}^N\omega(k)-(N+1)\omega(N+1).
	\end{aligned}
\]
Moreover, monotonicity and summability imply
\[
	(N+1)\omega(N+1)
	\leq2\sum_{k=\lfloor(N+1)/2\rfloor}^{N+1}\omega(k)
	\leq2\sum_{k=\lfloor(N+1)/2\rfloor}^{\infty}\omega(k)
	\longrightarrow0.
\]
Letting $N\to\infty$ therefore yields
\begin{equation}\label{eq:one-depth-delta-mass}
	\sum_{k=0}^{\infty}(k+1)\delta(k)=\|\omega\|_{\ell^1(\N)}.
\end{equation}

Take
\[
	\Theta_\omega:=\{(k,r):k\geq0,\ 0\leq r\leq k\},
	\qquad
	\mathcal A_\omega:=\mathcal P(\Theta_\omega),
\]
and define the discrete probability measure
\[
	\mathbb P_\omega(\{(k,r)\})
	:=\frac{\delta(k)}{\|\omega\|_{\ell^1(\N)}}.
\]
Identity \eqref{eq:one-depth-delta-mass} shows that the masses sum to one. For
$\theta=(k,r)\in\Theta_\omega$, mark the integers
\[
	\mathcal T(\theta):=\{n\in\Z:n\equiv r\pmod{k+1}\}.
\]
For $a\in\Z$ and $b\geq0$, let $\mathsf S^\omega_{a,b}$ be the event that $a$ is
marked while $a+1,\ldots,a+b$ are unmarked. If $\theta=(k,r)$ belongs to both
$\mathsf S^\omega_{a,b}$ and $\mathsf S^\omega_{a',b'}$, then
$k\geq\max\{b,b'\}$ and both $a$ and $a'$ are congruent to $r$ modulo $k+1$. Thus,
if $a\ne a'$, then
\[
	|a-a'|\geq k+1>\max\{b,b'\}.
\]
This proves \eqref{eq:stationary-two-sided-exclusion}; in particular, the events form
an exclusion family on $\Z$.

For a fixed $k\geq b$, there is exactly one $r\in\{0,\ldots,k\}$ for which $a$ is
marked, and the next marked integer is $a+k+1>a+b$. If $k<b$, no such $r$ contributes.
Consequently, \eqref{eq:one-depth-weight-telescoping} gives
\[
	\mathbb P_\omega(\mathsf S^\omega_{a,b})
	=\frac1{\|\omega\|_{\ell^1(\N)}}\sum_{k=b}^{\infty}\delta(k)
	=\frac{\omega(b)}{\|\omega\|_{\ell^1(\N)}},
\]
which is \eqref{eq:stationary-one-depth-probability}.
\end{proof}

\section{Concrete fixed-total weights}
\label{sec:concrete-fixed-total-weights}

The fixed-total criterion becomes concrete once membership in
$\mathcal W_{\mathrm{ft}}\cap\mathcal W_{12}$ and the corresponding fixed-total norm
have been determined. We first determine the exact range of the two-parameter power
weights, including quantitative bounds on the critical line. We then give an asymmetric
logarithmic endpoint and a genuinely coupled ratio-logarithmic weight.

\subsection{Power weights and the sharp range}
\label{sec:power-weight-range}

For $\alpha,\beta\geq0$, define
\[
	w_{\alpha,\beta}(a,b):=(a+1)^{-\alpha}(b+1)^{-\beta},
	\qquad a,b\geq0.
\]
Plainly $w_{\alpha,\beta}\in\mathcal W_{12}$. The following calculation completely
determines when it has finite fixed-total norm.

\begin{lem}[Fixed-total norm of the power weights]
\label{lem:power-weight-fixed-total-range}
For $\alpha,\beta\geq0$, we have $w_{\alpha,\beta}\in\mathcal W_{\mathrm{ft}}$ if and
only if
\[
	\alpha+\beta\geq1,
	\qquad
	(\alpha,\beta)\notin\{(1,0),(0,1)\}.
\]
More quantitatively, if $\alpha,\beta>0$ and $\alpha+\beta\geq1$, then, with
\[
	m_{\alpha,\beta}:=\min\{\alpha,\beta,1/2\},
\]
we have
\[
	\|w_{\alpha,\beta}\|_{\mathrm{ft}}
	\leq\frac1{m_{\alpha,\beta}(1-m_{\alpha,\beta})}.
\]
On the critical line $\alpha+\beta=1$ with $\alpha,\beta>0$, this dependence on
$\alpha$ and $\beta$ is sharp up to an absolute factor:
\begin{equation}\label{eq:power-weight-critical-line}
	\frac1{4\alpha\beta}
	\leq\|w_{\alpha,\beta}\|_{\mathrm{ft}}
	\leq\frac1{\alpha\beta}.
\end{equation}
The admissible axis cases satisfy
\[
	\|w_{\alpha,0}\|_{\mathrm{ft}}=\sum_{k=1}^\infty k^{-\alpha}
	\quad(\alpha>1),
	\qquad
	\|w_{0,\beta}\|_{\mathrm{ft}}=\sum_{k=1}^\infty k^{-\beta}
	\quad(\beta>1).
\]
\end{lem}

\begin{proof}
	Put
	\[
		S_N(\alpha,\beta)
		:=\sum_{a=0}^N(a+1)^{-\alpha}(N-a+1)^{-\beta}.
	\]
	We first estimate the critical line. Fix $0<s<1$ and put $L:=(N+2)/2$. On the
	part of the sum where $a\leq N/2$, we have $N-a+1\geq L$, and hence
	\[
		\begin{aligned}
		\sum_{0\leq a\leq N/2}w_{s,1-s}(a,N-a)
		&\leq L^{-(1-s)}\sum_{0\leq a\leq N/2}(a+1)^{-s}\\
		&\leq L^{-(1-s)}\int_0^L x^{-s}\,\ud x
		=\frac1{1-s}.
		\end{aligned}
	\]
	On the complementary part, put $b:=N-a$. Then $0\leq b<N/2$ and
	$N-b+1>L$, so the symmetric calculation gives
	\[
		\sum_{N/2<a\leq N}w_{s,1-s}(a,N-a)\leq\frac1s.
	\]
	Consequently,
	\begin{equation}\label{eq:power-weight-critical-upper}
		\|w_{s,1-s}\|_{\mathrm{ft}}
		\leq\frac1{1-s}+\frac1s
		=\frac1{s(1-s)}.
	\end{equation}

	We next prove the lower bound in \eqref{eq:power-weight-critical-line}. Assume first
	that $0<s\leq1/2$ and restrict to odd totals $N=2M-1$. After the
	change of variables $k=2M-a$, the terms with $1\leq k\leq M$ satisfy
	$a+1=2M-k+1\leq2M$. Therefore
	\[
		\begin{aligned}
		S_{2M-1}(s,1-s)
		&\geq(2M)^{-s}\sum_{k=1}^{M}k^{s-1}\\
		&\geq\frac{(2M)^{-s}}s\bigl((M+1)^s-1\bigr).
		\end{aligned}
	\]
	Letting $M\to\infty$ gives
	\[
		\|w_{s,1-s}\|_{\mathrm{ft}}
		\geq\frac{2^{-s}}s
		\geq\frac1{2s}
		\geq\frac1{4s(1-s)}.
	\]
	For $1/2\leq s<1$, interchange the variables. Consequently, for
	$w_s:=w_{s,1-s}$, we have proved
	\begin{equation}\label{eq:critical-power-fixed-total}
		\frac1{4s(1-s)}
		\leq\|w_s\|_{\mathrm{ft}}
		\leq\frac1{s(1-s)},
		\qquad 0<s<1.
	\end{equation}
	This proves \eqref{eq:power-weight-critical-line}.

	Now suppose that $\alpha,\beta>0$ and $\alpha+\beta\geq1$. Choose
	$t\in(0,1)$ by
	\[
		t:=
		\begin{cases}
			\alpha,&\alpha<1/2,\\
			1-\beta,&\beta<1/2,\\
			1/2,&\alpha,\beta\geq1/2.
		\end{cases}
	\]
	The first two cases cannot occur simultaneously. In every case,
	$t\leq\alpha$, $1-t\leq\beta$, and
	$t(1-t)=m_{\alpha,\beta}(1-m_{\alpha,\beta})$. Thus
	$w_{\alpha,\beta}\leq w_{t,1-t}$ pointwise, and
	\eqref{eq:power-weight-critical-upper} yields the asserted quantitative upper bound.

	If $\beta=0$, then
	\[
		\|w_{\alpha,0}\|_{\mathrm{ft}}
		=\sup_{N\geq0}\sum_{a=0}^N(a+1)^{-\alpha}
		=\sum_{k=1}^\infty k^{-\alpha},
	\]
	which is finite precisely when $\alpha>1$. The case $\alpha=0$ is symmetric.

	If $\alpha+\beta<1$, the terms with $N/3\leq a\leq2N/3$ give, for $N\geq6$,
	\[
		S_N(\alpha,\beta)
		\geq\left(\frac N3-1\right)(N+1)^{-\alpha-\beta}
		\longrightarrow\infty.
	\]
	At $(\alpha,\beta)=(1,0)$ the sums $S_N(1,0)$ are the harmonic partial sums, and
	the case $(0,1)$ is symmetric. Thus the fixed-total norm is infinite in every remaining
	case.
\end{proof}

The critical packing cannot follow from one-depth estimates by pointwise domination: if
$w_s(a,b)\leq u(a)+v(b)$ with $u,v\in\ell^1(\N)$, then
$1/(n+1)=w_s(n,n)\leq u(n)+v(n)$, contradicting summability.

\subsection{A logarithmic critical weight}
\label{sec:logarithmic-weight}

The power weights are not the only genuinely critical weights allowed by the
fixed-total criterion. The following asymmetric example decays harmonically in one depth
and only logarithmically in the other. Put
\[
	\Lambda(n):=1+\log(n+1),
	\qquad n\geq0,
\]
and define
\[
	w_{\log}(a,b)
	:=\frac1{\Lambda(a)(b+1)},
	\qquad a,b\geq0.
\]
Plainly $w_{\log}\in\mathcal W_{12}$, although neither of its one-variable
factors is summable.

\begin{lem}[The logarithmic weight]
\label{lem:logarithmic-critical-weight}
The weight $w_{\log}$ belongs to
$\mathcal W_{\mathrm{ft}}\cap\mathcal W_{12}$, and
\[
	\|w_{\log}\|_{\mathrm{ft}}\leq2.
\]
\end{lem}

\begin{proof}
Fix $N\geq0$, put $L:=(N+2)/2$, and split the fixed-total line at $N/2$.
On the first half, $N-a+1\geq L$ and $\Lambda(a)\geq1$, so
\[
	\sum_{0\leq a\leq N/2}\frac1{\Lambda(a)(N-a+1)}
	\leq\frac1L\sum_{0\leq a\leq N/2}1\leq1.
\]
On the complementary half, set $b:=N-a$. Then $0\leq b<N/2$ and
$N-b+1>L$, so $\Lambda(N-b)>1+\log L$. The usual harmonic-sum estimate then gives
\[
	\begin{aligned}
	\sum_{N/2<a\leq N}\frac1{\Lambda(a)(N-a+1)}
	&=\sum_{0\leq b<N/2}\frac1{\Lambda(N-b)(b+1)}\\
	&\leq\frac1{1+\log L}\sum_{0\leq b<N/2}\frac1{b+1}
	\leq1.
	\end{aligned}
\]
Adding these estimates and taking the supremum over $N$ proves the asserted norm bound.
\end{proof}

This weight is not obtained by comparison with any single critical power weight. Indeed,
for every fixed $0<s<1$,
\[
	\frac{w_{\log}(a,0)}{w_s(a,0)}
	=\frac{(a+1)^s}{\Lambda(a)}\longrightarrow\infty
	\qquad\text{as }a\to\infty.
\]
Thus the logarithmic weight gives a genuinely different application of the exact
fixed-total criterion.

\subsection{A genuinely coupled weight}
\label{sec:ratio-logarithmic-weight}

The preceding examples have tensor-product form. The next one combines harmonic decay in
the smaller depth with logarithmic decay in the ratio of the two shifted depths. Define
\begin{equation*}
	w_{\mathrm{rat}}(a,b)
	:=\frac{1}{(1+\min\{a,b\})
		\left(2+\left|\log\frac{a+1}{b+1}\right|\right)^2},
	\qquad a,b\geq0.
\end{equation*}

\begin{lem}[A coupled fixed-total weight]
\label{lem:ratio-logarithmic-weight}
The weight $w_{\mathrm{rat}}$ belongs to
$\mathcal W_{\mathrm{ft}}\cap\mathcal W_{12}$, and
\begin{equation}\label{eq:ratio-logarithmic-fixed-total}
	\|w_{\mathrm{rat}}\|_{\mathrm{ft}}\leq\frac32.
\end{equation}
\end{lem}

\begin{proof}
	We first verify coordinatewise monotonicity. Since
	$|\log(x/y)|=-\log(x/y)=\log(y/x)$ when $1\leq x\leq y$, write on this region
	$F(x,y):=\{x[2+\log(y/x)]^2\}^{-1}$.
	The denominator is nondecreasing in $y$, while its derivative with respect to $x$ is
	$[2+\log(y/x)]\log(y/x)\geq0$.
	Thus $F$ is nonincreasing in both variables on this region. By symmetry, $F(y,x)$ gives
	the same conclusion on $1\leq y\leq x$. Hence $w_{\mathrm{rat}}\in\mathcal W_{12}$.

	Fix $N\geq0$. By symmetry,
	\begin{equation*}
		\sum_{a=0}^Nw_{\mathrm{rat}}(a,N-a)
		\leq2\sum_{0\leq a\leq N/2}w_{\mathrm{rat}}(a,N-a).
	\end{equation*}
	For $0\leq a\leq N/2$, we have $N-a+1\geq(N+2)/2$, and hence
	\begin{equation*}
		w_{\mathrm{rat}}(a,N-a)\leq g_N(a+1),
		\qquad
		g_N(t):=\frac{1}{t[2+\log((N+2)/(2t))]^2},
		\quad 1\leq t\leq\frac{N+2}{2}.
	\end{equation*}
	The function $g_N$ is decreasing on this interval: the derivative of its denominator is
	$[2+\log((N+2)/(2t))]\log((N+2)/(2t))\geq0$.
	Since $g_N(1)\leq1/4$, the integral comparison for a decreasing function now gives
	\begin{align*}
		\sum_{0\leq a\leq N/2}w_{\mathrm{rat}}(a,N-a)
		&\leq\frac14+
		\int_1^{(N+2)/2}\frac{\ud t}{t[2+\log((N+2)/(2t))]^2}\\
		&=\frac14+
		\left.\frac{1}{2+\log((N+2)/(2t))}\right|_{t=1}^{t=(N+2)/2}\\
		&=\frac14+\frac12-\frac{1}{2+\log((N+2)/2)}\leq\frac34.
	\end{align*}
	This proves \eqref{eq:ratio-logarithmic-fixed-total}.
\end{proof}

Suppose that $w_{\mathrm{rat}}\leq C\widetilde w$ for some $C>0$ and some product-form
weight $\widetilde w(a,b)=u(a)v(b)$ with $u,v\geq0$. Set $L_n:=2+\log(n+1)$, so that
$w_{\mathrm{rat}}(n,0)=w_{\mathrm{rat}}(0,n)=L_n^{-2}$. Then $\widetilde w(0,0)>0$ and
$\widetilde w(a,b)=\widetilde w(a,0)\widetilde w(0,b)/\widetilde w(0,0)$. For $N\geq1$,
apply this identity to the fixed-total sum with total $N$. Since
$L_a,L_{N-a}\leq L_N$ for $0\leq a\leq N$, the axis formula gives
\begin{equation*}
	\|\widetilde w\|_{\mathrm{ft}}
	\geq\frac1{C^2\widetilde w(0,0)}\sum_{a=0}^{N}\frac1{L_a^2L_{N-a}^2}
	\geq\frac{N+1}{C^2\widetilde w(0,0)L_N^4}.
\end{equation*}
The right-hand side tends to infinity with $N$, so $\|\widetilde w\|_{\mathrm{ft}}=\infty$.
Thus no product-form weight in $\mathcal W_{\mathrm{ft}}$ dominates $w_{\mathrm{rat}}$
up to a constant.

\section{Two-depth packing in the bi-parameter product setting}
\label{sec:planar-critical}

We first record the simple layered family that forces the fixed-total condition. We then
prove the converse, together with the marginal and sparse conclusions of Theorem
\ref{thm:intro-planar-admissible}. The critical power estimate follows directly from
\eqref{eq:critical-power-fixed-total}. The sufficiency argument serves as the model for the
Zygmund proof.

\subsection{The layered test family}

The following family contains the whole necessity argument. It is related to the simplex
construction for dyadic antichains in
\cite{Rey2026antichain}*{Section 3.2.2}, but here every layer is completed to a partition.
We write the example with one-dimensional coordinate factors. In arbitrary fixed
dimensions the construction and proof are unchanged, with dyadic cubes in place of
intervals.

\begin{prop}[The layered test family]\label{prop:layered-test-family}
Work with the standard dyadic grids on $\R$ and put $\Omega:=[0,1)^2$. For
$N\geq0$, let $\mathcal U_N$ consist of all bi-parameter dyadic rectangles
$I=I^1\times I^2\subset\Omega$ such that, for some $0\leq a\leq N$,
\[
	\ell(I^1)=2^{-a},
	\qquad
	\ell(I^2)=2^{-(N-a)}.
\]
Then $\mathcal U_N$ is finite and pairwise incomparable. Moreover, for every
$w\colon\N^2\to[0,\infty)$,
\begin{equation}\label{eq:layered-height-identity}
	H_{w,N}(x)
	:=\sum_{I\in\mathcal U_N}
	w(e_1(I;\Omega),e_2(I;\Omega))1_I(x)
	=\sum_{a=0}^Nw(a,N-a),
	\qquad x\in\Omega,
\end{equation}
and hence
\begin{equation}\label{eq:layered-mass-identity}
	\frac1{|\Omega|}\sum_{I\in\mathcal U_N}
	w(e_1(I;\Omega),e_2(I;\Omega))|I|
	=\sum_{a=0}^Nw(a,N-a).
\end{equation}
Consequently, the constant $C_w$ in any uniform two-depth packing estimate with weight
$w$ satisfies $\|w\|_{\mathrm{ft}}\leq C_w$.
\end{prop}

\begin{proof}
	For each fixed $a$, the corresponding rectangles form the $a$th layer and partition
	$\Omega$. Every member of $\mathcal U_N$ has area $2^{-N}$, so distinct members cannot
	contain one another. Thus $\mathcal U_N$ is pairwise incomparable.

	If $I$ belongs to the $a$th layer, then
	\begin{equation}\label{eq:planar-layer-depths}
		e_1(I;\Omega)=a,
		\qquad e_2(I;\Omega)=N-a.
	\end{equation}
	We prove the first equality. Since $(I^1)^{(a)}=[0,1)$, we have
	$(I^1)^{(a)}\times I^2\subset\Omega\subset\widetilde\Omega$, and hence
	$e_1(I;\Omega)\geq a$. Choose
	\[
		x_1\in(1,2),
		\qquad x_2\in\operatorname{int}(I^2).
	\]
	Then $(x_1,x_2)\in(I^1)^{(a+1)}\times I^2$. If a dyadic product rectangle
	$R^1\times R^2$ through this point meets $\Omega$, then $R^1$ meets $[0,1)$ and
	contains a point of $(1,2)$. Hence $R^1\supset[0,2)$, and therefore
	\[
		\frac{|(R^1\times R^2)\cap\Omega|}{|R^1\times R^2|}
		=\frac{|R^1\cap[0,1)|}{|R^1|}
		 \frac{|R^2\cap[0,1)|}{|R^2|}
		\leq\frac12.
	\]
	A rectangle that misses $\Omega$ has density zero. Taking the supremum over all dyadic
	product rectangles through $(x_1,x_2)$ gives
	$M_{\calD^1\times\calD^2}1_\Omega(x_1,x_2)\leq1/2$. Thus
	$(I^1)^{(a+1)}\times I^2$ is not contained in $\widetilde\Omega$, proving the first
	equality in \eqref{eq:planar-layer-depths}. The second is symmetric.

	For every $x\in\Omega$ and every $0\leq a\leq N$, exactly one rectangle in the $a$th
	layer contains $x$. Combining this partition property with
	\eqref{eq:planar-layer-depths} proves \eqref{eq:layered-height-identity}; integration
	gives \eqref{eq:layered-mass-identity}. If a uniform packing estimate holds with constant
	$C_w$, then \eqref{eq:layered-mass-identity} gives
	\[
		\sum_{a=0}^Nw(a,N-a)\leq C_w,
		\qquad N\geq0.
	\]
	Taking the supremum over $N$ proves $\|w\|_{\mathrm{ft}}\leq C_w$.
\end{proof}

Taking $w(a,b)=\omega(a)$ in \eqref{eq:layered-mass-identity} gives
\begin{equation}\label{eq:layered-one-depth-identity}
	\frac1{|\Omega|}\sum_{I\in\mathcal U_N}\omega(e_1(I;\Omega))|I|
	=\sum_{a=0}^N\omega(a).
\end{equation}
Thus \eqref{eq:layered-one-depth-identity} gives the one-depth necessity
$\omega\in\ell^1(\N)$ for the same family.

For every nonzero admissible weight $w$, define
\begin{equation}\label{eq:general-marginal-coefficient}
	\kappa_w:=\frac1{160\|w\|_{\mathrm{ft}}}.
\end{equation}

\begin{thm}[Admissible two-depth bi-parameter packing]
\label{thm:planar-admissible-weight}
Let $(\Omega,\mathcal U)$ be an admissible pair in the bi-parameter setting, and let $w$ be
a nonzero admissible weight. Then
there exist a probability space $(\Theta,\mathcal A,\mathbb P)$ and, for every
$\theta\in\Theta$, a $1/8$-sparse subfamily $\mathcal G(\theta)\subset\mathcal U$ such that
\begin{equation}\label{eq:planar-general-marginals}
	\mathbb P\bigl(\{\theta\in\Theta:I\in\mathcal G(\theta)\}\bigr)
	\geq\kappa_w W_w(I;\Omega),
	\qquad I\in\mathcal U.
\end{equation}
Moreover, there is a $1/8$-sparse subfamily $\mathcal G_0\subset\mathcal U$ such that
\begin{equation}\label{eq:planar-general-sparse-packing}
	\int_\Omega H_w(x)\,\ud x
	=\sum_{I\in\mathcal U}W_w(I;\Omega)|I|
	\leq2560\|w\|_{\mathrm{ft}}|\operatorname{sh}(\mathcal G_0)|.
\end{equation}
\end{thm}

\subsection{Ordering the rectangles and the two possible intersections}

Fix an injective numbering
$\iota\colon\calD^1\times\calD^2\to\mathbb N$. For $I\in\mathcal U$, write
\[
	\ell(I^i)=2^{-\nu_i(I)},
	\qquad
	\tau(I):=\nu_1(I)+\nu_2(I),
\]
so that
$\ell(I^1)\ell(I^2)=2^{-\tau(I)}$. Thus $\tau(I)$ is the combined dyadic generation of
the two coordinate cubes: smaller $\tau(I)$ means a larger product of their side lengths.
For distinct $I,J\in\mathcal U$, define their order as follows. If
$\tau(J)\neq\tau(I)$, put $J\prec I$ precisely when $\tau(J)<\tau(I)$. If
$\tau(J)=\tau(I)$, put $J\prec I$ precisely when $\iota(J)<\iota(I)$. Because $\iota$ is
injective, the pairs $(\tau(I),\iota(I))$ are distinct. Comparing their first entries and,
only when those agree, their second entries gives a strict total order on $\mathcal U$.
In particular, $J\prec I$ implies $\tau(J)\leq\tau(I)$.

If $J\prec I$ and $J\cap I\neq\varnothing$, the two cubes in each coordinate meet and
are therefore nested. Since $J\prec I$, the rectangles are distinct. Their pairwise
incomparability rules out both
\[
	J^1\subset I^1,\quad J^2\subset I^2,
	\qquad\text{and}\qquad
	I^1\subset J^1,\quad I^2\subset J^2,
\]
because either pair of inclusions would make one product rectangle contain the other.
Moreover, equality cannot hold in either coordinate. For example, if $J^1=I^1$, then
nesting of $J^2$ and $I^2$ would make one of the two displayed pairs hold. The same
argument applies if $J^2=I^2$. Thus the nesting is strict in both coordinates, and exactly
two opposite-direction possibilities remain:
	\begin{align}
		J^1\supsetneq I^1,\quad J^2\subsetneq I^2,
		&\qquad I\cap J=I^1\times J^2;
		\label{eq:planar-first-intersection}\\
		J^1\subsetneq I^1,\quad J^2\supsetneq I^2,
		&\qquad I\cap J=J^1\times I^2.
		\label{eq:planar-second-intersection}
	\end{align}
In the first case, put
	\[
		p=\nu_1(I)-\nu_1(J),
		\qquad q=\nu_2(J)-\nu_2(I).
	\]
Since $J\prec I$, we have $\tau(J)\leq\tau(I)$, and hence
	\[
		0\leq\tau(I)-\tau(J)=p-q,
		\qquad\text{so }p\geq q\geq1.
	\]
Call the pair near when $p\leq e_1(I;\Omega)$. In the second case, put
\[
	p=\nu_1(J)-\nu_1(I),
	\qquad
	q=\nu_2(I)-\nu_2(J).
\]
Since $J\prec I$, we have $\tau(J)\leq\tau(I)$, and hence
\[
	0\leq\tau(I)-\tau(J)=q-p,
	\qquad\text{so }q\geq p\geq1.
\]
Call the pair near when $q\leq e_2(I;\Omega)$. For every $J\prec I$ with
$J\cap I\neq\varnothing$, declare the pair near according to the preceding rule and far
otherwise. Thus $(J,I)$ near or far is oriented notation: its first entry is the earlier
rectangle and its second entry is the later rectangle. Both rules compare two numbers
measured in dyadic generations. In the first
case, $J^1=(I^1)^{(p)}$, so passing from $I^1$ to $J^1$ requires $p$ parent steps. In the
second case, $J^2=(I^2)^{(q)}$, so passing from $I^2$ to $J^2$ requires $q$ parent steps.
Thus nearness means that this number of parent steps does not exceed the corresponding
embeddedness depth of $I$. The next lemma shows that a fixed positive portion of $I$
avoids all far predecessors; the near pairs are handled by the weighted summation in the
following subsection.

The labels \emph{near} and \emph{far} are specific to the present proof. Since the
rectangles intersect, they describe scale separation rather than Euclidean distance. There
are one-depth precedents for the far part of the argument. Hyt\"onen--Martikainen
\cite{Hytonen2014}*{Section 8.1} group rectangles of embeddedness at most $k$ by one
coordinate generation modulo $k+1$, so that the relevant nested cubes are separated by
more than the available depth. Cabrelli--Lacey--Molter--Pipher
\cite{CLMP2006}*{Sections 3.1--3.2} obtain related major subsets through
essentially-disjoint and good--bad decompositions. The new two-depth step here is to retain
all near pairs and control their weighted overlap jointly in the two depths; only after
that estimate do we use Lemma \ref{lem:local-overlap-to-marginals} to produce sparse
subfamilies with a quantitative inclusion probability for every rectangle.

The density calculations below repeatedly use the following elementary observation,
formulated for an arbitrary axis-parallel rectangle basis so that it applies in both
settings.

\begin{lem}[Half-slab halo criterion]\label{lem:half-slab-halo}
	Let $\mathcal B$ be a collection of axis-parallel rectangles in $\R^d$, let $G$ be
	measurable, and set
	$\widetilde G_{\mathcal B}:=\{M_{\mathcal B}1_G>1/2\}$.
	Fix
	\[
		Q=Q^1\times Q^2\in\mathcal B,
	\]
	where $Q^1$ and $Q^2$ occupy complementary groups of coordinate variables. Suppose
	that there are measurable sets $E^2\subset Q^2$ and $L^1\subset Q^1$ such that
	\[
		|E^2|\geq\frac12|Q^2|,
		\qquad |L^1|>0,
		\qquad Q^1\times E^2\subset G,
		\qquad L^1\times Q^2\subset G,
	\]
	then $Q\subset\widetilde G_{\mathcal B}$.
	The analogous assertion with superscripts $1$ and $2$ interchanged also holds.
\end{lem}

\begin{proof}
	The sets $Q^1\times E^2$ and $L^1\times(Q^2\setminus E^2)$ are disjoint subsets of
	$G\cap Q$. Hence, writing
	$\theta=|E^2|/|Q^2|\geq1/2$, we have
	\[
		\ave{1_G}_Q
		\geq\theta+(1-\theta)\frac{|L^1|}{|Q^1|}
		>\frac12.
	\]
	Since $Q\in\mathcal B$, every $x\in Q$ satisfies
	$M_{\mathcal B}1_G(x)\geq\ave{1_G}_Q>1/2$.
	The coordinate-reversed assertion follows by the same calculation.
\end{proof}

\begin{exmp}[The same pair can be near or far]\label{ex:planar-near-far}
This is an example of the second case above:
$J^1\subsetneq I^1$ and $J^2\supsetneq I^2$, so the pair is near precisely when
$q\leq e_2(I;\Omega)$. We keep $I$ and $J$ fixed and vary only the ambient set $\Omega$.
Work with the standard dyadic grids on $\R\times\R$ and set
\[
	I=[0,1)\times[0,1),
	\qquad
	J=[0,\tfrac14)\times[0,8).
\]
Here $p=2$, $q=3$, and $\tau(J)=-1<0=\tau(I)$, so $J\prec I$. Write
\[
	Q_k:=I^1\times(I^2)^{(k)}=[0,1)\times[0,2^k).
\]
If $\Omega_{\mathrm{near}}=Q_4$, then
$Q_4\subset\widetilde\Omega_{\mathrm{near}}$ and
\[
	e_2(I;\Omega_{\mathrm{near}})\geq4>q.
\]
Thus $(J,I)$ is near. For the same pair, instead let
$\Omega_{\mathrm{far}}=I\cup J$. Apply Lemma \ref{lem:half-slab-halo} with
$\mathcal B=\calD^1\times\calD^2$, $G=\Omega_{\mathrm{far}}$, $Q=Q_1$,
$E^2=I^2$, and $L^1=J^1$. Indeed, $E^2$ occupies exactly half of
$(I^2)^{(1)}$, while the two required slabs are contained in $I$ and $J$, respectively.
Thus $Q_1\subset\widetilde\Omega_{\mathrm{far}}$. Choose
$x\in(1/2,1)\times(2,4)\subset Q_2$. If a dyadic product rectangle
$R=R^1\times R^2$ through $x$ meets $I$, then dyadic nesting forces
$R^2\supset[0,4)$. If $R$ meets $J$, then dyadic nesting forces
$R^1\supset[0,1)$. These containments give the coordinatewise estimates
\[
	\begin{aligned}
		\frac{|R\cap I|}{|R|}
		&=\frac{|R^1\cap I^1|}{|R^1|}
		  \frac{|R^2\cap I^2|}{|R^2|}
		\leq1\cdot\frac14=\frac14,\\
		\frac{|R\cap J|}{|R|}
		&=\frac{|R^1\cap J^1|}{|R^1|}
		  \frac{|R^2\cap J^2|}{|R^2|}
		\leq\frac14\cdot1=\frac14.
	\end{aligned}
\]
Consequently,
\[
	\ave{1_{\Omega_{\mathrm{far}}}}_R
	\leq\frac{|R\cap I|}{|R|}+\frac{|R\cap J|}{|R|}
	\leq\frac12.
\]
Taking the supremum over $R$ gives $x\notin\widetilde\Omega_{\mathrm{far}}$. Thus
$Q_2\not\subset\widetilde\Omega_{\mathrm{far}}$, and the nesting of the $Q_k$ gives
\[
e_2(I;\Omega_{\mathrm{far}})=1<q.
\]
Thus $(J,I)$ is far. We also note that both ambient sets can be realized as shadows of
pairwise incomparable dyadic families containing $I$ and $J$. Indeed,
$\Omega_{\mathrm{far}}=\operatorname{sh}(\{I,J\})$, while
$Q_4\setminus(I\cup J)$ has a finite partition into dyadic rectangles. Adding those
rectangles to $\{I,J\}$ gives a pairwise incomparable family with shadow
$Q_4=\Omega_{\mathrm{near}}$. Figure \ref{fig:near-far-planar} shows the two ambient sets.
Only the ambient sets and the relevant dyadic rectangles are drawn; the halos are not
shaded.
\end{exmp}

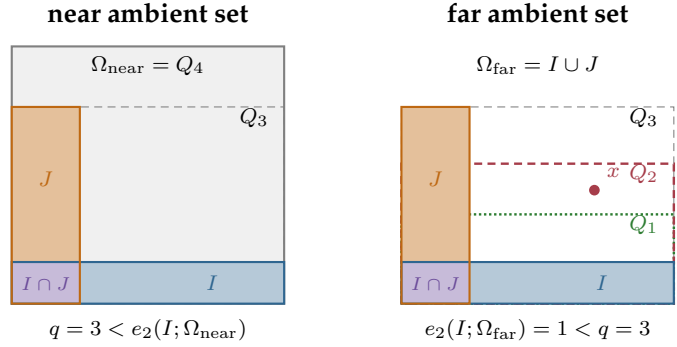
\begin{figure}[H]
	\centering
	\begin{tikzpicture}[x=1cm,y=1cm,font=\small]
		\begin{scope}
			\node[font=\small\bfseries] at (1.8,3.85) {near ambient set};
			\fill[black!6] (0,0) rectangle (3.6,3.4);
			\draw[gray,line width=.8pt] (0,0) rectangle (3.6,3.4);
			\draw[gray,densely dashed] (0,0) rectangle (3.6,2.6);
			\node[font=\scriptsize,anchor=east] at (3.53,2.46) {$Q_3$};
			\fill[zygblue!35] (0,0) rectangle (3.6,.55);
			\fill[zygorange!42] (0,0) rectangle (.9,2.6);
			\fill[zygviolet!38] (0,0) rectangle (.9,.55);
			\draw[zygblue,line width=.8pt] (0,0) rectangle (3.6,.55);
			\draw[zygorange,line width=.8pt] (0,0) rectangle (.9,2.6);
			\node[text=zygblue,font=\scriptsize\bfseries] at (2.65,.27) {$I$};
			\node[text=zygorange,font=\scriptsize\bfseries] at (.45,1.65) {$J$};
			\node[text=zygviolet,font=\tiny\bfseries] at (.45,.27) {$I\cap J$};
			\node[font=\scriptsize] at (1.8,3.15) {$\Omega_{\mathrm{near}}=Q_4$};
			\node[font=\scriptsize] at (1.8,-.35)
				{$q=3<e_2(I;\Omega_{\mathrm{near}})$};
		\end{scope}

		\begin{scope}[shift={(5.15,0)}]
			\node[font=\small\bfseries] at (1.8,3.85) {far ambient set};
			\draw[gray,densely dashed] (0,0) rectangle (3.6,2.6);
			\node[font=\scriptsize,anchor=east] at (3.53,2.46) {$Q_3$};
			\draw[zyggreen,densely dotted,line width=.9pt] (0,0) rectangle (3.6,1.18);
			\node[text=zyggreen,font=\scriptsize,anchor=east] at (3.53,1.04) {$Q_1$};
			\draw[zygred,densely dashed,line width=.9pt] (0,0) rectangle (3.6,1.85);
			\node[text=zygred,font=\scriptsize,anchor=east] at (3.53,1.71) {$Q_2$};
			\fill[zygblue!35] (0,0) rectangle (3.6,.55);
			\fill[zygorange!42] (0,0) rectangle (.9,2.6);
			\fill[zygviolet!38] (0,0) rectangle (.9,.55);
			\draw[zygblue,line width=.8pt] (0,0) rectangle (3.6,.55);
			\draw[zygorange,line width=.8pt] (0,0) rectangle (.9,2.6);
			\node[text=zygblue,font=\scriptsize\bfseries] at (2.65,.27) {$I$};
			\node[text=zygorange,font=\scriptsize\bfseries] at (.45,1.65) {$J$};
			\node[text=zygviolet,font=\tiny\bfseries] at (.45,.27) {$I\cap J$};
			\fill[zygred] (2.55,1.5) circle (2pt);
			\node[text=zygred,font=\scriptsize,above right=1pt] at (2.55,1.5) {$x$};
			\node[font=\scriptsize] at (1.8,3.15) {$\Omega_{\mathrm{far}}=I\cup J$};
			\node[font=\scriptsize] at (1.8,-.35)
				{$e_2(I;\Omega_{\mathrm{far}})=1<q=3$};
		\end{scope}
	\end{tikzpicture}
	\caption{The same incomparable pair is near for $\Omega_{\mathrm{near}}$ and far for
	$\Omega_{\mathrm{far}}$; see Example~\ref{ex:planar-near-far}. The vertical scale records
	dyadic generations schematically; in the actual geometry, $I=[0,1)^2$ is a square.}
	\label{fig:near-far-planar}
\end{figure}

\subsection{Far predecessors}

For a fixed later rectangle $I$, define
\[
	\mathcal F_1(I)
	:=\bigcup_{\substack{J\prec I:\ \eqref{eq:planar-second-intersection}\text{ holds}\\
		(J,I)\text{ far}}}J^1\subset I^1,
	\qquad
	\mathcal F_2(I)
	:=\bigcup_{\substack{J\prec I:\ \eqref{eq:planar-first-intersection}\text{ holds}\\
		(J,I)\text{ far}}}J^2\subset I^2.
\]
These sets are measurable because $\mathcal U$ is countable.

\begin{lem}[Bi-parameter far predecessors leave a major subset]\label{lem:planar-far}
	For every $I\in\mathcal U$,
	\[
		|\mathcal F_i(I)|<\frac12|I^i|,
		\qquad i=1,2.
	\]
	Consequently,
	\begin{equation*}
		A_I
		:=\big(I^1\setminus\mathcal F_1(I)\big)
		\times\big(I^2\setminus\mathcal F_2(I)\big)
	\end{equation*}
	has measure greater than $|I|/4$ and avoids every far predecessor of $I$.
\end{lem}

\begin{proof}
	We prove the first claim for $\mathcal F_2(I)$. Put
	$k=e_1(I;\Omega)+1$ and
	\[
		Q=(I^1)^{(k)}\times I^2.
	\]
	For every far predecessor in \eqref{eq:planar-first-intersection}, the first-coordinate
	gap is at least $k$, so $J^1\supset(I^1)^{(k)}$. If
	$y_2\in\mathcal F_2(I)$, choose such a predecessor $J$ with $y_2\in J^2$.
	Since $(I^1)^{(k)}\subset J^1$, every point of
	$(I^1)^{(k)}\times\mathcal F_2(I)$ belongs to one of these predecessors. Hence
	\[
		(I^1)^{(k)}\times\mathcal F_2(I)\subset\Omega.
	\]
	If $|\mathcal F_2(I)|\geq|I^2|/2$, apply Lemma \ref{lem:half-slab-halo} with
	$\mathcal B=\calD^1\times\calD^2$, $G=\Omega$, the rectangle $Q$ above,
	$E^2=\mathcal F_2(I)$, and $L^1=I^1$.
	The lemma gives $Q\subset\widetilde\Omega$. This contradicts the definition of
	$e_1(I;\Omega)$, and therefore
	$|\mathcal F_2(I)|<|I^2|/2$. The
	proof for $\mathcal F_1(I)$ is identical with
	$Q=I^1\times(I^2)^{(e_2(I;\Omega)+1)}$. Moreover,
	$|A_I|=|I^1\setminus\mathcal F_1(I)|\,
	|I^2\setminus\mathcal F_2(I)|>|I|/4$ by the two estimates just proved.
	Let $J\prec I$ be a far predecessor. If \eqref{eq:planar-first-intersection}
	holds, then $J^2\subset\mathcal F_2(I)$, whereas the second-coordinate factor
	of $A_I$ is $I^2\setminus\mathcal F_2(I)$; hence $A_I\cap J=\varnothing$.
	The case in which \eqref{eq:planar-second-intersection} holds is symmetric.
\end{proof}

\subsection{Near successors}\label{sec:planar-near-successors}

Lemma \ref{lem:planar-far} verifies hypothesis (i) of
Lemma \ref{lem:local-overlap-to-marginals} with $\eta=1/4$. To apply that lemma, it
remains to prove the local weighted-overlap estimate for the later near successors of each
fixed earlier rectangle. Lemma \ref{lem:fixed-total-packing} was not used in the
proof of the abstract selection lemma; its role begins here. For the rest of this proof,
write $W_w(R)$ for $W_w(R;\Omega)$ whenever $R\in\mathcal U$.

\begin{lem}[Bi-parameter local near-successor estimate]\label{lem:planar-near}
	For every nonzero admissible weight $w$ and every
	$J\in\mathcal U$,
	\begin{equation}\label{eq:planar-local-near-overlap}
		\sum_{\substack{I:\,J\prec I\\(J,I)\text{ near}}}
		W_w(I)|J\cap I|\leq4\|w\|_{\mathrm{ft}}|J|.
	\end{equation}
\end{lem}

\begin{proof}
Fix the earlier rectangle $J$. The sum in \eqref{eq:planar-local-near-overlap} splits
according to the two coordinate configurations \eqref{eq:planar-first-intersection} and
\eqref{eq:planar-second-intersection}. We bound each part by
$2\|w\|_{\mathrm{ft}}|J|$, using Lemma
\ref{lem:fixed-total-packing} once for each configuration.

All notation that follows is local to one fixed earlier rectangle $J$ and one of the two
coordinate configurations. To state the suppressed dependence precisely, temporarily label
the configurations by $r\in\{1,2\}$, with $r=1$ corresponding to
\eqref{eq:planar-first-intersection} and $r=2$ to
\eqref{eq:planar-second-intersection}. Let $\mathcal N_r(J)$ denote the later near
successors of $J$ in configuration $r$. Fully explicit notation would then be
\[
	(P_I^{(r)})_{I\in\mathcal N_r(J)},\qquad
	p_J^{(r)}(I),\qquad q_J^{(r)}(I),\qquad
	a_J^{(r)}(I),\qquad b^{(r)}(I).
\]
The individual cube $P_I^{(r)}$ and depth $b^{(r)}(I)$ depend on $I$ and the configuration,
but not on $J$. The fixed rectangle $J$ determines which successors $I$ belong to the
indexed family, and the generation gaps $p_J^{(r)}(I),q_J^{(r)}(I)$ and
$a_J^{(r)}(I)$ also depend on $J$. We suppress $J$ and $r$ throughout, writing simply
$(P_I)_I,p(I),q(I),a(I),b(I)$. This notation is reset when either the fixed rectangle or
the coordinate configuration changes.

For $I\in\mathcal N_1(J)$, the definitions for Lemma
\ref{lem:fixed-total-packing} and the resulting geometric identities are
\[
\begin{gathered}
	p(I):=\nu_1(I)-\nu_1(J),\qquad
	q(I):=\nu_2(J)-\nu_2(I),\qquad
	P_I:=I^1\subset J^1,\\
	a(I):=q(I),\qquad b(I):=e_2(I;\Omega),\qquad
	I^2=(J^2)^{(q(I))},\qquad J\cap I=P_I\times J^2,\\
	p(I)\geq q(I)\geq1,\qquad e_1(I;\Omega)\geq p(I).
\end{gathered}
\]
For $I\in\mathcal N_2(J)$, they are
\[
\begin{gathered}
	p(I):=\nu_1(J)-\nu_1(I),\qquad
	q(I):=\nu_2(I)-\nu_2(J),\qquad
	P_I:=I^2\subset J^2,\\
	a(I):=p(I),\qquad b(I):=e_1(I;\Omega),\qquad
	I^1=(J^1)^{(p(I))},\qquad J\cap I=J^1\times P_I,\\
	q(I)\geq p(I)\geq1,\qquad e_2(I;\Omega)\geq q(I).
\end{gathered}
\]
In the last line of each display, the comparison between $p(I)$ and $q(I)$ follows
from $\tau(J)\leq\tau(I)$, the lower bound $1$ follows from strict dyadic containment,
and the embeddedness inequality is exactly the near condition.
The quantities $p(I)$ and $q(I)$ are the gaps denoted by $p$ and $q$ in the ordering
subsection, now written as functions of the varying successor $I$. In the first
configuration, nearness uses the first-coordinate gap and depth, while $a(I)$ and
$b(I)$ are the second-coordinate gap and depth. In the second configuration, nearness
uses the second-coordinate quantities, while $a(I)$ and $b(I)$ come from the first
coordinate. Thus $a(I)$ and $b(I)$ belong to the coordinate complementary to the one
used in the near condition: $a(I)$ measures the enlargement from $J$ to $I$ in that
other coordinate, whereas $b(I)$ is the embeddedness depth of the later rectangle $I$
there.

We apply Lemma \ref{lem:fixed-total-packing} separately to
$\mathcal N_1(J)$ and $\mathcal N_2(J)$, so no comparison between the two families is
needed. Within either family, $K$ denotes another successor of the same fixed $J$. For
a fixed $I$, only the two comparisons in the geometric hypotheses of Lemma
\ref{lem:exclusion-packing} matter: $a(K)=a(I)$ in hypothesis (i), and
$a(K)>a(I)+b(I)$ in hypothesis (ii).

We begin by running the argument in $\mathcal N_1(J)$. Recalling that
$a(I)=q(I)$ in this configuration, we can already verify hypothesis (i). If
$I\neq K$ and $a(I)=a(K)$, then $q(I)=q(K)$ and hence
\[
	I^2=K^2=(J^2)^{(q(I))}.
\]
If $P_I$ and $P_K$ met, dyadic nesting together with the displayed common coordinate
cube would make one full rectangle contain the other. Pairwise incomparability therefore
gives $P_I\cap P_K=\varnothing$.

This disjointness controls one value of $q$ at a time, but it does not control overlap
between different values of $q$. Hypothesis (ii) supplies the cross-level control required
by the fixed-total packing criterion. For a fixed
$I$ and every integer $h\geq1$, we will prove
\[
	\left|P_I\cap\bigcup_{a(K)\geq a(I)+h}P_K\right|
	\geq\frac12|P_I|
	\quad\Longrightarrow\quad b(I)\geq h.
\]
Taking $h=b(I)+1$ and using that $a$ is integer-valued gives hypothesis (ii), because
$a(K)\geq a(I)+b(I)+1$ is equivalent to $a(K)>a(I)+b(I)$. Lemma
\ref{lem:fixed-total-packing} will then bound the inner sum for this configuration.

We now prove the implication needed for hypothesis (ii). For $h\geq1$,
\begin{equation}\label{eq:planar-first-density-implication}
		\left|P_I\cap
		\bigcup_{\substack{K:q(K)\geq q(I)+h}}P_K\right|
		\geq\frac12|P_I|
		\quad\Longrightarrow\quad e_2(I;\Omega)\geq h.
\end{equation}
This is the half-slab density calculation in Lemma \ref{lem:half-slab-halo}.
Indeed, $K^2=(J^2)^{(q(K))}$ and $I^2=(J^2)^{(q(I))}$, so
$q(K)\geq q(I)+h$ implies $K^2\supset(I^2)^{(h)}$. Set
\[
	A:=P_I\cap\bigcup_{q(K)\geq q(I)+h}P_K,
	\qquad Q:=I^1\times(I^2)^{(h)}.
\]
If $x_1\in A$, choose $K$ with
$x_1\in P_K=K^1$. For every $x_2\in(I^2)^{(h)}$, the preceding containment gives
$(x_1,x_2)\in K\subset\Omega$. Thus
	\[
		A\times(I^2)^{(h)}\subset\Omega.
	\]
	Since $P_I=I^1$, the hypothesis in
	\eqref{eq:planar-first-density-implication} gives $|A|\geq|I^1|/2$.
	Apply the coordinate-reversed form of Lemma \ref{lem:half-slab-halo} with
	$\mathcal B=\calD^1\times\calD^2$, $G=\Omega$, the rectangle $Q$ above,
	$E^1=A$, and $L^2=I^2$. It gives $Q\subset\widetilde\Omega$, and therefore
	$e_2(I;\Omega)\geq h$.

Recalling that $b(I)=e_2(I;\Omega)$, the reduction above shows that
\eqref{eq:planar-first-density-implication} verifies hypothesis (ii). We now apply
Lemma \ref{lem:fixed-total-packing}, with weight $w$, to
\[
	X=J^1,\qquad \mathcal A=\mathcal N_1(J),\qquad
	P_I=I^1,\qquad a(I)=q(I),\qquad b(I)=e_2(I;\Omega).
\]
The index set $\mathcal N_1(J)$ is countable, and
$P_I\subset J^1=X$ by the data recorded above. Hypotheses (i) and (ii) have just been
verified. Therefore Lemma \ref{lem:fixed-total-packing} gives
	\[
		\sum_{I\in\mathcal N_1(J)}
		w(q(I),e_2(I;\Omega))|P_I|
		\leq2\|w\|_{\mathrm{ft}}|J^1|.
	\]
Using the inequalities recorded in the definition of $\mathcal N_1(J)$,
\[
	W_w(I)
	=w(e_1(I;\Omega),e_2(I;\Omega))
	\leq w(q(I),e_2(I;\Omega)).
\]
Combining this pointwise bound with the preceding packing estimate and the identity
$|J\cap I|=|P_I||J^2|$ gives
	\begin{equation}\label{eq:planar-first-near-successors}
	\begin{aligned}
		\sum_{\substack{I:J\prec I,\ J^1\supsetneq I^1\\(J,I)\text{ near}}}
		W_w(I)|J\cap I|
		&=|J^2|\sum_{I\in\mathcal N_1(J)}W_w(I)|P_I|\\
		&\leq |J^2|\sum_{I\in\mathcal N_1(J)}
		w(q(I),e_2(I;\Omega))|P_I|\\
		&\leq2\|w\|_{\mathrm{ft}}|J^2||J^1|=2\|w\|_{\mathrm{ft}}|J|.
	\end{aligned}
	\end{equation}

The configuration $\mathcal N_2(J)$ is symmetric: exchange the first and second
coordinates in the preceding argument. Apply Lemma
\ref{lem:fixed-total-packing} with weight $w^\top$,
$X=J^2$, $\mathcal A=\mathcal N_2(J)$, $P_I=I^2$, $a(I)=p(I)$ and
$b(I)=e_1(I;\Omega)$. Since $e_2(I;\Omega)\geq q(I)\geq p(I)$,
\[
	W_w(I)\leq w(e_1(I;\Omega),p(I))
	=w^\top(p(I),e_1(I;\Omega)).
\]
Since $\|w^\top\|_{\mathrm{ft}}=\|w\|_{\mathrm{ft}}$, Lemma
\ref{lem:fixed-total-packing} therefore gives
	\begin{equation}\label{eq:planar-second-near-successors}
		\sum_{\substack{I:J\prec I,\ J^1\subsetneq I^1\\(J,I)\text{ near}}}
		W_w(I)|J\cap I|\leq2\|w\|_{\mathrm{ft}}|J|.
	\end{equation}

Adding \eqref{eq:planar-first-near-successors} and
\eqref{eq:planar-second-near-successors} proves
\eqref{eq:planar-local-near-overlap}.
\end{proof}

\subsection{Completion of the two-depth packing proof}\label{sec:planar-completion}

\begin{proof}[Proof of Theorem \ref{thm:planar-admissible-weight}]
	Lemma \ref{lem:planar-far} verifies hypothesis (i) of Lemma
	\ref{lem:local-overlap-to-marginals} with $\eta=1/4$, while Lemma
	\ref{lem:planar-near} verifies hypothesis (ii) for $W_w$ with constant $4\|w\|_{\mathrm{ft}}$.
	Since $w$ is nonzero and nonnegative, $\|w\|_{\mathrm{ft}}>0$. Moreover, every
	$w(a,b)$ is one term in the fixed-total sum with $N=a+b$, and hence
	$0\leq \|w\|_{\mathrm{ft}}^{-1}W_w(I;\Omega)\leq1$. Apply Lemma
	\ref{lem:local-overlap-to-marginals} to this normalized weight. Its local near-overlap
	constant is $4$, so its inclusion-probability conclusion becomes
	\[
		\mathbb P\bigl(\{\theta\in\Theta:I\in\mathcal G(\theta)\}\bigr)
		\geq\frac{\|w\|_{\mathrm{ft}}^{-1}W_w(I;\Omega)}{32(4+1)}
		=\kappa_wW_w(I;\Omega).
	\]
	This is \eqref{eq:planar-general-marginals}.

	Since $\operatorname{sh}(\mathcal U)\subset\Omega$, Remark
	\ref{rem:marginals-to-scalar}, applied with this lower bound, directly gives a
	$1/8$-sparse subfamily $\mathcal G_0\subset\mathcal U$ such that
	\[
		\sum_{I\in\mathcal U}W_w(I;\Omega)|I|
		\leq\frac{16}{\kappa_w}|\operatorname{sh}(\mathcal G_0)|
		=2560\|w\|_{\mathrm{ft}}|\operatorname{sh}(\mathcal G_0)|,
	\]
	which completes \eqref{eq:planar-general-sparse-packing}.
\end{proof}

\subsection{Exponential overlap of sparse incomparable families}

The following is the form of Rey's exponential-overlap theorem needed here. Rey uses the
term \emph{antichain} for a pairwise incomparable family and proves the result for planar
dyadic rectangles; see \cite{Rey2026antichain}*{Theorem A and Section 3.1}. We include a
concise proof for the reader's convenience. It records the explicit dependence on the
sparseness parameter and the numerical constants used below.

\begin{lem}[Exponential overlap of a sparse incomparable family]
	\label{lem:incomparable-exponential-overlap}
	Let $\calD^1$ and $\calD^2$ be dyadic grids in $\R^{d_1}$ and $\R^{d_2}$,
	respectively. Let $\mathcal G\subset\calD^1\times\calD^2$ be a pairwise
	incomparable, $\eta$-sparse family, where $0<\eta\leq1$. Put
	\[
		h_{\mathcal G}(x):=\sum_{I\in\mathcal G}1_I(x),
		\qquad
		t_\eta:=\log\left(1+\frac\eta{32}\right).
	\]
	Then
	\begin{equation}\label{eq:incomparable-exponential-excess}
		\int_{\R^{d_1+d_2}}\bigl(e^{t_\eta h_{\mathcal G}(x)}-1\bigr)\,\ud x
		\leq\frac1{16}|\operatorname{sh}(\mathcal G)|.
	\end{equation}
\end{lem}

\begin{proof}
	If $|\operatorname{sh}(\mathcal G)|=\infty$, there is nothing to prove. Thus we may
	assume that the shadow has finite measure.
	By passing to increasing finite subfamilies and using the Monotone Convergence Theorem,
	it is enough to consider a finite family $\mathcal G$. Choose pairwise disjoint measurable
	sets $E_I\subset I$ such that $|E_I|\geq\eta|I|$, and put
	$\varphi_I:=|I|^{-1/2}1_I$. For $I,J\in\mathcal G$, define their normalized overlap by
	\begin{equation*}
		\rho(I,J):=\int\varphi_I\varphi_J
		=\frac{|I\cap J|}{|I|^{1/2}|J|^{1/2}}.
	\end{equation*}
	The disjointness of the sets $E_I$ and the standard $L^2$ estimate
	$\|M_{\calD^1\times\calD^2}\|_{2\to2}\leq4$ give
	\begin{align}
		\sum_{I\in\mathcal G}|\langle f,\varphi_I\rangle|^2
		&\leq\frac1\eta\sum_{I\in\mathcal G}
		\int_{E_I}\bigl(M_{\calD^1\times\calD^2}f\bigr)^2
		\leq\frac{16}{\eta}\|f\|_2^2.
		\label{eq:normalized-indicator-l2}
	\end{align}
	For every choice of scalars $(c_J)_{J\in\mathcal G}$, an immediate duality consequence
	of \eqref{eq:normalized-indicator-l2} is
	\begin{equation}\label{eq:normalized-overlap-l2}
		c'_I:=\sum_{J\in\mathcal G}\rho(I,J)c_J,
		\qquad
		\sum_{I\in\mathcal G}|c'_I|^2
		\leq\left(\frac{16}{\eta}\right)^2
		\sum_{J\in\mathcal G}|c_J|^2.
	\end{equation}

	Now let $I_1,\ldots,I_k$ be distinct members of $\mathcal G$ such that
	$\bigcap_{j=1}^k I_j\neq\emptyset$. Dyadic nesting and incomparability show that they can
	be reordered uniquely so that
	\[
		I_1^1\subsetneq I_2^1\subsetneq\cdots\subsetneq I_k^1
		\qquad\text{and}\qquad
		I_1^2\supsetneq I_2^2\supsetneq\cdots\supsetneq I_k^2.
	\]
	Thus $\bigcap_{j=1}^kI_j=I_1^1\times I_k^2$, while
	\[
		\rho(I_j,I_{j+1})
		=\left(\frac{|I_j^1|}{|I_{j+1}^1|}\right)^{1/2}
		\left(\frac{|I_{j+1}^2|}{|I_j^2|}\right)^{1/2}.
	\]
	The two products therefore telescope to give
	\begin{equation}
		\label{eq:multiple-intersection-product}
		|I_1|^{1/2}|I_k|^{1/2}\prod_{j=1}^{k-1}\rho(I_j,I_{j+1})
		=|I_1^1||I_k^2|=\left|\bigcap_{j=1}^kI_j\right|.
	\end{equation}
	For each $x$, the quantity $\binom{h_{\mathcal G}(x)}k$ counts the $k$-element
	subfamilies of $\mathcal G$ that contain $x$. Integrating, applying
	\eqref{eq:multiple-intersection-product} to the unique ordering of every intersecting
	subfamily, and enlarging the resulting sum to all ordered $k$-tuples gives
	\begin{equation}\label{eq:all-overlap-strings}
		\int_{\R^{d_1+d_2}}\binom{h_{\mathcal G}(x)}k\,\ud x
		\leq\sum_{I_1,\ldots,I_k\in\mathcal G}
		|I_1|^{1/2}|I_k|^{1/2}\prod_{j=1}^{k-1}\rho(I_j,I_{j+1}).
	\end{equation}
	Put $v_I^{(0)}:=|I|^{1/2}$ and, recursively,
	\begin{equation*}
		v_I^{(m+1)}:=\sum_{J\in\mathcal G}\rho(I,J)v_J^{(m)},
		\qquad m\geq0.
	\end{equation*}
	The sum on the right of \eqref{eq:all-overlap-strings} equals
	$\sum_Iv_I^{(0)}v_I^{(k-1)}$. Thus Cauchy--Schwarz, repeated use of
	\eqref{eq:normalized-overlap-l2}, and sparseness yield
	\begin{equation}\label{eq:binomial-moment-bound}
		\int_{\R^{d_1+d_2}}\binom{h_{\mathcal G}(x)}k\,\ud x
		\leq\left(\frac{16}{\eta}\right)^{k-1}\sum_{I\in\mathcal G}|I|
		\leq\frac1\eta\left(\frac{16}{\eta}\right)^{k-1}
		|\operatorname{sh}(\mathcal G)|,
		\qquad k\geq1.
	\end{equation}

	Finally, put $q:=e^t-1$. The binomial identity and
	\eqref{eq:binomial-moment-bound} give, whenever $16q/\eta<1$,
	\begin{align*}
		\int_{\R^{d_1+d_2}}\bigl(e^{t h_{\mathcal G}(x)}-1\bigr)\,\ud x
		&=\sum_{k=1}^{\infty}q^k
		\int_{\R^{d_1+d_2}}\binom{h_{\mathcal G}(x)}k\,\ud x\\
		&\leq\frac q\eta\sum_{k=1}^{\infty}
		\left(\frac{16q}{\eta}\right)^{k-1}
		|\operatorname{sh}(\mathcal G)|
		=\frac{q/\eta}{1-16q/\eta}|\operatorname{sh}(\mathcal G)|.
	\end{align*}
	For $t=t_\eta$, we have $q=\eta/32$, and the last coefficient is $1/16$.
	This proves \eqref{eq:incomparable-exponential-excess}.
\end{proof}

\subsection{Sparse exponential comparison}
\label{sec:planar-exponential-comparison}

\begin{thm}[Sparse exponential comparison for an admissible weight]
\label{thm:planar-admissible-exponential}
Let $(\Omega,\mathcal U)$ be an admissible pair in the bi-parameter setting, let $w$ be a
nonzero admissible weight, and put
\[
	t_0:=\log(1+1/256).
\]
Then there is a $1/8$-sparse subfamily $\mathcal G_0\subset\mathcal U$ such that
\begin{equation}\label{eq:planar-general-sparse-exponential}
	\int_{\R^{d_1+d_2}}\bigl(e^{t_0\kappa_wH_w(x)}-1\bigr)\,\ud x
	\leq\int_{\R^{d_1+d_2}}
	\bigl(e^{t_0h_{\mathcal G_0}(x)}-1\bigr)\,\ud x
	\leq\frac1{16}|\operatorname{sh}(\mathcal G_0)|.
\end{equation}
Consequently,
\begin{equation}\label{eq:planar-general-exponential}
	\int_\Omega e^{t_0\kappa_wH_w(x)}\,\ud x\leq\frac{17}{16}|\Omega|.
\end{equation}
\end{thm}

\begin{proof}
	The proof of Theorem \ref{thm:planar-admissible-weight} produces a probability space
	$(\Theta,\mathcal A,\mathbb P)$ and $1/8$-sparse subfamilies
	$\mathcal G(\theta)\subset\mathcal U$ satisfying
	\eqref{eq:planar-general-marginals}. Abbreviate
	$h_\theta:=h_{\mathcal G(\theta)}$. The exponential-overlap lemma concerns
	$h_\theta$, not $H_w$: the weighted height enters only through Lemma
	\ref{lem:inclusion-probabilities-to-convex-overlap}.

	For every $\theta\in\Theta$, the family $\mathcal G(\theta)$ is $1/8$-sparse and remains
	pairwise incomparable because it is a subfamily of $\mathcal U$. Moreover,
	$\operatorname{sh}(\mathcal G(\theta))\subset\Omega$. Lemma
	\ref{lem:incomparable-exponential-overlap}, applied with $\eta=1/8$, therefore gives
	\[
		\int_{\R^{d_1+d_2}}\bigl(e^{t_0h_\theta(x)}-1\bigr)\,\ud x
		\leq\frac1{16}|\operatorname{sh}(\mathcal G(\theta))|
		\leq\frac1{16}|\Omega|.
	\]

	The uniform estimate shows that the averaged integral in
	\eqref{eq:abstract-convex-transfer} is finite for $\Phi(u)=e^{t_0u}-1$. Lemma
	\ref{lem:inclusion-probabilities-to-convex-overlap}, applied with
	$w(I)=W_w(I;\Omega)$ and $\kappa=\kappa_w$, provides $\theta_0\in\Theta$ such that,
	with $\mathcal G_0:=\mathcal G(\theta_0)$,
	\[
		\int\bigl(e^{t_0\kappa_wH_w}-1\bigr)\,\ud x
		\leq\int\bigl(e^{t_0h_{\mathcal G_0}}-1\bigr)\,\ud x
		\leq\frac1{16}|\operatorname{sh}(\mathcal G_0)|.
	\]
	This is \eqref{eq:planar-general-sparse-exponential}. Since $H_w=0$ outside
	$\operatorname{sh}(\mathcal U)\subset\Omega$, adding $|\Omega|$ gives
	\eqref{eq:planar-general-exponential}.
\end{proof}

\section{Two-depth packing in the Zygmund setting}\label{sec:critical-proof}

We first transfer the layered family from Proposition \ref{prop:layered-test-family} to
the Zygmund boundary, which proves that the fixed-total condition is necessary here as
well. We then prove the corresponding sufficiency, marginal and sparse conclusions of
Theorem \ref{thm:intro-zygmund-admissible}. The sufficiency proof follows the bi-parameter argument from
Section \ref{sec:planar-critical}: embeddedness removes far predecessors, Lemma
\ref{lem:fixed-total-packing} controls the near successors, and Lemma
\ref{lem:local-overlap-to-marginals} produces random sparse subfamilies. After completing
this argument, we return to the marginal probabilities and combine them with the Zygmund
maximal theorem to obtain a sparse square-root exponential comparison.

The key new point is geometric. Because the third-coordinate cubes may vary, Lemma
\ref{lem:fixed-total-packing} is now applied to the lower-dimensional dyadic rectangles
$I^1\times I^3$ and $I^2\times I^3$, rather than to individual coordinate cubes.
Subsection \ref{sec:zygmund-full-exponential} later recovers full exponential
integrability under the stronger diagonal-envelope summability condition.

\subsection{The layered test family on the Zygmund boundary}

The lift rests on the following elementary product-slice identity: on every slice
$x_3\in K^3$, the sub-Zygmund maximal function of $1_{\Omega\times K^3}$ agrees with
the bi-parameter maximal function of $1_\Omega$.

\begin{lem}[A product-slice identity]\label{lem:product-slice}
	Let $\Omega\subset\R^{d_1+d_2}$ be measurable and have finite measure, and let
	$K^3\in\calD^3$. For every $(x_1,x_2)\in\R^{d_1+d_2}$ and $x_3\in K^3$,
	\[
		M_{\calD_{\mathrm{sZ}}}1_{\Omega\times K^3}(x_1,x_2,x_3)
		=M_{\calD^1\times\calD^2}1_\Omega(x_1,x_2).
	\]
\end{lem}

\begin{proof}
	Fix $x_3\in K^3$. For every
	$Q=Q^1\times Q^2\times Q^3\in\calD_{\mathrm{sZ}}$ containing
	$(x_1,x_2,x_3)$,
	\[
		\ave{1_{\Omega\times K^3}}_Q
		=\ave{1_\Omega}_{Q^1\times Q^2}\ave{1_{K^3}}_{Q^3}
		\leq\ave{1_\Omega}_{Q^1\times Q^2}
		\leq M_{\calD^1\times\calD^2}1_\Omega(x_1,x_2).
	\]
	Taking the supremum over $Q$ gives one inequality. Conversely, for every bi-parameter
	dyadic rectangle $Q^1\times Q^2$ containing $(x_1,x_2)$, choose a sufficiently small
	dyadic cube $Q^3\subset K^3$ through $x_3$ so that
	$\ell(Q^3)\leq\ell(Q^1)\ell(Q^2)$. Then
	$Q^1\times Q^2\times Q^3\in\calD_{\mathrm{sZ}}$, and
	\[
		\ave{1_{\Omega\times K^3}}_{Q^1\times Q^2\times Q^3}
		=\ave{1_\Omega}_{Q^1\times Q^2}.
	\]
	Taking the supremum over $Q^1\times Q^2$ proves the reverse inequality.
\end{proof}

As in Proposition \ref{prop:layered-test-family}, we write the lifted example with
one-dimensional coordinate factors; the same construction works in arbitrary fixed
dimensions with dyadic cubes.

\begin{prop}[The layered Zygmund test family]
\label{prop:zygmund-layered-test-family}
Work with the standard dyadic grid on $\R^3$. For $N\geq0$, let $\Omega=[0,1)^2$ and
let $\mathcal U_N$ be the bi-parameter family in Proposition
\ref{prop:layered-test-family}. Put
\[
	K_N^3:=[0,2^{-N}),
	\qquad
	\Omega_N^Z:=\Omega\times K_N^3,
	\qquad
	\mathcal U_N^Z:=\{I\times K_N^3:I\in\mathcal U_N\}.
\]
Then $\mathcal U_N^Z\subset\calD_Z$ is finite, pairwise incomparable and contained in
$\Omega_N^Z$. Moreover, for every $w\colon\N^2\to[0,\infty)$,
\begin{equation}\label{eq:zygmund-layered-height-identity}
	H^Z_{w,N}(x)
	:=\sum_{J\in\mathcal U_N^Z}
	w(e_1(J;\Omega_N^Z),e_2(J;\Omega_N^Z))1_J(x)
	=\sum_{a=0}^Nw(a,N-a),
	\qquad x\in\Omega_N^Z,
\end{equation}
and hence
\begin{equation}\label{eq:zygmund-layered-mass-identity}
	\frac1{|\Omega_N^Z|}\sum_{J\in\mathcal U_N^Z}
	w(e_1(J;\Omega_N^Z),e_2(J;\Omega_N^Z))|J|
	=\sum_{a=0}^Nw(a,N-a).
\end{equation}
Consequently, the constant $C_w$ in any uniform Zygmund two-depth packing estimate with
weight $w$ satisfies $\|w\|_{\mathrm{ft}}\leq C_w$.
\end{prop}

\begin{proof}
	If $I\times K_N^3\subset I'\times K_N^3$ for $I,I'\in\mathcal U_N$, then
	$I\subset I'$. Thus $\mathcal U_N^Z$ is pairwise incomparable because
	$\mathcal U_N$ is. Moreover, if $I$ belongs to the $a$th layer, then
	\[
		\ell(I^1)\ell(I^2)
		=2^{-N}
		=\ell(K_N^3),
	\]
	so $I\times K_N^3\in\calD_Z$.

	Let $J=I\times K_N^3\in\mathcal U_N^Z$. Here
	\[
		\widetilde{\Omega_N^Z}
		:=\{M_{\calD_{\mathrm{sZ}}}1_{\Omega_N^Z}>1/2\}
		\quad\text{and}\quad
		\widetilde\Omega
		:=\{M_{\calD^1\times\calD^2}1_\Omega>1/2\}
	\]
	are, respectively, the sub-Zygmund and bi-parameter halos. For $i=1,2$, Lemma
	\ref{lem:product-slice} gives
	\[
		J_i^{(k)}=I_i^{(k)}\times K_N^3,
		\qquad
		J_i^{(k)}\subset\widetilde{\Omega_N^Z}
		\quad\Longleftrightarrow\quad
		I_i^{(k)}\subset\widetilde\Omega.
	\]
	Taking the largest admissible $k$ yields
	\[
		e_i(J;\Omega_N^Z)=e_i(I;\Omega),
		\qquad i=1,2.
	\]
	Thus, if $I$ belongs to the $a$th layer, Proposition
	\ref{prop:layered-test-family} gives
	\begin{equation}\label{eq:zygmund-layer-depths}
		e_1(J;\Omega_N^Z)=a,
		\qquad e_2(J;\Omega_N^Z)=N-a.
	\end{equation}
	Every lifted layer partitions $\Omega_N^Z$. Combining this fact with
	\eqref{eq:zygmund-layer-depths} proves \eqref{eq:zygmund-layered-height-identity};
	integration gives \eqref{eq:zygmund-layered-mass-identity}. Taking the supremum over
	$N$ proves the final assertion.
\end{proof}

The proposition proves the necessity in Theorem
\ref{thm:intro-zygmund-admissible}. We now turn to the converse and its sparse refinement.
We use $W_w,H_w$ from the introduction, now with the Zygmund depths, and $\kappa_w$ from
\eqref{eq:general-marginal-coefficient}.

\begin{thm}[Admissible two-depth Zygmund packing]
\label{thm:zygmund-admissible-weight}
Let $(\Omega,\mathcal U)$ be an admissible pair in the Zygmund setting, and let $w$ be a
nonzero admissible weight. Then
there exist a probability space $(\Theta,\mathcal A,\mathbb P)$ and, for every
$\theta\in\Theta$, a $1/8$-sparse subfamily $\mathcal G(\theta)\subset\mathcal U$ such that
\begin{equation}\label{eq:zygmund-general-marginals}
	\mathbb P\bigl(\{\theta\in\Theta:I\in\mathcal G(\theta)\}\bigr)
	\geq\kappa_w W_w(I;\Omega),
	\qquad I\in\mathcal U.
\end{equation}
Moreover, there is a $1/8$-sparse subfamily $\mathcal G_0\subset\mathcal U$ such that
\begin{equation}\label{eq:zygmund-general-sparse-packing}
	\int_\Omega H_w(x)\,\ud x
	=\sum_{I\in\mathcal U}W_w(I;\Omega)|I|
	\leq2560\|w\|_{\mathrm{ft}}|\operatorname{sh}(\mathcal G_0)|.
\end{equation}
\end{thm}

\subsection{Third-coordinate order and the two intersections}

Since $\calD_Z$ is countable, so is $\mathcal U$. Fix an injective numbering
$\iota\colon\calD_Z\to\mathbb N$. For $I\in\mathcal U$, write
$\ell(I^i)=2^{-\nu_i(I)}$, $i=1,2,3$. On the Zygmund boundary,
$\nu_3(I)=\nu_1(I)+\nu_2(I)$, so $\nu_3(I)$ is precisely the combined generation used in
the bi-parameter proof. For distinct $I,J\in\mathcal U$, define their order as follows. If
$\nu_3(J)\neq\nu_3(I)$, put $J\prec I$ precisely when $\nu_3(J)<\nu_3(I)$. If
$\nu_3(J)=\nu_3(I)$, put $J\prec I$ precisely when $\iota(J)<\iota(I)$. As in the
bi-parameter proof, this defines a strict total order on $\mathcal U$. By construction,
\[
	J\prec I\quad\Longrightarrow\quad \nu_3(J)\leq\nu_3(I).
\]
Since $\nu_3=\nu_1+\nu_2$, this also gives the equivalent comparisons
\begin{equation}\label{eq:zygmund-order-gap-comparisons}
	\nu_1(I)-\nu_1(J)\geq\nu_2(J)-\nu_2(I),
	\qquad
	\nu_2(I)-\nu_2(J)\geq\nu_1(J)-\nu_1(I).
\end{equation}
If also $J\cap I\neq\varnothing$, the meeting dyadic cubes $J^3$ and $I^3$ are nested, and
the inequality $\nu_3(J)\leq\nu_3(I)$ gives $J^3\supset I^3$.

\begin{lem}[The two possible intersections]\label{lem:two-intersections}
	Suppose that $J\prec I$ and $J\cap I\neq\varnothing$. Exactly one of the following holds:
	\begin{align}
		J^1\supsetneq I^1,
		 & \qquad J^2\subsetneq I^2,
		 & I\cap J                   & =I^1\times J^2\times I^3;
		\label{eq:first-intersection}                            \\
		J^1\subsetneq I^1,
		 & \qquad J^2\supsetneq I^2,
		 & I\cap J                   & =J^1\times I^2\times I^3.
		\label{eq:second-intersection}
	\end{align}
\end{lem}

\begin{proof}
	As already noted, meeting dyadic cubes are nested in every coordinate, and the order gives
	$J^3\supset I^3$.
	We first rule out equality in the first two coordinates. If $J^1=I^1$, the first
	comparison in \eqref{eq:zygmund-order-gap-comparisons} gives
	$\nu_2(J)\leq\nu_2(I)$. Since $J^2$ and $I^2$ meet, this means $J^2\supset I^2$, and
	hence $J\supset I$, contrary to pairwise incomparability. If $J^2=I^2$, the second
	comparison gives $J^1\supset I^1$ and leads to the same contradiction.

	The first two coordinate inclusions are therefore strict. The possibility
	$J^1\supsetneq I^1$ and $J^2\supsetneq I^2$ would give $J\supset I$, contrary to
	incomparability. The other
	same-direction possibility, $J^1\subsetneq I^1$ and $J^2\subsetneq I^2$, would make
	the left side of the first comparison in \eqref{eq:zygmund-order-gap-comparisons}
	negative and its right side positive, which is impossible.
\end{proof}

For every $J\prec I$ with $J\cap I\neq\varnothing$, Lemma
\ref{lem:two-intersections} places the pair in exactly one of the two configurations above.
In the first configuration, call the pair \emph{near} when its positive first-coordinate
generation gap satisfies
\begin{equation}\label{eq:first-near}
	\nu_1(I)-\nu_1(J)\leq e_1(I;\Omega).
\end{equation}
In the second configuration, call it near when its positive second-coordinate generation
gap satisfies
\begin{equation*}
	\nu_2(I)-\nu_2(J)\leq e_2(I;\Omega).
\end{equation*}
An intersecting ordered pair that is not near is \emph{far}. As before, $(J,I)$ near or far is
oriented notation: $J\prec I$, and nearness is measured using the depth of the second,
later rectangle $I$.

\subsection{Far predecessors}

For a fixed later rectangle $I$, define the far predecessor sets
\[
	\mathcal F_1(I)
	:=\bigcup_{\substack{J\prec I:\ \eqref{eq:second-intersection}\text{ holds}\\(J,I)\text{ far}}}J^1
	\subset I^1,
	\qquad
	\mathcal F_2(I)
	:=\bigcup_{\substack{J\prec I:\ \eqref{eq:first-intersection}\text{ holds}\\(J,I)\text{ far}}}J^2
	\subset I^2.
\]
These sets are measurable because $\mathcal U$ is countable.

\begin{lem}[Zygmund far predecessors leave a major subset]\label{lem:large-gap-predecessors}
	For every $I\in\mathcal U$,
	\begin{equation}\label{eq:far-half}
		|\mathcal F_i(I)|<\frac12|I^i|,
		\qquad i=1,2.
	\end{equation}
	Consequently,
	\begin{equation}\label{eq:far-product-subset}
		A_I:=\big(I^1\setminus\mathcal F_1(I)\big)
		\times\big(I^2\setminus\mathcal F_2(I)\big)\times I^3
	\end{equation}
	has measure greater than $|I|/4$ and avoids every far predecessor of $I$.
\end{lem}

\begin{proof}
	We prove the assertion for $\mathcal F_2(I)$. Put $k=e_1(I;\Omega)+1$ and
	\[
		Q=(I^1)^{(k)}\times I^2\times I^3=I_1^{(k)}\in\calD_{\mathrm{sZ}}.
	\]
	For every far pair satisfying \eqref{eq:first-intersection},
	$\nu_1(I)-\nu_1(J)\geq k$, so $J^1\supset(I^1)^{(k)}$. Suppose, toward a
	contradiction, that $|\mathcal F_2(I)|\geq|I^2|/2$.
	Let
	$x_2\in\mathcal F_2(I)$ and choose a corresponding far predecessor $J$ with
	$x_2\in J^2$. We have $(I^1)^{(k)}\subset J^1$, while $I^3\subset J^3$ by the
	third-coordinate order. Since $J\subset\Omega$, this proves
	\[
		(I^1)^{(k)}\times\mathcal F_2(I)\times I^3\subset\Omega.
	\]
	Apply Lemma \ref{lem:half-slab-halo} with $\mathcal B=\calD_{\mathrm{sZ}}$,
	$G=\Omega$, the rectangle $Q$ above,
	$E^2=\mathcal F_2(I)\times I^3$, and $L^1=I^1$. The assumed measure bound and
	$I\subset\Omega$ verify its remaining hypotheses, so $Q\subset\widetilde\Omega$.
	This contradicts the
	definition of $e_1(I;\Omega)$ because $Q=I_1^{(e_1(I;\Omega)+1)}$, and proves
	$|\mathcal F_2(I)|<|I^2|/2$. The estimate
	$|\mathcal F_1(I)|<|I^1|/2$ follows by the same argument.
	By \eqref{eq:far-half}, the measure of the set in
	\eqref{eq:far-product-subset} is
	$|I^1\setminus\mathcal F_1(I)|\,
	|I^2\setminus\mathcal F_2(I)|\,|I^3|>|I|/4$.
	If a far predecessor $J\prec I$ satisfies \eqref{eq:first-intersection}, then
	\[
		J\cap I=I^1\times J^2\times I^3,
		\qquad J^2\subset\mathcal F_2(I).
	\]
	Hence it is disjoint from $A_I$, whose second
	factor is $I^2\setminus\mathcal F_2(I)$. If instead $J$ satisfies
	\eqref{eq:second-intersection}, then
	\[
		J\cap I=J^1\times I^2\times I^3,
		\qquad J^1\subset\mathcal F_1(I),
	\]
	and the first factor of $A_I$ makes the two sets disjoint. Thus $A_I$ avoids every far
	predecessor.
\end{proof}

\subsection{Near successors}\label{sec:zygmund-near-successors}

Lemma \ref{lem:large-gap-predecessors} verifies hypothesis (i) of Lemma
\ref{lem:local-overlap-to-marginals} with $\eta=1/4$. To apply that lemma, it remains to
prove the local weighted-overlap estimate for the later near successors of each fixed
earlier rectangle. As in the bi-parameter proof, this is where Lemma
\ref{lem:fixed-total-packing} enters. For the rest of this proof, write $W_w(R)$
for $W_w(R;\Omega)$ whenever $R\in\mathcal U$.

\begin{lem}[Zygmund local near-successor estimate]\label{lem:zygmund-near}
	For every nonzero admissible weight $w$ and every
	$J\in\mathcal U$,
	\begin{equation}\label{eq:zygmund-local-near-overlap}
		\sum_{\substack{I:\,J\prec I\\(J,I)\text{ near}}}
		W_w(I)|J\cap I|\leq4\|w\|_{\mathrm{ft}}|J|.
	\end{equation}
\end{lem}

\begin{proof}
	Fix the earlier rectangle $J$. The sum in \eqref{eq:zygmund-local-near-overlap} splits
	according to \eqref{eq:first-intersection} and \eqref{eq:second-intersection}. We bound
	each part by $2\|w\|_{\mathrm{ft}}|J|$, using Lemma
	\ref{lem:fixed-total-packing} once for each
	configuration. Denote the two families of later near successors by
	\[
	\begin{aligned}
		\mathcal N_1(J)&:=\{I:J\prec I,\ (J,I)\text{ near, and }
		\eqref{eq:first-intersection}\text{ holds}\},\\
		\mathcal N_2(J)&:=\{I:J\prec I,\ (J,I)\text{ near, and }
		\eqref{eq:second-intersection}\text{ holds}\}.
	\end{aligned}
	\]
	We give the details for $\mathcal N_1(J)$. The second configuration follows by exchanging
	the first and second coordinate roles, as at the end of the bi-parameter proof, and
	replacing $w$ by $w^\top$. Until the final symmetry paragraph, all notation below is
	local to the fixed $J$ and the first configuration.

	For $I\in\mathcal N_1(J)$, set
	\begin{equation}\label{eq:zygmund-first-data}
	\begin{gathered}
		p(I):=\nu_1(I)-\nu_1(J),\qquad
		q(I):=\nu_2(J)-\nu_2(I),\qquad
		P_I:=I^1\times I^3\subset J^1\times J^3,\\
		a(I):=q(I),\qquad b(I):=e_2(I;\Omega),\qquad
		I^2=(J^2)^{(q(I))},\qquad |J\cap I|=|P_I||J^2|,\\
		p(I)\geq q(I),\qquad e_1(I;\Omega)\geq p(I).
	\end{gathered}
	\end{equation}
	The first comparison in \eqref{eq:zygmund-order-gap-comparisons} gives
	$p(I)\geq q(I)$, while the near condition \eqref{eq:first-near} gives
	$e_1(I;\Omega)\geq p(I)$.

	We verify hypothesis (i) of Lemma
	\ref{lem:exclusion-packing}. For each fixed integer $q$, the rectangles
	$P_I$ with $I\in\mathcal N_1(J)$ and $q(I)=q$ are pairwise disjoint. Indeed, if
	$P_I\cap P_K\neq\varnothing$, then $I^1$ and $K^1$ meet and are dyadically nested.
	After interchanging $I$ and $K$ if necessary, assume that $I^1\subset K^1$.
	The third-coordinate cubes also meet. Moreover, $q(I)=q(K)$ gives
	$I^2=K^2=(J^2)^{(q(I))}$. The Zygmund relation therefore gives
	\[
		\frac{\ell(I^3)}{\ell(K^3)}
		=\frac{\ell(I^1)\ell(I^2)}{\ell(K^1)\ell(K^2)}
		=\frac{\ell(I^1)}{\ell(K^1)}.
	\]
	The displayed equality gives $\ell(I^3)\leq\ell(K^3)$. Dyadic nesting of the meeting
	third-coordinate cubes now gives
	$I^3\subset K^3$. Together with $I^2=K^2$, this yields $I\subset K$, contrary to
	pairwise incomparability unless $I=K$. Thus distinct sets $P_I$ cannot meet.

	As in the bi-parameter proof, fixed-$q$ disjointness does not control overlap between
	different values of $q$. To verify hypothesis (ii) of Lemma
	\ref{lem:exclusion-packing}, we prove that, for every integer $h\geq1$,
	\begin{equation}\label{eq:zygmund-first-density-implication}
		\Bigg|P_I\cap
		\bigcup_{\substack{K\in\mathcal N_1(J)\\q(K)\geq q(I)+h}}P_K\Bigg|
		\geq\frac12|P_I|
		\quad\Longrightarrow\quad
		e_2(I;\Omega)\geq h.
	\end{equation}
	This is the density calculation in
	\eqref{eq:planar-first-density-implication}, now with the first and third coordinates
	grouped into $P_I=I^1\times I^3$. Since
	$K^2=(J^2)^{(q(K))}$ and $I^2=(J^2)^{(q(I))}$, every $K$ in the union satisfies
	$K^2\supset(I^2)^{(h)}$. Set
	\[
		A:=P_I\cap
		\bigcup_{\substack{K\in\mathcal N_1(J)\\q(K)\geq q(I)+h}}P_K,
		\qquad
		Q:=I^1\times(I^2)^{(h)}\times I^3=I_2^{(h)}\in\calD_{\mathrm{sZ}}.
	\]
	If $(x_1,x_3)\in A$, choose $K$ with $(x_1,x_3)\in P_K=K^1\times K^3$.
	Then every $x_2\in(I^2)^{(h)}\subset K^2$ gives
	$(x_1,x_2,x_3)\in K\subset\Omega$. Consequently,
	\[
		\{(x_1,x_2,x_3):(x_1,x_3)\in A,\ x_2\in(I^2)^{(h)}\}\subset\Omega.
	\]
	The hypothesis in \eqref{eq:zygmund-first-density-implication} gives
	$|A|\geq|P_I|/2$. Apply the coordinate-reversed form of Lemma
	\ref{lem:half-slab-halo} with $\mathcal B=\calD_{\mathrm{sZ}}$, $G=\Omega$, the
	rectangle $Q$ above, $E^1=A$, and $L^2=I^2$. Since $I\subset\Omega$, the lemma gives
	$Q\subset\widetilde\Omega$, and hence $e_2(I;\Omega)\geq h$.

	Taking $h=e_2(I;\Omega)+1$ in the contrapositive of
	\eqref{eq:zygmund-first-density-implication} verifies hypothesis (ii). We may therefore
	apply Lemma \ref{lem:fixed-total-packing}, with weight $w$, to
	\[
		X=J^1\times J^3,\qquad \mathcal A=\mathcal N_1(J),\qquad
		P_I=I^1\times I^3,\qquad a(I)=q(I),\qquad b(I)=e_2(I;\Omega).
	\]
	The index set is countable,
	and $P_I\subset X$ by \eqref{eq:zygmund-first-data}. Thus
	Lemma \ref{lem:fixed-total-packing} gives
	\begin{equation}\label{eq:zygmund-first-projection-packing}
		\sum_{I\in\mathcal N_1(J)}
		w(q(I),e_2(I;\Omega))|P_I|
		\leq2\|w\|_{\mathrm{ft}}|J^1\times J^3|.
	\end{equation}
	The inequalities in \eqref{eq:zygmund-first-data} give
	\[
		W_w(I)=w(e_1(I;\Omega),e_2(I;\Omega))
		\leq w(q(I),e_2(I;\Omega)).
	\]
	Combining this pointwise bound with
	\eqref{eq:zygmund-first-projection-packing} and
	$|J\cap I|=|P_I||J^2|$ yields
	\begin{equation*}
	\begin{aligned}
		\sum_{I\in\mathcal N_1(J)}W_w(I)|J\cap I|
		&\leq |J^2|\sum_{I\in\mathcal N_1(J)}
		w(q(I),e_2(I;\Omega))|P_I|\\
		&\leq2\|w\|_{\mathrm{ft}}|J^2||J^1\times J^3|=2\|w\|_{\mathrm{ft}}|J|.
	\end{aligned}
	\end{equation*}

	By the coordinate symmetry described above and
	$\|w^\top\|_{\mathrm{ft}}=\|w\|_{\mathrm{ft}}$, the same bound holds with
	$\mathcal N_1(J)$ replaced by $\mathcal N_2(J)$. Adding the two bounds proves
	\eqref{eq:zygmund-local-near-overlap}.
\end{proof}

\subsection{Completion of the two-depth packing proof}\label{sec:zygmund-completion}

\begin{proof}[Proof of Theorem \ref{thm:zygmund-admissible-weight}]
	Lemma \ref{lem:large-gap-predecessors} verifies hypothesis (i) of Lemma
	\ref{lem:local-overlap-to-marginals} with $\eta=1/4$, while Lemma
	\ref{lem:zygmund-near} verifies hypothesis (ii) for $W_w$ with constant $4\|w\|_{\mathrm{ft}}$.
	As in the bi-parameter proof, every value of $w$ is at most
	$\|w\|_{\mathrm{ft}}$. Apply Lemma
	\ref{lem:local-overlap-to-marginals} to $\|w\|_{\mathrm{ft}}^{-1}W_w$. The resulting
	inclusion-probability estimate is
	\[
		\mathbb P\bigl(\{\theta\in\Theta:I\in\mathcal G(\theta)\}\bigr)
		\geq\frac{\|w\|_{\mathrm{ft}}^{-1}W_w(I;\Omega)}{32(4+1)}
		=\kappa_wW_w(I;\Omega),
	\]
	which is \eqref{eq:zygmund-general-marginals}.

	Since $\operatorname{sh}(\mathcal U)\subset\Omega$, Remark
	\ref{rem:marginals-to-scalar}, applied with this lower bound, directly gives a
	$1/8$-sparse subfamily $\mathcal G_0\subset\mathcal U$ such that
	\[
		\sum_{I\in\mathcal U}W_w(I;\Omega)|I|
		\leq\frac{16}{\kappa_w}|\operatorname{sh}(\mathcal G_0)|
		=2560\|w\|_{\mathrm{ft}}|\operatorname{sh}(\mathcal G_0)|,
	\]
	which completes \eqref{eq:zygmund-general-sparse-packing}.
\end{proof}

\subsection{Square-root exponential comparison}

The preceding proof constructed sparse families with marginal lower bounds and used only an
$L^1$ estimate for each selected family. We now separate the two additional ingredients that
give a stronger conclusion. First, every sparse Zygmund family satisfies $L^q$ overlap bounds
with growth $q^2$, and hence a square-root exponential estimate. This fact concerns sparse
families alone; it neither assumes nor uses incomparability. Second, the marginal lower bounds
allow us to transfer that unweighted overlap estimate to the weighted height by convexity and
then to choose one sparse subfamily.

For the transfer we use $\cosh(\lambda\sqrt t)-1$. Unlike the more familiar
$e^{\lambda\sqrt t}$, this function is increasing and convex in $t$ on $[0,\infty)$, while
having the same square-root exponential growth for large $t$. This explains the form of the
sparse comparison below; the usual exponential estimate then follows from
$e^u\leq2\cosh u$. The result is the square-root analogue of
\eqref{eq:planar-general-sparse-exponential}.

The estimate \eqref{eq:zygmund-general-Lp} below is stated for the full range
$1\leq q<\infty$. At $q=1$, Theorem \ref{thm:zygmund-admissible-weight} gives a better
numerical constant; the new point here is the $q^2$ growth for $q>1$.

\begin{thm}[Square-root exponential comparison for an admissible weight]
\label{thm:zygmund-admissible-exponential}
Let $(\Omega,\mathcal U)$ be an admissible pair in the Zygmund setting, and let $w$ be a
nonzero admissible weight. There exist an absolute constant $b_{\mathrm Z}>0$ and a
$1/8$-sparse subfamily
$\mathcal G_0\subset\mathcal U$ such that
\begin{equation}\label{eq:zygmund-general-sparse-exponential}
\begin{aligned}
	&\int_{\R^{d_1+d_2+d_3}}
	\left[\cosh\!\bigl(b_{\mathrm Z}\sqrt{\kappa_wH_w(x)}\bigr)-1\right]\,\ud x\\
	&\quad\leq\int_{\R^{d_1+d_2+d_3}}
	\left[\cosh\!\bigl(b_{\mathrm Z}\sqrt{h_{\mathcal G_0}(x)}\bigr)-1\right]\,\ud x
	\leq\frac14|\operatorname{sh}(\mathcal G_0)|.
\end{aligned}
\end{equation}
Consequently,
\begin{equation}\label{eq:zygmund-general-exponential}
	\int_\Omega
	\exp\!\bigl(b_{\mathrm Z}\sqrt{\kappa_wH_w(x)}\bigr)\,\ud x
	\leq\frac52|\Omega|.
\end{equation}
The same sparse family satisfies, for every $1\leq q<\infty$,
\begin{equation}\label{eq:zygmund-general-Lp}
	\|H_w\|_{L^q(\Omega)}
	\leq\frac{2q^2}{b_{\mathrm Z}^2\kappa_w}
	|\operatorname{sh}(\mathcal G_0)|^{1/q}.
\end{equation}
\end{thm}

The unweighted input in the preceding results is isolated next. Its square-root exponential
conclusion is also implicit in Rey's general maximal-function-to-overlap principle
\cite{Rey2026antichain}*{Theorem 1.1(a)}, combined with the weak $L\log L$ estimate for the
dyadic Zygmund maximal function. We give the short direct proof because it records both the
precise $q^2$ growth and the convex $\cosh$ estimate used in the transfer to $H_w$.

\begin{lem}[Square-root overlap for sparse Zygmund families]
	\label{lem:sparse-zygmund-square-root-overlap}
	There exist absolute constants $C_{\mathrm Z}\geq1$ and $b_{\mathrm Z}>0$ with the
	following property. Let $\mathcal G\subset\calD_Z$ be a $1/8$-sparse family with
	$|\operatorname{sh}(\mathcal G)|<\infty$. Then, for every $1\leq q<\infty$,
	\begin{equation}\label{eq:sparse-zygmund-Lp}
		\|h_{\mathcal G}\|_{L^q}
		\leq C_{\mathrm Z}q^2|\operatorname{sh}(\mathcal G)|^{1/q},
	\end{equation}
	and
	\begin{equation}\label{eq:sparse-zygmund-cosh}
		\int_{\R^{d_1+d_2+d_3}}
		\left[\cosh\!\bigl(b_{\mathrm Z}\sqrt{h_{\mathcal G}(x)}\bigr)-1\right]\,\ud x
		\leq\frac14|\operatorname{sh}(\mathcal G)|.
	\end{equation}
	Consequently,
	\[
		\int_{\operatorname{sh}(\mathcal G)}
		\exp\!\bigl(b_{\mathrm Z}\sqrt{h_{\mathcal G}(x)}\bigr)\,\ud x
		\leq\frac52|\operatorname{sh}(\mathcal G)|.
	\]
\end{lem}

\begin{proof}
	The classical dyadic Zygmund maximal theorem of C\'{o}rdoba
	\cite{Cordoba1979}, in its weak $L\log L$ form, together with standard endpoint
	interpolation, allows us to fix an absolute constant $A_{\mathrm Z}\geq1$ such that
	\begin{equation}\label{eq:zygmund-maximal-Lp}
		\|M_{\calD_Z}\|_{L^p\to L^p}
		\leq A_{\mathrm Z}(p-1)^{-2},
		\qquad1<p\leq2.
	\end{equation}
	See also \cite{FEPI}. Rey recently gave another proof of the underlying dyadic weak
	$L\log L$ endpoint, using the bi-parameter exponential-overlap estimate in Lemma
	\ref{lem:incomparable-exponential-overlap} as a key input; see
	\cite{Rey2026antichain}*{Theorem A and Sections 1.3 and 2}. His argument, written in
	$\R^3$, works verbatim in
	$\R^{d_1+d_2+d_3}$.

	Put $C_{\mathrm Z}:=8A_{\mathrm Z}$. We first prove
	\eqref{eq:sparse-zygmund-Lp}. Let $q\geq2$, put $p=q'$, and take
	$g\geq0$ with $\|g\|_p=1$. Choose pairwise disjoint measurable sets
	$E_I\subset I$ such that $|E_I|\geq|I|/8$ for every $I\in\mathcal G$. For
	$x\in E_I$, we have $M_{\calD_Z}g(x)\geq\ave{g}_I$. Hence
	\begin{align*}
		\int h_{\mathcal G}g
		&=\sum_{I\in\mathcal G}|I|\ave{g}_I
		\leq8\sum_{I\in\mathcal G}|E_I|\ave{g}_I
		\leq8\int_{\operatorname{sh}(\mathcal G)}M_{\calD_Z}g\\
		&\leq8A_{\mathrm Z}(q-1)^2
		|\operatorname{sh}(\mathcal G)|^{1/q}
		\leq C_{\mathrm Z}q^2|\operatorname{sh}(\mathcal G)|^{1/q}.
	\end{align*}
	Here we used \eqref{eq:zygmund-maximal-Lp} with $p=q'$, so
	$p-1=1/(q-1)$. Taking the supremum over $g$ proves the estimate for $q\geq2$.
	At $q=1$, sparseness gives
	\[
		\|h_{\mathcal G}\|_{L^1}
		=\sum_{I\in\mathcal G}|I|
		\leq8|\operatorname{sh}(\mathcal G)|
		\leq C_{\mathrm Z}|\operatorname{sh}(\mathcal G)|.
	\]
	At $q=2$, the preceding calculation gives the sharper bound
	$\|h_{\mathcal G}\|_{L^2}\leq C_{\mathrm Z}|\operatorname{sh}(\mathcal G)|^{1/2}$.
	For $1<q<2$, H\"older's inequality and the fact that $h_{\mathcal G}$ is supported on
	$\operatorname{sh}(\mathcal G)$ therefore give the stated estimate.
	This proves \eqref{eq:sparse-zygmund-Lp} in the full range.

	Set
	\[
		b_{\mathrm Z}:=\frac{2}{e\sqrt{5C_{\mathrm Z}}}.
	\]
	For $t\geq0$, the relevant power series is
	\[
		\cosh(b_{\mathrm Z}\sqrt t)-1
		=\sum_{k=1}^{\infty}\frac{b_{\mathrm Z}^{2k}}{(2k)!}t^k.
	\]
	To estimate its $k$th term after substituting $t=h_{\mathcal G}$, use
	\eqref{eq:sparse-zygmund-Lp} with $q=k$:
	\[
		\int h_{\mathcal G}^k
		\leq(C_{\mathrm Z}k^2)^k|\operatorname{sh}(\mathcal G)|.
	\]
	Since $(2k)!\geq(2k/e)^{2k}$, it follows that
	\[
		\frac{b_{\mathrm Z}^{2k}}{(2k)!}\int h_{\mathcal G}^k
		\leq\left(\frac{b_{\mathrm Z}^2C_{\mathrm Z}e^2}{4}\right)^k
		|\operatorname{sh}(\mathcal G)|
		=5^{-k}|\operatorname{sh}(\mathcal G)|.
	\]
	Summing these estimates over the nonnegative power series gives
	\[
		\int\left[\cosh\!\bigl(b_{\mathrm Z}\sqrt{h_{\mathcal G}}\bigr)-1\right]
		\leq|\operatorname{sh}(\mathcal G)|\sum_{k=1}^{\infty}5^{-k}
		=\frac14|\operatorname{sh}(\mathcal G)|,
	\]
	which is \eqref{eq:sparse-zygmund-cosh}. Finally, $e^u\leq2\cosh u$ for $u\geq0$
	gives the last displayed estimate in the statement.
\end{proof}

\emph{Remark.} The order is deliberate. We first sum the $L^q$ estimates
\eqref{eq:sparse-zygmund-Lp} into the convex estimate
\eqref{eq:sparse-zygmund-cosh}. The marginal transfer then gives
\eqref{eq:zygmund-general-sparse-exponential}, from which we recover
\eqref{eq:zygmund-general-Lp}. This route produces one family $\mathcal G_0$ that works
for every $q$; transferring $t^q$ separately could select a different sparse family for
each $q$.

The random families produced by Theorem \ref{thm:zygmund-admissible-weight} are both sparse
and incomparable, but Lemma \ref{lem:sparse-zygmund-square-root-overlap} and the transfer
below use only their sparseness. A full exponential improvement would therefore require
exploiting their additional geometry without losing the marginal bounds
\eqref{eq:zygmund-general-marginals}.

\begin{proof}[Proof of Theorem \ref{thm:zygmund-admissible-exponential}]
	The proof of Theorem \ref{thm:zygmund-admissible-weight} produces a probability space
	$(\Theta,\mathcal A,\mathbb P)$ and, for every $\theta\in\Theta$, a
	$1/8$-sparse family $\mathcal G(\theta)\subset\mathcal U$ satisfying
	\eqref{eq:zygmund-general-marginals}. Let $b_{\mathrm Z}$ be the constant in Lemma
	\ref{lem:sparse-zygmund-square-root-overlap}, and put
	\[
		\Psi(t):=\cosh(b_{\mathrm Z}\sqrt t)-1.
	\]
	The power series for $\cosh$ shows that $\Psi$ is increasing and convex on
	$[0,\infty)$ and satisfies $\Psi(0)=0$. Since every $\mathcal G(\theta)$ is
	$1/8$-sparse, Lemma \ref{lem:sparse-zygmund-square-root-overlap} gives
	\begin{equation}\label{eq:zygmund-selected-cosh}
		\int\Psi(h_{\mathcal G(\theta)})
		\leq\frac14|\operatorname{sh}(\mathcal G(\theta))|,
		\qquad\theta\in\Theta.
	\end{equation}

	Since $\mathcal G(\theta)\subset\mathcal U$, its shadow is contained in $\Omega$, so
	\eqref{eq:zygmund-selected-cosh} shows that the averaged integral in
	\eqref{eq:abstract-convex-transfer} is finite. Lemma
	\ref{lem:inclusion-probabilities-to-convex-overlap}, applied with
	$w(I)=W_w(I;\Omega)$, $\kappa=\kappa_w$ and $\Phi=\Psi$, therefore combines with
	the marginal estimate \eqref{eq:zygmund-general-marginals} to give
	$\theta_0\in\Theta$ such that
	\[
		\int\Psi(\kappa_wH_w)
		\leq\int\Psi(h_{\mathcal G(\theta_0)}).
	\]
	Set $\mathcal G_0:=\mathcal G(\theta_0)$. Together with
	\eqref{eq:zygmund-selected-cosh}, the preceding inequality is precisely
	\eqref{eq:zygmund-general-sparse-exponential}.

	Put $Y:=\kappa_wH_w$.
	For $u\geq0$, $e^u\leq2\cosh u$. Since $\Psi\geq0$,
	\eqref{eq:zygmund-general-sparse-exponential} gives
	\[
		\int_\Omega e^{b_{\mathrm Z}\sqrt Y}
		\leq2|\Omega|+2\int_\Omega\Psi(\kappa_wH_w)
		\leq2|\Omega|+\frac12|\operatorname{sh}(\mathcal G_0)|
		\leq\frac52|\Omega|.
	\]
	This proves \eqref{eq:zygmund-general-exponential}.

	It remains to extract the sparse $L^q$ estimates. For $u\geq0$ and $q\geq1$,
	\begin{equation}\label{eq:square-root-cosh-to-Lp}
		u^{2q}\leq8\left(\frac{2q}{e}\right)^{2q}(\cosh u-1).
	\end{equation}
	Indeed, if $0\leq u\leq1$, then
	$u^{2q}\leq u^2\leq2(\cosh u-1)$, where the second inequality follows from
	the power series for $\cosh$. The coefficient on the right of
	\eqref{eq:square-root-cosh-to-Lp} is at least $2$. If $u\geq1$, then
	$e^{-u}\leq e^{-1}<1/2$, and therefore
	\[
		\cosh u-1=\frac{e^u}{2}(1-e^{-u})^2\geq\frac{e^u}{8}.
	\]
	Moreover, $u^{2q}\leq(2q/e)^{2q}e^u$ by maximizing
	$u^{2q}e^{-u}$.

	Apply \eqref{eq:square-root-cosh-to-Lp} pointwise with
	\[
		u:=b_{\mathrm Z}\sqrt{\kappa_wH_w(x)}.
	\]
	Using \eqref{eq:zygmund-general-sparse-exponential}, we obtain
	\[
		b_{\mathrm Z}^{2q}\kappa_w^q\int_\Omega H_w^q
		\leq2\left(\frac{2q}{e}\right)^{2q}
		|\operatorname{sh}(\mathcal G_0)|.
	\]
	Taking the $q$th root and using
	$2^{1/q}(2/e)^2<2$ proves \eqref{eq:zygmund-general-Lp}.
\end{proof}

\subsection{Full exponential packing under diagonal summability}
\label{sec:zygmund-full-exponential}

The preceding square-root estimate applies to every admissible weight. A smaller natural class admits
full exponential integrability even in the Zygmund setting. The useful viewpoint is to
replace $w$ by its one-variable \emph{diagonal envelope}
\[
	w^\Delta(N):=\max_{0\leq a\leq N}w(a,N-a),
	\qquad
	\|w\|_{\mathrm{env}}:=\|w^\Delta\|_{\ell^1(\N)}
	=\sum_{N=0}^{\infty}w^\Delta(N).
\]
Thus the relevant condition is precisely $w^\Delta\in\ell^1(\N)$. We denote by
$\mathcal W_{\mathrm{env}}$ the weights satisfying this condition.
For every $w\in\mathcal W_{12}$,
\[
	\|w\|_{\mathrm{ft}}
	\leq\|w\|_{\mathrm{env}}
	\leq\|w\|_{\ell^1(\N^2)}.
\]
The second inequality is immediate. For the first, note that $w^\Delta$ is decreasing:
from any $(a,b)$ with $a+b=N+1$, subtracting $1$ from a positive coordinate produces
a point on the preceding diagonal whose weight is at least $w(a,b)$. Thus
$w^\Delta(N+1)\leq w^\Delta(N)$, and
\[
	\sum_{a=0}^Nw(a,N-a)
	\leq(N+1)w^\Delta(N)
	\leq\sum_{j=0}^Nw^\Delta(j)
	\leq\|w\|_{\mathrm{env}}.
\]
Diagonal-envelope summability is strictly weaker than $\ell^1$ summability: for
$w(a,b)=(a+b+1)^{-2}$, the envelope norm is
$\sum_{N\geq0}(N+1)^{-2}<\infty$, whereas the $\ell^1$ norm is
$\sum_{N\geq0}(N+1)^{-1}=\infty$. More generally, every total-depth weight
$w(a,b)=\omega(a+b)$ with $\omega$ decreasing and summable satisfies
$w^\Delta=\omega$ and $\|w\|_{\mathrm{env}}=\|\omega\|_{\ell^1(\N)}$.

\begin{lem}[Product-sparse form of Rey's estimate]
\label{lem:product-sparse-rey}
Let $\mathcal G\subset\calD^1\times\calD^2\times\calD^3$, and let
$0<\eta\leq1$. For $x_3\in\R^{d_3}$, set
\[
	\mathcal G_{x_3}:=\{I\in\mathcal G:x_3\in I^3\},
	\qquad
	\mathcal B_{x_3}:=\{I^1\times I^2:I\in\mathcal G_{x_3}\}.
\]
Suppose that, for every $x_3$, the projection $I\mapsto I^1\times I^2$ is injective
on $\mathcal G_{x_3}$ and $\mathcal B_{x_3}$ is pairwise incomparable. Suppose also
that there are
measurable sets $E_I^{12}\subset I^1\times I^2$ such that
\[
	|E_I^{12}|\geq\eta|I^1\times I^2|
\]
and the sets $E_I^{12}\times I^3$, $I\in\mathcal G$, are pairwise disjoint. Then,
with $t_\eta=\log(1+\eta/32)$,
\begin{equation}\label{eq:product-sparse-rey}
	\int_{\R^{d_1+d_2+d_3}}
	\left[e^{t_\eta h_{\mathcal G}(x)}-1\right]\,\ud x
	\leq\frac1{16}|\operatorname{sh}(\mathcal G)|.
\end{equation}
\end{lem}

\begin{proof}
The argument is the slicing argument of Rey
\cite{Rey2026antichain}*{Sections 1.3 and 2}. Fix $x_3\in\R^{d_3}$. By injectivity,
every $B\in\mathcal B_{x_3}$ has a unique preimage
$I_B\in\mathcal G_{x_3}$. Set $E_B:=E_{I_B}^{12}$. Since all the
corresponding cubes $I_B^3$ contain $x_3$, the sets $E_B$ are pairwise disjoint. Thus
$\mathcal B_{x_3}$ is $\eta$-sparse and, by assumption, pairwise incomparable.
Moreover,
\[
	h_{\mathcal G}(x_{12},x_3)
	=\sum_{B\in\mathcal B_{x_3}}1_B(x_{12}),
	\qquad
	\operatorname{sh}(\mathcal B_{x_3})
	=\{x_{12}:(x_{12},x_3)\in\operatorname{sh}(\mathcal G)\}.
\]
Lemma \ref{lem:incomparable-exponential-overlap} applied to each planar slice,
followed by Fubini's theorem, gives \eqref{eq:product-sparse-rey}.
\end{proof}

\begin{thm}[Diagonal-envelope exponential packing]
\label{thm:diagonal-envelope-exponential}
Let $(\Omega,\mathcal U)$ be an admissible pair in the Zygmund setting. Let
$w\in\mathcal W_{\mathrm{env}}\cap\mathcal W_{12}$ be nonzero, and put
$t_1:=\log(1+1/128)$.
Then there is a $1/4$-sparse subfamily $\mathcal G_0\subset\mathcal U$ such that
\begin{equation}\label{eq:diagonal-envelope-sparse-exponential}
\begin{aligned}
	&\int_{\R^{d_1+d_2+d_3}}
	\left[\exp\!\left(t_1\|w\|_{\mathrm{env}}^{-1}
	H_w(x)\right)-1\right]\,\ud x\\
	&\quad\leq
	\int_{\R^{d_1+d_2+d_3}}
	\left[e^{t_1h_{\mathcal G_0}(x)}-1\right]\,\ud x
	\leq\frac1{16}|\operatorname{sh}(\mathcal G_0)|.
\end{aligned}
\end{equation}
\end{thm}

\begin{proof}
Put $v:=w^\Delta$. Then $v\colon\N\to[0,\infty)$ is decreasing, summable and nonzero.
Lemma \ref{lem:stationary-one-depth-events} provides a probability space
$(\Theta_v,\mathcal A_v,\mathbb P_v)$ and events $\mathsf S^v_{a,b}$ satisfying
\[
	\mathbb P_v(\mathsf S^v_{a,b})
	=\frac{v(b)}{\|w\|_{\mathrm{env}}},
	\qquad a\in\Z,\quad b\geq0,
\]
as well as the two-sided exclusion property
\eqref{eq:stationary-two-sided-exclusion}. Put
$D_I:=e_1(I;\Omega)+e_2(I;\Omega)$, and, for $\theta\in\Theta_v$, define
\[
	\mathcal G(\theta)
	:=\{I\in\mathcal U:\theta\in\mathsf S^v_{\nu_1(I),D_I}\}.
\]
For every $I\in\mathcal U$,
\begin{equation}\label{eq:diagonal-envelope-marginals}
\begin{aligned}
	\mathbb P_v\bigl(\{\theta:I\in\mathcal G(\theta)\}\bigr)
	&=\mathbb P_v(\mathsf S^v_{\nu_1(I),D_I})
	=\frac{v(D_I)}{\|w\|_{\mathrm{env}}}\\
	&=\frac{\max_{a+b=D_I}w(a,b)}{\|w\|_{\mathrm{env}}}
	\geq\frac{W_w(I;\Omega)}{\|w\|_{\mathrm{env}}}.
\end{aligned}
\end{equation}

Fix $\theta\in\Theta_v$, and let $I\neq J$ be intersecting members of
$\mathcal G(\theta)$. After interchanging them if necessary, assume that $J\prec I$.
Lemma \ref{lem:two-intersections} gives $0<|\nu_1(I)-\nu_1(J)|$. Since $\theta$ belongs
to both corresponding exclusion events, the two-sided exclusion property
\eqref{eq:stationary-two-sided-exclusion} gives
\[
	|\nu_1(I)-\nu_1(J)|>\max\{D_I,D_J\}\geq D_I.
\]
In configuration \eqref{eq:first-intersection}, this gives
$\nu_1(I)-\nu_1(J)>D_I\geq e_1(I;\Omega)$, so $(J,I)$ is far. In configuration
\eqref{eq:second-intersection}, the second comparison in
\eqref{eq:zygmund-order-gap-comparisons} gives
\[
	\nu_2(I)-\nu_2(J)
	\geq\nu_1(J)-\nu_1(I)
	>D_I\geq e_2(I;\Omega),
\]
so $(J,I)$ is again far. Thus any two distinct intersecting members of
$\mathcal G(\theta)$ form a far pair.

Write the major subset in \eqref{eq:far-product-subset} as
\[
	A_I=A_I^{12}\times I^3,
	\qquad
	A_I^{12}:=(I^1\setminus\mathcal F_1(I))
	\times(I^2\setminus\mathcal F_2(I)).
\]
Then $|A_I^{12}|>|I^1\times I^2|/4$. Since every pair of distinct intersecting members of
$\mathcal G(\theta)$ is far, and $A_I$ avoids every far predecessor of $I$, the sets
$A_I$, $I\in\mathcal G(\theta)$, are pairwise disjoint. Hence these sets verify the
product-form hypothesis of Lemma \ref{lem:product-sparse-rey}, with
$E_I^{12}=A_I^{12}$ and $\eta=1/4$. They also show that
$\mathcal G(\theta)$ itself is $1/4$-sparse.

It remains to verify the slice hypothesis of Lemma \ref{lem:product-sparse-rey}.
Fix $x_3\in\R^{d_3}$ and set
\[
	\mathcal G_{x_3}(\theta):=\{I\in\mathcal G(\theta):x_3\in I^3\},
	\qquad
	\mathcal B_{x_3}(\theta)
	:=\{I^1\times I^2:I\in\mathcal G_{x_3}(\theta)\}.
\]
The projection $I\mapsto I^1\times I^2$ is injective on
$\mathcal G_{x_3}(\theta)$. Indeed, the Zygmund relation determines $\ell(I^3)$ from
$I^1\times I^2$, and $x_3$ then determines the dyadic cube $I^3$. If
$I,J\in\mathcal G_{x_3}(\theta)$ and
$I^1\times I^2\subset J^1\times J^2$, the same relation gives
$\ell(I^3)\leq\ell(J^3)$. Since both third-coordinate cubes contain $x_3$, we have
$I^3\subset J^3$, and hence $I\subset J$. The incomparability of $\mathcal U$ forces
$I=J$. Thus $\mathcal B_{x_3}(\theta)$ is pairwise incomparable. Lemma
\ref{lem:product-sparse-rey} now gives
\begin{equation}\label{eq:diagonal-envelope-selected-exponential}
	\int_{\R^{d_1+d_2+d_3}}
	\left[e^{t_1h_{\mathcal G(\theta)}(x)}-1\right]\,\ud x
	\leq\frac1{16}|\operatorname{sh}(\mathcal G(\theta))|
	\leq\frac1{16}|\Omega|.
\end{equation}

The usual convex transfer in Lemma
\ref{lem:inclusion-probabilities-to-convex-overlap} now completes the proof. Apply it on
this probability space with $w(I)=W_w(I;\Omega)$,
$\kappa=\|w\|_{\mathrm{env}}^{-1}$ and $\Phi(u)=e^{t_1u}-1$, using
\eqref{eq:diagonal-envelope-marginals} and
\eqref{eq:diagonal-envelope-selected-exponential}. It selects
$\theta_0\in\Theta_v$ such that the $1/4$-sparse family
$\mathcal G_0:=\mathcal G(\theta_0)$ satisfies
\eqref{eq:diagonal-envelope-sparse-exponential}.
\end{proof}

\section{One-depth Zygmund packing without global incomparability}\label{sec:geometry}

In the bi-parameter setting it is already known that a one-depth theorem does not require
global incomparability: one may instead assume, for example, maximality in only one
coordinate direction; compare \cites{Pipher1986, CLMP2006} and
\cite{Hytonen2014}*{Section 8.1}. The following formulation isolates the fixed-coordinate
disjointness hypothesis implicit in standard one-depth proofs. One-sided maximality and
incomparability are generally not comparable, but both imply this strictly weaker
condition.

This section treats only the corresponding Zygmund question. For pairwise incomparable
Zygmund families, the admissible-weight theorems already contain the one-depth estimates:
one simply takes $w(a,b)=\omega(a)$, for which
$\|w\|_{\mathrm{ft}}=\|\omega\|_{\ell^1(\N)}$. The second lemma below shows that either
incomparability or one-sided Zygmund maximality implies fixed-coordinate disjointness on
the Zygmund boundary. Together with the usual embeddedness tail estimate, this places the
result directly within the exclusion-packing framework of Lemma
\ref{lem:exclusion-packing}.
The corresponding two-depth question---whether the fixed-total characterization persists
under a natural hypothesis weaker than global incomparability---is left open.

\subsection{Boundary maximality and fixed-coordinate disjointness}

For $I\in\calD_Z$ and $k\in\N$, define the two diagonal ancestors
\[
	I_{13}^{(k)}:=(I^1)^{(k)}\times I^2\times(I^3)^{(k)},
	\qquad
	I_{23}^{(k)}:=I^1\times(I^2)^{(k)}\times(I^3)^{(k)}.
\]
Both remain on the Zygmund boundary. For a measurable set $\Omega$, define
\begin{align*}
	\mathscr M_{13}(\Omega)
	 & :=\{I\in\calD_Z:I\subset\Omega,\ I_{13}^{(1)}\not\subset\Omega\}, \\
	\mathscr M_{23}(\Omega)
		 & :=\{I\in\calD_Z:I\subset\Omega,\ I_{23}^{(1)}\not\subset\Omega\}.
\end{align*}
Unlike embeddedness, these two maximality conditions are measured in $\Omega$ itself, not in
its halo.
Let $\mathscr M_Z(\Omega)$ denote the family of rectangles that are maximal under inclusion
among the Zygmund rectangles contained in $\Omega$.

\begin{lem}[Boundary maximality and incomparability]
	\label{lem:two-maximalities-incomparable}
	One has
	\[
		\mathscr M_Z(\Omega)
		=\mathscr M_{13}(\Omega)\cap\mathscr M_{23}(\Omega).
	\]
	In particular,
	$\mathscr M_{13}(\Omega)\cap\mathscr M_{23}(\Omega)$ is pairwise incomparable.
\end{lem}

\begin{proof}
	If $I\in\mathscr M_Z(\Omega)$, then neither diagonal parent of $I$ is
	contained in $\Omega$, since both are strictly larger Zygmund rectangles. Hence
	$I\in\mathscr M_{13}(\Omega)\cap\mathscr M_{23}(\Omega)$.

	Conversely, let
	$I\in\mathscr M_{13}(\Omega)\cap\mathscr M_{23}(\Omega)$, and let $J\in\calD_Z$
	strictly contain $I$. Write
	$J^1=(I^1)^{(a)}$ and $J^2=(I^2)^{(b)}$. The boundary relation gives
	$J^3=(I^3)^{(a+b)}$, and hence $J=(I_{13}^{(a)})_{23}^{(b)}$, where
	$a,b\geq0$ are not both zero. If $a\geq1$, then
	$I_{13}^{(1)}\subset J$; if $b\geq1$, then $I_{23}^{(1)}\subset J$. At least one of
	these alternatives holds, and the corresponding parent is not contained in $\Omega$.
	Therefore $J\not\subset\Omega$, so $I\in\mathscr M_Z(\Omega)$. This proves the identity.
	If two distinct members of $\mathscr M_Z(\Omega)$ were comparable, the smaller one would
	not be maximal. Thus this family is pairwise incomparable.
\end{proof}

\begin{lem}[Fixed-coordinate order and disjointness on the Zygmund boundary]
	\label{lem:fixed-coordinate-boundary}
	If $I,J\in\calD_Z$ meet and $I^1=J^1$, then, after interchanging $I$ and $J$ if
	necessary, for some $m\geq0$,
	\begin{equation}\label{eq:coupled-order}
		J^2=(I^2)^{(m)},
		\qquad J^3=(I^3)^{(m)},
		\qquad J=I_{23}^{(m)}.
	\end{equation}
	Consequently, if $\mathcal U\subset\mathscr M_{23}(\Omega)$ or if
	$\mathcal U\subset\calD_Z$ is pairwise incomparable, then distinct rectangles in
	$\mathcal U$ with the same first-coordinate cube are disjoint.
\end{lem}

\begin{proof}
	For the first assertion, after interchanging $I$ and $J$ assume that
	$\ell(I^2)\leq\ell(J^2)$. Dyadic nesting gives
	$J^2=(I^2)^{(m)}$ for some $m\geq0$. Since both rectangles satisfy the boundary
	identity and have the same first cube, $\ell(J^3)=2^m\ell(I^3)$. Their third-coordinate
	cubes meet, so $J^3=(I^3)^{(m)}$. This proves \eqref{eq:coupled-order}.

	Let $I,J\in\mathcal U$ meet and have the same first cube. By
	\eqref{eq:coupled-order}, after interchanging them if necessary,
	$J=I_{23}^{(m)}$. If $m\geq1$ and
	$\mathcal U\subset\mathscr M_{23}(\Omega)$, then
	$I_{23}^{(1)}\subset J\subset\Omega$, contrary to
	$I\in\mathscr M_{23}(\Omega)$. If instead $\mathcal U$ is pairwise incomparable,
	$m\geq1$ would give $I\subsetneq J$. Hence in either case $m=0$ and $I=J$.
	This proves the fixed-coordinate disjointness assertion.
\end{proof}

\subsection{One-depth packing with fixed-coordinate disjointness}\label{sec:one-depth}

The following theorem is stated directly under this common weaker condition. It applies to
the full sub-Zygmund basis: no boundary relation, incomparability or maximality is assumed.
That basis enters only because it contains every flat first-coordinate ancestor used to
define the depth.

\begin{thm}[One-depth packing in the sub-Zygmund basis]
	\label{thm:one-depth-scalar}
	Let $\Omega\subset\R^{d_1+d_2+d_3}$ be measurable and have finite measure, and let
	\[
		\mathcal U\subset\{I\in\calD_{\mathrm{sZ}}:I\subset\Omega\}.
	\]
	Suppose that distinct rectangles with the same first-coordinate cube are disjoint:
	\begin{equation}\label{eq:fixed-coordinate-disjointness}
		I,J\in\mathcal U,\quad I^1=J^1,\quad I\neq J
		\quad\Longrightarrow\quad I\cap J=\varnothing.
	\end{equation}
	If $\omega\colon\N\to[0,\infty)$ is decreasing, summable and nonzero, there is a $1/2$-sparse
	subfamily $\mathcal G\subset\mathcal U$ such that
	\begin{equation}\label{eq:one-depth-direct}
		\sum_{I\in\mathcal U}\omega(e_1(I;\Omega))|I|
		\leq2\|\omega\|_{\ell^1(\N)}
		|\operatorname{sh}(\mathcal G)|.
	\end{equation}
\end{thm}

\begin{proof}
	Apply Lemma \ref{lem:exclusion-packing} with $X:=\Omega$, $P_I:=I$,
	$a(I):=-\nu_1(I)$ and $b(I):=e_1(I;\Omega)$, using the exclusion family from
	Lemma \ref{lem:stationary-one-depth-events}. If $a(I)=a(J)$, then $I^1$ and $J^1$
	have the same generation. Their first cubes are therefore either disjoint or equal; in
	the latter case \eqref{eq:fixed-coordinate-disjointness} applies. This proves condition~(i).

	For condition~(ii), fix $I\in\mathcal U$, put $d:=e_1(I;\Omega)$ and define
	\[
		B_I:=(I^2\times I^3)\cap
		\bigcup_{\substack{K\in\mathcal U\\
		\nu_1(K)<\nu_1(I)-d\\K\cap I\neq\varnothing}}
		(K^2\times K^3).
	\]
	For every $K$ occurring in this union, integrality of the generations gives
	$\nu_1(K)\leq\nu_1(I)-d-1$. Since $K^1\cap I^1\neq\varnothing$, dyadic nesting
	therefore gives $(I^1)^{(d+1)}\subset K^1$. As every such $K$ is contained in
	$\Omega$, we obtain
	\[
		E_I:=I\cap
		\bigcup_{\substack{K\in\mathcal U\\
		\nu_1(K)<\nu_1(I)-d}}K
		=I^1\times B_I,
		\qquad
		(I^1)^{(d+1)}\times B_I\subset\Omega.
	\]
	If $|E_I|>|I|/2$, then $|B_I|>|I^2\times I^3|/2$. Apply Lemma
	\ref{lem:half-slab-halo} with $\mathcal B=\calD_{\mathrm{sZ}}$, $G=\Omega$,
	$Q=I_1^{(d+1)}$, $E^2=B_I$, and $L^1=I^1$. The two slab inclusions follow from
	the preceding display and $I\subset\Omega$, so the lemma gives
	$I_1^{(d+1)}\subset\widetilde\Omega$, contradicting the definition of $d$. Therefore
	$|E_I|\leq|I|/2$, which is condition~(ii). Lemmas \ref{lem:exclusion-packing} and
	\ref{lem:stationary-one-depth-events} now give a $1/2$-sparse family $\mathcal G$ with
	\[
		\frac1{\|\omega\|_{\ell^1(\N)}}
		\sum_{I\in\mathcal U}\omega(e_1(I;\Omega))|I|
		\leq\sum_{I\in\mathcal G}|I|
		\leq2|\operatorname{sh}(\mathcal G)|,
	\]
	which proves \eqref{eq:one-depth-direct}.
\end{proof}

Theorem \ref{thm:one-depth-scalar} and Lemma
\ref{lem:fixed-coordinate-boundary} give the following boundary result.

\begin{cor}[One-depth packing on the Zygmund boundary]
	\label{cor:one-depth-incomparable}
	Let $\Omega\subset\R^{d_1+d_2+d_3}$ be measurable and have finite measure. Let
	$\mathcal U\subset\calD_Z$ be contained in $\Omega$, and let
	$\omega\colon\N\to[0,\infty)$ be decreasing, summable and nonzero. If $\mathcal U$ is pairwise
	incomparable, then the conclusion below holds for each $i=1,2$. The conclusion for
	$i=1$ also holds if instead $\mathcal U\subset\mathscr M_{23}(\Omega)$, and the
	conclusion for $i=2$ holds if instead $\mathcal U\subset\mathscr M_{13}(\Omega)$.
	Namely, there is a $1/2$-sparse subfamily $\mathcal G_i\subset\mathcal U$ such that
	\begin{equation}\label{eq:one-depth-incomparable}
		\sum_{I\in\mathcal U}\omega(e_i(I;\Omega))|I|
		\leq2\|\omega\|_{\ell^1(\N)}
		|\operatorname{sh}(\mathcal G_i)|.
	\end{equation}
\end{cor}

\begin{proof}
	Lemma \ref{lem:fixed-coordinate-boundary} gives the fixed-coordinate disjointness
	required by Theorem \ref{thm:one-depth-scalar} for $i=1$, under either incomparability
	or $23$-maximality. Applying the theorem yields \eqref{eq:one-depth-incomparable}
	with $i=1$. Interchanging the first two coordinates gives the $i=2$ conclusion under
	either incomparability or $13$-maximality.
\end{proof}

\section{Geometric boundaries of the theory}\label{sec:geometric-boundaries}

The fixed-total characterization already settles the sharp range of the depth weights and
the optimal order of their packing constants. We record several complementary geometric
observations: allowing initial rectangles beyond the Zygmund boundary forces full
summability, the sub-Zygmund and tri-parameter halos can produce arbitrarily different
depths, diagonal depths do not detect the relevant layers, and direct embeddedness cannot
replace embeddedness in a maximal-function halo.

\subsection{Beyond the Zygmund boundary: full summability}
\label{sec:subzygmund-obstruction}

The layered Zygmund family in Proposition \ref{prop:zygmund-layered-test-family} samples
one fixed-total line $a+b=N$. If the initial family may instead range through
$\calD_{\mathrm{sZ}}$, the additional layers sample every lattice point in the triangle
$a+b\leq N$.

\begin{prop}[Full summability beyond the Zygmund boundary]
\label{prop:subzygmund-family-failure}
In the standard dyadic grid on $\R^3$, let $\Omega=[0,1)^3$. For every integer
$N\geq0$, there is a finite pairwise incomparable family
$\mathcal S_N\subset\calD_{\mathrm{sZ}}$ of rectangles contained in $\Omega$ such that,
for every $w\colon\N^2\to[0,\infty)$,
\begin{equation}\label{eq:subzygmund-triangular-identity}
	\frac1{|\Omega|}\sum_{I\in\mathcal S_N}W_w(I;\Omega)|I|
	=\sum_{\substack{a,b\geq0\\a+b\leq N}}w(a,b).
\end{equation}
Consequently, any uniform two-depth packing estimate for sub-Zygmund initial families
forces $w\in\ell^1(\N^2)$.
\end{prop}

\begin{proof}
Fix $N\geq0$ and let
\[
	\mathcal S_N:=\left\{I^1\times I^2\times I^3\subset\Omega:
	\ell(I^j)=2^{-n_j},\quad n_1+n_2+n_3=2N,\quad n_3\geq n_1+n_2\right\},
\]
where $n_1,n_2,n_3\in\N$. Every member has volume $2^{-2N}$, so
$\mathcal S_N\subset\calD_{\mathrm{sZ}}$ is pairwise incomparable. For each fixed scale
triple, the corresponding rectangles partition $\Omega$.

For $i=1,2$, we have $I_i^{(n_i)}\subset\Omega$. Choose
$x\in I_i^{(n_i+1)}$ with $x_i\in(1,2)$ and
$x_j\in\operatorname{int}(I^j)$ for $j\neq i$. If
$Q\in\calD_{\mathrm{sZ}}$ contains $x$ and meets $\Omega$, dyadic nesting gives
$Q^i\supset[0,2)$, and therefore
\[
	\ave{1_\Omega}_Q
	\leq\frac{|Q^i\cap[0,1)|}{|Q^i|}
	\leq\frac12.
\]
Thus $x\notin\widetilde\Omega$ and
$e_i(I;\Omega)=n_i$ for $i=1,2$.

Put $m=n_1+n_2$. The conditions defining $\mathcal S_N$ are equivalent to
$0\leq m\leq N$, $0\leq n_1\leq m$, $n_2=m-n_1$, and $n_3=2N-m$. The partition at each
scale triple now gives
\[
	\frac1{|\Omega|}\sum_{I\in\mathcal S_N}W_w(I;\Omega)|I|
	=\sum_{m=0}^{N}\sum_{n_1=0}^{m}w(n_1,m-n_1)
	=\sum_{\substack{a,b\geq0\\a+b\leq N}}w(a,b),
\]
which is \eqref{eq:subzygmund-triangular-identity}.
\end{proof}

The same identity holds when the depths are measured in the unrestricted tri-parameter
halo, since the density estimate in the proof uses only the enlarged coordinate. In
particular, if $\omega\colon\N\to[0,\infty)$ is decreasing and nonzero, then taking
$w(a,b)=\omega(a)$ in \eqref{eq:subzygmund-triangular-identity} gives the lower bound
$(N+1)\omega(0)$. Thus no nontrivial one-depth estimate holds beyond the Zygmund boundary.

On the Zygmund boundary, by contrast, $n_3=n_1+n_2$ forces $n_1+n_2=N$, leaving
precisely the fixed-total line used in Proposition
\ref{prop:zygmund-layered-test-family}.

\subsection{Sub-Zygmund versus tri-parameter halo depths}

The gap between depths measured in the sub-Zygmund and tri-parameter halos can be
arbitrarily large. Denote these depths by $e_i^{\mathrm{sZ}}$ and
$e_i^{\mathrm{tri}}$, respectively. In the standard dyadic grid, for $N\geq1$, set
\[
	F_N=[0,2^{-N})\times[0,1/2)\times[0,2^{-(N+1)}),
	\qquad
	G=[0,1)\times[0,1/2)\times[1/2,1),
\]
and $\Omega_N:=F_N\cup G$. Then $F_N\in\calD_Z$, and a direct density calculation gives
\[
	e_1^{\mathrm{sZ}}(F_N;\Omega_N)=0,
	\qquad
	e_1^{\mathrm{tri}}(F_N;\Omega_N)=N.
\]

\subsection{Flat versus diagonal depths}

For $I\in\calD_Z$, define the diagonal depths
\[
	e_{13}(I;\Omega):=\max\{k\geq0:I_{13}^{(k)}\subset\widetilde\Omega\},
	\qquad
	e_{23}(I;\Omega):=\max\{k\geq0:I_{23}^{(k)}\subset\widetilde\Omega\}.
\]
For the lifted family $\mathcal U_N^Z$ of Proposition
\ref{prop:zygmund-layered-test-family}, both diagonal depths vanish:
\[
	e_{13}(J;\Omega_N^Z)=e_{23}(J;\Omega_N^Z)=0,
	\qquad J\in\mathcal U_N^Z.
\]
Indeed, choose $x\in J_{13}^{(1)}$ with its first two coordinates in the interiors of
$J^1$ and $J^2$ and its third coordinate in the other child of $(K^3)^{(1)}$. Every
$Q\in\calD_{\mathrm{sZ}}$ through $x$ that meets
$\Omega_N^Z=\Omega\times K^3$ contains $(K^3)^{(1)}$ in its third coordinate. Hence
\[
	\frac{|Q\cap\Omega_N^Z|}{|Q|}
	\leq\frac{|Q^3\cap K^3|}{|Q^3|}
	\leq\frac12.
\]
Thus $x\notin\widetilde{\Omega_N^Z}$ and $e_{13}(J;\Omega_N^Z)=0$; the other direction
is symmetric. Yet every one of the $N+1$ layers partitions $\Omega_N^Z$, and therefore
\[
	\sum_{J\in\mathcal U_N^Z}|J|
	=(N+1)|\Omega_N^Z|.
\]
Thus diagonal depths provide no decay for these $N+1$ overlapping layers, whereas the flat
depths $n$ and $N-n$ in \eqref{eq:zygmund-layer-depths} distinguish them.

\subsection{Direct embeddedness versus halo embeddedness}

We finish with a brief observation, independent of the Zygmund geometry, explaining why
halo embeddedness cannot in general be replaced by direct containment. For a
bi-parameter dyadic rectangle $I$ contained in a finite-measure set $\Omega$, define
\[
	e_i^{\mathrm{dir}}(I;\Omega)
	:=\max\{k\geq0:|I_i^{(k)}\setminus\Omega|=0\},
	\qquad i=1,2.
\]

\begin{exmp}[Direct embeddedness does not suffice]
	For every integer $N\geq1$, one can find a finite family
	$\mathcal U_N^{\mathrm{dir}}$ of bi-parameter dyadic rectangles and a set
	$\Omega_N^{\mathrm{dir}}=\operatorname{sh}(\mathcal U_N^{\mathrm{dir}})\subset[0,1)^2$
	such that every member of $\mathcal U_N^{\mathrm{dir}}$ is maximal among the dyadic
	rectangles contained in $\Omega_N^{\mathrm{dir}}$ and
	\[
		e_1^{\mathrm{dir}}(I;\Omega_N^{\mathrm{dir}})
		=e_2^{\mathrm{dir}}(I;\Omega_N^{\mathrm{dir}})=0,
		\qquad I\in\mathcal U_N^{\mathrm{dir}},
	\]
	whereas
	\[
		\sum_{I\in\mathcal U_N^{\mathrm{dir}}}|I|=\frac N2,
		\qquad |\Omega_N^{\mathrm{dir}}|\leq1.
	\]
	Consequently, direct-depth packing fails for every weight $w$ with $w(0,0)>0$, even
	under maximality and hence pairwise incomparability.
\end{exmp}

\begin{proof}
	Put $Q=[0,1)^2$. For $0\leq j<2^N$, let $\operatorname{rev}_N(j)$ be the integer
	obtained by reversing the $N$ binary digits of $j$, and set
	\[
		H_N:=\left\{\left(\frac{j+1/3}{2^N},
		\frac{\operatorname{rev}_N(j)+1/3}{2^N}\right):0\leq j<2^N\right\}.
	\]
	This is the classical van der Corput dyadic net translated by
	$(2^{-N}/3,2^{-N}/3)$; see \cite{BLPV2009}*{Section 3}. It has the elementary
	one-point property
	\begin{equation}\label{eq:one-point-property}
		\#(P\cap H_N)=1
		\quad\text{for every dyadic rectangle }P\subset Q\text{ with }|P|=2^{-N}.
	\end{equation}
	Indeed, if the side lengths of $P$ are $2^{-a}$ and $2^{-(N-a)}$, its first side fixes
	the first $a$ binary digits of $j$, while its second side fixes the remaining digits
	through the reversal. The shift ensures that no point of $H_N$ lies on a dyadic boundary.

	For $1\leq a\leq N$, put $b=N+1-a$ and let
	\[
		\mathcal U_{N,a}^{\mathrm{dir}}
		:=\{I=I^1\times I^2\subset Q:
		|I^1|=2^{-a},\ |I^2|=2^{-b},\ I\cap H_N=\varnothing\}.
	\]
	The dyadic rectangles with these side lengths partition $Q$ into $2^{N+1}$ rectangles.
	Each contains at most one point of $H_N$, since each coordinate parent has area $2^{-N}$
	and hence contains exactly one point by \eqref{eq:one-point-property}. Since the $2^N$
	points of $H_N$ avoid the boundaries of this partition, precisely $2^N$ rectangles
	contain one such point and the remaining $2^N$ avoid $H_N$. Therefore
	\[
		\sum_{I\in\mathcal U_{N,a}^{\mathrm{dir}}}|I|=\frac12.
	\]
	Set
	$\mathcal U_N^{\mathrm{dir}}:=\bigcup_{a=1}^N\mathcal U_{N,a}^{\mathrm{dir}}$ and
	$\Omega_N^{\mathrm{dir}}:=\operatorname{sh}(\mathcal U_N^{\mathrm{dir}})$. Summing the
	last identity gives the asserted mass, and $\Omega_N^{\mathrm{dir}}\subset Q$ gives its
	measure bound.

	Fix $I\in\mathcal U_{N,a}^{\mathrm{dir}}$. Both coordinate parents remain in $Q$ and
	have area $2^{-N}$, so each contains a point of $H_N$. Since the selected family is
	finite, its members avoid $H_N$, and no point of $H_N$ lies on a dyadic boundary, a small
	positive-measure neighborhood around each point inside the corresponding parent misses
	$\Omega_N^{\mathrm{dir}}$. Thus both direct depths of $I$ are zero. Moreover, every
	dyadic rectangle strictly containing $I$ contains at least one of its coordinate parents
	and therefore cannot be contained in $\Omega_N^{\mathrm{dir}}$. Hence $I$ is maximal.
\end{proof}

This is precisely the obstruction avoided by measuring embeddedness in the halo.

\bibliography{references}

\end{document}